\documentclass[11pt]{article}

\usepackage{amssymb,graphics,amsmath,amsopn,amstext,amsfonts,color,fullpage,fancybox,hyperref,bm}
\usepackage{amsthm}
\usepackage{subfigure,placeins} % make it possible to include more than one captioned figure/table in a single float
\usepackage{graphicx}
\usepackage{verbatim}
\usepackage{algorithm}
\usepackage{algorithmic}
\usepackage{longtable}
\usepackage{multirow}

\newcommand{\gap}{\vspace{0.1in}}
\newcommand{\epc}{\hspace{1pc}}

\newcommand{\wh}{\widehat}
\newcommand{\wt}{\widetilde}

\newcommand{\norm}[1]{\left\| \, #1 \, \right\|}

\newtheorem{theorem}{Theorem}
\newtheorem{proposition}[theorem]{Proposition}
\newtheorem{lemma}[theorem]{Lemma}
\newtheorem{corollary}[theorem]{Corollary}
\newtheorem{definition}[theorem]{Definition}

\usepackage{tabularx}

   	\title{
   	Solving Nonsmooth Nonconvex Compound Stochastic Programs
   	with Applications to Risk Measure Minimization}
   
   \author{Junyi Liu \footnote{Department of Industrial Engineering, Tsinghua University,
   		Beijing, China. Email: junyiliu@mail.tsinghua.edu.cn.}  \and Ying Cui\footnote{Department of Industrial and Systems Engineering,  University of Minnesota, USA. Email: yingcui@umn.edu} \and Jong-Shi Pang \footnote{Daniel J.\ Epstein Department of Industrial and Systems Engineering, University of Southern California, USA. Email: jongship@usc.edu}}
    
\begin{document}

	\maketitle
	
	\begin{abstract}
		This paper studies a structured compound stochastic program (SP) involving multiple
		expectations coupled by nonconvex and nonsmooth functions.
		We present a successive convex-programming based sampling algorithm and establish its subsequential convergence.
		We describe stationarity properties of the limit points for several classes of the compound SP.  We further
		discuss probabilistic stopping rules based on the computable error-bound for the algorithm.
		% We present a successive convex-programming based sampling algorithm for solving this problem, establish its
		% subsequential convergence, and discuss probabilistic  stopping rules based on the computable
		% error-bound for the algorithm.
		We present several risk measure minimization problems that can be formulated as such a compound stochastic program;
		these include generalized deviation optimization problems based on the optimized certainty
		equivalent and buffered probability of exceedance (bPOE), a
		distributionally robust bPOE optimization problem, and a
		multiclass classification problem employing the cost sensitive error criteria with bPOE risk measure.
	\end{abstract}
\textbf{Key words}: stochastic programming, nonconvex optimization, risk measure optimization

\section{Introduction}
This paper studies the so-called \emph{compound stochastic program} (compound SP) that was originally introduced in  \cite{ermoliev2013sample}.	
Let $\tilde{\xi}: \Omega \to \Xi \subseteq \mathbb{R}^m$ be an $m$-dimensional random variable defined on a probability space
$(\Omega, \mathcal{A}, \mathbb{P})$, where $\Omega$ is the sample
space, $\Xi$ is a measurable subset of $\mathbb{R}^m$, $\mathcal{A}$ is a
$\sigma$-field on $\Omega$, and $\mathbb{P}$ is a probability measure
defined on the pair $(\Omega,\mathcal{A})$. We use $\xi$ without tilde to denote the realization of the random variable $\tilde{\xi}$.
The compound SP considered in this
paper is the following class of optimization problems that involves the nested
expectations of random functions in the objective:
\begin{equation} \label{eq:compound_composite_SP}
\displaystyle{
\operatornamewithlimits{\mbox{minimize}}_{x\in X}
} \;\; \Theta(x) \, \triangleq \, \psi\Big( \, \mathbb{E}\Big[ \,
\varphi\Big( G(x, \tilde{\xi} ), \, \mathbb{E} \, [ \, F(x, \tilde{\xi} ) \, ] \Big) \, \Big] \, \Big) ,
\end{equation}
where $X\subseteq \mathbb{R}^n$ is a closed convex set contained in an open set ${\cal O}$, and
% with  corresponding CDF $p(\cdot)$.
%Involving two expectations embedded in a vector bivariate function $\varphi$, the
%convex constrained compound stochastic programming problem is formulated as where,
for some nonnegative integers $\ell_G$ and $\ell_F$ with $\ell \triangleq \ell_G + \ell_F > 0$, and a positive integer
$\ell_{\varphi}$, $( G,F ) : {\cal O} \times \Xi \, \subseteq \, \mathbb{R}^{n+m} \to \mathbb{R}^{\ell}$,
$\varphi : \mathbb{R}^{\ell} \to \mathbb{R}^{\ell_{\varphi}}$ and $\psi : \mathbb{R}^{\ell_{\varphi}} \to \mathbb{R}$.
We assume throughout that $F$ and $G$ are Carath{\' e}odory functions;
i.e., $F(\bullet,\xi)$ and $G(\bullet,\xi)$
are continuous for all $\xi \in \Xi$ and $F(x,\bullet)$ and $G(x,\bullet)$ are measurable (in fact, with
$\mathbb{E}\left[ \| \, G(x,\tilde{\xi}) \, \| \right]$ and $\mathbb{E}\left[ \| \, F(x,\tilde{\xi}) \, \| \right]$
both finite in our context) for all $x \in X$, and $\varphi$ and $\psi$ are continuous with more assumptions on
them to be stated later.  The notation $\mathbb{E}$ denotes the expectation operator associated with the random
variable $\tilde{\xi}$.

\gap

Clearly, for the compound SP \eqref{eq:compound_composite_SP}, without the outer function $\psi$, we must have $\ell_{\varphi} = 1$;
moreover in this case, the objective function $\Theta$ can be written alternatively as
$\mathbb{E}  [ \, h (x,\mathbb{E}[ \, F(x, \tilde{\xi} ) \, ],\tilde{\xi} \, ) \,  ]$ for
some scalar-valued function $h : X \times \mathbb{R}^{\ell_F} \times \Xi \to \mathbb{R}$; this is the form of the
  % and probably nonsmooth and nonconvex
 two-layer expectation functions
studied in \cite{ermoliev2013sample}.
Besides the feature of the two-layer expectation, the objective function $\Theta$ in
\eqref{eq:compound_composite_SP} involves an outer (vector)
deterministic function $\psi(\bullet)$ and an additional inner function $G$.
For	mathematical analysis, the compact form
$\mathbb{E}  [ \, h (x,\mathbb{E}[ \, F(x, \tilde{\xi} ) \, ],\tilde{\xi} \, ) \,  ]$ is
convenient and broad; for the algorithmic development and the applications to be discussed later, the form
(\ref{eq:compound_composite_SP}) lends itself to provide a more convenient framework.
When $\ell_G = 0$ (or $\ell_F = 0$), the objective function $\Theta$ essentially becomes
${\varphi} ( \,\mathbb{E} [{F}(x,\tilde{\xi} \, )  ] \,  )$.  This  case extends beyond a single expectation
functional which is the building block of a standard SP. It is worth noting that although both multi-stage SP and compound SP involve  multiple layers of expectations,
the composition of layers of expectations in compound SP  is built upon functions ($\varphi$, $\psi$, $F$ and $G$), % without the inner optimization involved,
which distinguishes itself from multi-stage SP where  layers of expectations are composited by the random value functions
of parametric optimization problems.
	%where
	%$\wt{\varphi} = \psi \circ \varphi$.  The other case is when $\ell_F = 0$ so that
	%$\Theta$ becomes $\psi ( \, \mathbb{E} [ \,
	%\wh{G}(x, \tilde{\xi} ) \,  ] \,  )$, where $\wh{G} = \varphi \circ G$.

% In particular, when $\ell_F = 0$ and $\psi$ is an identity function, the objective function $\Theta(x)$ is a standard expectation function,
% and thus (\ref{eq:compound_composite_SP}) is a standard SP which also covers two-stage SP problems.  However, it is worth noting that the general
% three-stage or multi-stage SP is not a class of compound SP (1), since the compositional structure of multi-stage SP involves  not only expectations
% but also value functions in multiple stages.

\gap
	
The compound SP in \eqref{eq:compound_composite_SP} arises frequently in risk
measure minimization problems.  First proposed
in \cite{rockafellar2006generalized}, a generalized deviation measure quantifies
the deviation of a random variable from its expectation by a risk measure.  In this
paper, we focus on two such measures, one is based on the optimized certainty
equivalent (OCE) \cite{ben1986expected,ben2007old} of a random variable that predates and
extends the  conditional value-at-risk (CVaR)
\cite{rockafellar2000optimizationCVAR,rockafellar2002conditional} used extensively in
financial engineering and risk management.  The other type of application we are interested in is based on the reliability measure of a random variable, which is called the buffered probability of exceedance (bPOE).
First employed by Rockafellar and Royset \cite{RockafellarRoysett10} as a structural reliability measure in
the design and optimization of structures, the study of bPOE  has been further explored by Uryasev and his collaborators in a series of
publications \cite{mafusalov2018estimation,mafusalov2018buffered,norton2017error,norton2019maximization}.
In addition, in the area of distributionally robust optimization (DRO), the robustification of  the (conditional) value-at-risk
has been studied and applied extensively to portfolio management and other areas;
see e.g.\ \cite{ChanMahmoudzadehPurdie14,ElGhouiOksOustry03,LotfiZenios18,pflug2012}
for the applications of (C)VaR-based DRO, and \cite{rahimian2019distributionally} for a thorough literature review of DRO.
Motivated by the growing importance of the bPOE and the significant interests and
advances in distributionally robust optimization, we are interested in the study of  a distributionally
robust mixed bPOE optimization problem. It turns out that this problem can be formulated
as a nonconvex compound stochastic program \eqref{eq:compound_composite_SP}
due to the mixture of distributions and the mixture of bPOEs at multiple thresholds.
	
\gap
	
The departure of our study from the sample average approximation (SAA) analysis in \cite{ermoliev2013sample,hu2020sample} is that
we aim to design a practically implementable algorithm for approximating a stationary
solution of \eqref{eq:compound_composite_SP}.  Due to the lack of convexity   in the applications to be presented later, seeking a stationary solution of
a sharp kind is the best we can hope for in practical computations.  This task is
complicated by the coupled nonconvexity and nonsmoothness  that renders
gradient/subgradient based algorithms difficult to be applied.  This feature highlights the
contributions of the present paper and distinguishes it from a series of recent works on the SAA convergence analysis
of the globally optimal solutions and objective values \cite{ermoliev2013sample, dentcheva2017statistical, hu2020sample} and the  development of stochastic gradient
type algorithms \cite{WangFangLiu17,GhadimiRuszczynskiWang18,YangWangFang19} under the differentiability assumption
on the nested SPs,  which can be seen as an extension of \eqref{eq:compound_composite_SP} to multiple layers of expectations.
% The asymptotic convergence of the SAA estimations are established using the Delta method in \cite{dentcheva2017statistical},
% and   stochastic gradient type algorithms are analyzed, all of which are based on the assumption that the composite random functions in all layers are differentiable.
	
\gap
	
In a nutshell, the major contributions of our work are summarized as follows:
	
\gap
	
1.  We introduce a practically implementable convex-programming based
stochastic majorization minimization (SMM)
algorithm with incremental sampling for solving the nonconvex
nondifferentiable compound SP.  With the
sample size increasing roughly linearly, we establish the subsequential convergence
of the algorithm to a fixed point of the algorithmic map in an almost sure manner.
We relate such a point to several kinds of stationary solutions of the original
problem under different assumptions on the component functions $G$ and $F$ and a
blanket assumption on $\psi$ and $\varphi$
(see Subsection~\ref{subsec:assumptions}).
	
\gap
	
2. We provide a computable probabilistic stopping rule for the proposed SMM algorithm with the aid of
a novel surrogate-function based error-bound theory of   locally Lipchitz
programs.  Such a theory aims to bound the distance of a given iterate to the set of
stationary points in terms of a computable residual.  To the best of
our knowledge, this is the
first attempt to develop such an error bound theory for majorization-minimization
algorithms for nonsmooth nonconvex optimization problems, and in both
deterministic and stochastic settings.
	
\gap
	
3.  We present several applications of the compound SP for risk management based on some
topical risk and reliability measures; these include OCE-based deviation minimization, bPOE-based
deviation minimization, and distributionally robust mixed bPOE optimization.
We also apply the bPOE as the measure to quantify the classification errors
in cost-sensitive multi-class classification problems to illustrate the promise of the compound SP
methodology in statistical learning.

\subsection{Discussion and Related Literature}

At the suggestion of a referee, we feel that it would be useful to place the proposed algorithm in the broader context
of possible approaches for solving the problem (\ref{eq:compound_composite_SP}).
 % Thus, in general the method for compound SP in the present paper is not suitable for  multi-stage SP problems, except for some special ones.  When there is no minimization operator starting from the second stage, meaning that there is no sequential decision-making process, this case of the multi-stage SP is actually a compound SP in which layers of expectations are combined by the sum operator.  Moreover, for a two-stage SP with a single expectation objective function, the SMM method might be applicable if the majorizing surrogate function of the nonconvex value function is available. An instance of such an application is the two-stage stochastic quadratic program studied in \cite{liu2020two} when the second-stage quadratic program is linearly bi-parameterized by the first-stage decisions with a positive definite quadratic term in the objective. }
For the problem (\ref{eq:compound_composite_SP}), one may use
the well-known sample-average approximation (SAA) scheme in which such approximations of the expectation
is constructed, leaving the nonconvexity and nondifferentiability in the original problem intact.
% and saving them to be handled when the discretized subproblems are solved.
One may next design a numerical algorithm to compute stationary solutions of the nonconvex and nonsmooth SAA subproblems
followed by an asymptotic analysis on
such points to a stationary point of the original problem (if possible) when the sample size increases to infinity.
Below we compare this approach with our proposed method.

\gap

First,  the statistical consistency results of the SAA problems in  \cite{ermoliev2013sample,dentcheva2017statistical,hu2020sample} on compound/conditional/nested stochastic programs
pertain to the convergence  of the globally {optimal} solutions and objective values.  In particular, the non-asymptotic convergence properties of the SAA estimator is analyzed in  \cite{ermoliev2013sample}    using the Rademacher average of the underlying random functions. Another work \cite{hu2020sample} focuses on a more general setting, refered as the conditional SP, which has a two-layer nested expectation objective function  with the inner conditional expectation.  It  analyzes the  sample complexity of the SAA under a variety of structural assumptions, such as Lipschitz
continuity and error bound conditions.    However, with the presence of  nonconvexity and nonsmoothness in
problem \eqref{eq:compound_composite_SP}, one may not be able to  compute such optimal solutions and values. Instead, one may be interested in the statistical consistency of computable stationary points.  When the functions $\psi$ and $\varphi$
in \eqref{eq:compound_composite_SP} are absent, the consistency of the (Clarke) stationary points of the SAA is established in
\cite{shapiro2007uniform} under the Clarke regularity of the objective function.  Asymptotic behavior of the directional stationary solutions of
the SAA scheme arising from structured statistical estimation problems is also analyzed in a recent work \cite{qi2019statistical} without the
Clarke regularity of the underlying objective function.  For the compound nonsmooth and nonconvex SP \eqref{eq:compound_composite_SP} that does
not fall into the aforementioned frameworks, an open question on the consistency of stationary points is saved for future study.

\gap

Second,  when the SAA
scheme is applied to nonconvex problems,  it does not provide computational algorithms for solving subproblems, so one need to design a numerical algorithm
for the determinisitic subproblems.   In contrast, the subproblems in the SMM algorithm
are convex; thus the algorithm is readily implementable in practice by
state-of-the-art convex programming methods. Moreover,  it may not be easy to determine the sample size in the SAA problem before executing the algorithm,
so practically one may need
to solve multiple nonconvex and nonsmooth SAA subproblems with independently generated sample sets of large sample sizes.
Instead of adopting the SAA scheme that only approximates the expectation, we simultaneously approximate the expectation
by incremental sample averaging and  convex surrogation;
by doing this, we may obtain satisfactory descent progress in the early iterations with relatively small sample sizes.

\gap

In fact, sequential sampling  has a long history in stochastic programming. Methods based on this technique include  the stochastic decomposition algorithm
\cite{HigleSen91} pertaining to two-stage stochastic linear programs, and the work of Homem-de-Mello \cite{Homem-de-Mello03} that examines consistency under
sample-augmentation without focusing on computations.  A key finding of the latter paper is that the sample size must grow sufficiently fast,
which is consistent to the proposed SMM algorithm in our more general setting.  A more recent work on sample augmentation is  \cite{RoysetPolak07}
by Royset and Polak  which increases the sample size based on progress of the algorithm and shows the consistency to stationary points.
This is aligned with the larger literature on trust region methods  with random models  \cite{ChenMenickellyScheinber18}; namely, if the progress of
an algorithm is less than the ``noise'' of sampling then one has to reduce the noise by increasing
the sample size.
% To what extend this sampling strategy can be applied to our subsequent algorithm for the problem (\ref{eq:compound_composite_SP}) requires further investigation.
Most recently, the work \cite{an2019stochastic} establishes the convergence of a stochastic difference-of-convex algorithm with the increase
of a single sample at each iteration. As we will see later, the problem (\ref{eq:compound_composite_SP}) does not have this luxury due to the compound nature
of   the objective function. One may also note the affine models used in the recent paper \cite{DuchiRuan18}, which in spirit are similar to the ones in our
work but for a much simpler setting.

\gap

Finally, for the problem (\ref{eq:compound_composite_SP}) that lacks convexity and differentiability, any attempt to directly apply a popular stochastic
approximation method in which a (sub)gradient is computed at each iteration will at best result in a heuristic method.  Since we are interested in proposing
a computational method that has a rigorous supporting theory and direct tie to practical implementation, we choose not to consider a method that requires
the computation of a ``gradient'' of a composite nondifferentiable function. While there is an abundance of advanced subderivatives of nonsmooth functions, as
exemplified by the advances in nonsmooth analysis \cite{rockafellar2009variational}, it might be too early to attempt to apply any such advanced object
to our problem on hand.

\subsection{Notation and organization}

We explain some notations that are used in the present paper. Let
$\mathbb{B}(x;\delta)$ denote the Euclidean ball
centered at a vector $x \in \mathbb{R}^n$ with a positive radius $\delta$; in particular,
we let $\mathbb{B}$ denote the unit ball in $\mathbb{R}^n$.
For a vector-valued function $\phi : \mathbb{R}^k \to \mathbb{R}^{\ell}$, we use
$\phi_i$ for $i=1, \cdots, \ell$ to denote each scalar-valued component function of
$\phi$; thus, $\phi_i : \mathbb{R}^k \to \mathbb{R}$.  We say that a function
$\phi : \mathbb{R}^k \to \mathbb{R}^{\ell}$ is \emph{isotone} if
$u_i \leq v_i$ for all $i = 1, \cdots, k$ implies
$\phi_j(u) \leq \phi_j(v)$ for all $j = 1, \cdots, \ell$.  A real-valued function
$\psi : \mathbb{R}^n \to \mathbb{R}$ is directionally differentiable at a vector
$x$ if the limit
\[
\psi^{\, \prime}(x;dx) \, \triangleq \, \displaystyle{
	\lim_{\tau \downarrow 0}
} \, \displaystyle{
	\frac{\psi(x + \tau \, dx) - \psi(x)}{\tau},
}
\]
called the directional derivative of $\psi$ at $x$ along the direction $dx$,
exists for all $dx \in \mathbb{R}^n$.  If $\psi$ is locally Lipschitz continuous
near $x$, then the Clarke directional derivative
\[
\psi^{\, \circ}(x;dx) \, \triangleq \, \displaystyle{
	\limsup_{z \to x; \, \tau \downarrow 0}
} \, \displaystyle{
	\frac{\psi(z + \tau \, dx) - \psi(z)}{\tau}
}
\]
is well defined for all $dx \in \mathbb{R}^n$.  The Clarke subdifferential of $\psi$ at
$x$ is the compact convex set
\[
\partial_C \,  \psi(x) \, \triangleq \, \left\{ \, v \, \in \, \mathbb{R}^n \, \mid \,
\psi^{\, \circ}(x;dx) \, \geq \, v^{\top}dx, \;\, \forall \, dx \, \in \, \mathbb{R}^n \,
\right\}.
\]
The function $\psi$ is \emph{Clarke regular} at $x$ if it is directionally differentiable
at $x$ and $\psi^{\, \prime}(x;dx) = \psi^{\, \circ}(x;dx)$ for all
$dx \in \mathbb{R}^n$. A real-valued convex function $\psi$ is always Clarke regular and its subdifferential is denoted by $\partial \, \psi(x)$ .
For a function $V(x,y)$ of two arguments such that
$V(\bullet,y)$ is convex, we write $\partial_x V(x,y)$ for the subdifferential
of the latter convex function $ V(\bullet, y)$   at $x$.
We refer to \cite{Clarke83} for basic results in nonsmooth analysis, particularly those
pertaining to the Clarke subdifferential.  Subsequently, we will freely use such results
when they arise.  The set of minimizers of a real-valued function $\psi$ on a set $X$ is
denoted $\displaystyle{
\operatornamewithlimits{\mbox{argmin}}
} \,\{ \psi(x): {x \in X}\}$.  When this set is a singleton with the single element $\bar{x}$; we will
write $\bar{x} = \displaystyle{
	\operatornamewithlimits{\mbox{argmin}}
} \, \{ \psi(x): {x \in X}\}$.  A major assumption in the present paper has to do with \emph{random sets} and \emph{integrable selections}; see e.g.\
\cite{artstein1981law} and \cite[Chapter~14]{rockafellar2009variational}
for properties that we will freely use below.  In particular, with the  probability space $(\Omega, \mathcal{A}, \mathbb{P})$, we say
a random set $\Phi$ in $\mathbb{R}^{\ell}$ is a
measurable set-valued map if it assigns to each element $\omega \in \Omega$
a nonempty closed set $\Phi(\omega) \subseteq \mathbb{R}^{\ell}$. We say that this random set $\Phi$ admits a $\mathbb{P}$-\emph{integrable selection}
if there exists a measurable function $\phi$ such that $\phi(\omega) \in \Phi(\omega)$ for almost all $\omega \in \Omega$, and
$\int_{\Omega} \norm{\phi(\omega)} \, d \,  \mathbb{P}(\omega)$ is finite. The expectation of a random set $\Phi$ with integrable
selections is the set of integrals:
\[
\int_{\Omega} \, \Phi(\omega) \,d \, \mathbb{P}(\omega)  \, \triangleq \,
\left\{ \, \int_{\Omega} \, \phi(\omega) \,d \, \mathbb{P}(\omega)  \, \Big| \, \, \phi \mbox{ is a
	$\mathbb{P}$-integrable selection of $\Phi$} \, \right\}.
\]
The rest of the paper is organized as follows. In Section~\ref{sec: algorithm}, we
present the promised stochastic majorization-minimization (SMM) algorithm for solving the
compound SP \eqref{eq:compound_composite_SP}.  The subsequential convergence of the
proposed algorithm is established in Section~\ref{sec: convergence}.
Following that, in Section~\ref{sec:stationarity} we discuss the stationary properties of the
accumulation points generated by the algorithm
for some special classes of   surrogate functions used in the SMM algorithm.  In
Section~\ref{sec:error_bounds}, we develop an error bound theory for deterministic,
nonconvex and nonsmooth programs based on convex surrogate functions.  This theory
allows us to design a  probabilistic stopping rule for the SMM algorithm.
Applications of the compound SP to the risk measure minimization problems mentioned
above are discussed in Section~\ref{sec:application}.  Finally, Section~\ref{sec:numerics}
provides computational results on a set of preliminary experiments that
lend support to the practical potential of the developed SMM algorithm.

\section{The Stochastic Majorization-Minimization Algorithm} \label{sec: algorithm}
	
In this section, we present the promised stochastic majorization-minimization  algorithm
for solving the compound SP \eqref{eq:compound_composite_SP}.  The algorithm
consists of solving a sequence of  convex subproblems.  The construction
of each such subproblem starts with the use of convex majorants of the component random
functions $\{G_i\}$ and $\{F_j\}$; these surrogate functions are then embedded in the two
functions $\psi$ and $\varphi$ along with
approximations of the double expectations by sample averages.  These two approximations
define an overall sampling-based convex surrogation of the objective function
$\Theta$.  The minimization of the resulting convexified function plus a proximal term completes a general
iterative step of the algorithm.  This approach of sequential sampling
has the advantage that it can respond to data arriving incrementally and also it can potentially achieve
much progress towards solutions in early iterations with small sample sizes.
	
\subsection{Assumptions} \label{subsec:assumptions}
	
We impose the following blanket assumptions on the problem (\ref{eq:compound_composite_SP}) as the basic requirements to define the
SMM algorithm: % for solving (\ref{eq:compound_composite_SP}):
	
\gap
	
\textbf{(A1)} the feasible set $X$ is nonempty, convex, and compact; {thus, it has a
finite diameter $D$};
	
\gap
	
\textbf{(A2)} $\psi$ and $\varphi$ are isotone and Lipschitz continuous
functions with Lipschitz constants $\mbox{Lip}_{\psi}$ and $\mbox{Lip}_{\varphi}$,
respectively; moreover, $\psi$ and each $\varphi_j$
(for $j = 1, \cdots, \ell_{\varphi}$) are convex;
	
\gap
	
\textbf{(A3$_G$)} for every % $x^{\, \prime} \in X$ and $\xi \in \Xi$,
pair $( \xi,x^{\, \prime} ) \in \Xi \times X$,
there exists a family ${\cal G}(x^{\, \prime}, \xi)$ consisting of functions
$\wh{G}(\bullet,\xi;x^{\, \prime}) : {\cal O} \to \mathbb{R}^{\ell_G}$
satisfying the following conditions (1)--(4):
	
\gap
	
(1) \emph{a touching condition}: $\wh{G}(x^{\, \prime},\xi;x^{\, \prime}) = G(x^{\, \prime}, \xi)$;
% $\wh{g} = G(x,\xi)$ for all $\wh{g} \, \in \, \wh{G}(x,\xi;x)$ and $(\xi;x) \, \in \, \Xi \, \times \, X$;
	
\gap
	
(2) \emph{majorization}: $\wh{G}(x,\xi;x^{\, \prime}) \geq G(x,\xi)$ for any $x \in X$;
% $\wh{g} \geq G(x,\xi)$ for all $\wh{g} \, \in \, \wh{G}(x,\xi;x^{\, \prime})$ and all $( x,\xi;x^{\, \prime} ) \, \in \, X \, \times \, \Xi \, \times X$;

\gap
	
(3) \emph{convexity}: each $\wh{G}_i(\bullet,\xi;x^{\, \prime})$ for
$i = 1, \cdots, \ell_G$ is a convex function on $X$;
% \item[\rm (4)] {\sl integrability}: each $\wh{G}_i(x,\bullet;x^{\, \prime})$ for $i = 1, \cdots, \ell_G$
% is an integrable function on $\Xi$ for each pair $(x,x^{\prime}) \in X \times X$;
	
\gap
	
(4) \emph{uniform outer semicontinuity}: the set-valued function ${\cal G}(\bullet, \xi)$
which maps from $X$ to a family of functions % is {\bf outer-semicontinuous}
satisfies the condition that there exists an % finite-valued
integrable function
$L : \Xi\subseteq \mathbb{R}^m \to \mathbb{R}_+$ such that for any
sequence $\{ x^{\, \nu} \} \subset X$ converging to
$x^{\, \infty}$ and any % bounded
sequence of functions
$ \{ \, \wh{G} (\bullet,\xi;x^{\, \nu}) \,  \}$ with
$\wh{G} (\bullet,\xi;x^{\, \nu}) \in {\cal G}(x^{\, \nu}, \xi)$
for all $\nu$ {and $\displaystyle{
\sup_{\nu}
} \, \displaystyle{
\sup_{x \in X}
} \, \| \, \wh{G} (x,\xi;x^{\, \nu}) \, \| < \infty$}, there exist a function
$\wh{G}(\bullet,\xi;x^{\, \infty})$ in the family ${\cal G}(x^{\, \infty}, \xi)$
and an infinite subset $\mathcal{K} \subseteq \{ 1, 2, \cdots \}$ such that for all $x \in X$ and any
$\varepsilon > 0$,  an integer $\bar{\nu}(x,\varepsilon)$ exists satisfying
\begin{equation} \label{eq:uniform continuity}
\| \, \wh{G} (x,\xi;x^{\, \nu}) - \wh{G} (x,\xi;x^{\, \infty}) \, \|
\, \leq \, \varepsilon \, L(\xi), \epc \mbox{for all $\nu \, (\in \mathcal{K}) \, \geq \,
\bar{\nu}(x,\varepsilon)$,  for  all $\xi \in \Xi$};
\end{equation}

	% notation-wise, this property is written as
	% $\displaystyle{
	% \limsup_{x \to x^{\, \infty}
	% } \, {\cal G}(\xi,x) \subseteq \, {\cal G}(\xi^{\, \infty},x^{\, \infty}) +
	% o(1) L(\xi) \mathbb{B}$ where $\mathbb{B}$ is the unit ball in $\mathbb{R}^{\ell_G}$;
	
\textbf{(A3$_F$)} this is the counterpart of (A3$_G$) for the function $F$; details are not repeated.  	 		
	
\gap
	% Conditions (A3$_G$) and (A3$_F$) are the key to the convergence of the
	% iterative algorithm for solving the problem
	% (\ref{eq:compound_composite_SP}).
Any function $\wh{G}(\bullet,\xi;x^{\, \prime}) \in {\cal G}(x^{\, \prime}, \xi)$ satisfying (1)
and (2) in (A3$_G$), is called a \emph{(majorizing) surrogate} of $G(\bullet,\xi)$
at $x^{\, \prime}$; similarly for
the family $\mathcal{F}(x^{\, \prime}, \xi)$.  Condition (3) in (A3$_G$) enables the
implementation of convex programming solvers for solving the sequence of convex
subproblems in the algorithm. However, this convexity assumption can be relaxed; see the remark
following the  proof of Theorem \ref{thm:convergence_smm} and the subsequent
Corollary \ref{co:convergence without convexity}. % The uniformity in  condition (4) requires that \eqref{eq:uniform continuity} holds for all $\xi \in \Xi$.
The uniform outer semicontinuity assumption on the family of surrogate functions is critical
in the convergence analysis of the developed algorithm.
% Condition (5) in (A3$_G$) intends to guarantee the law of large numbers of random sets.

\gap

Isotonicity in assumption (A2) is the multivariate   extension of   the nondecreasing property for  univariate functions.
The isotonicity-plus-convexity condition  is generalized from the basic pointwise maximum function and sum function; specifically, the two multivariate functions $
( t_1, \cdots, t_k ) \, \mapsto \, \displaystyle{ \max\{ t_i : 1 \leq i \leq k\}}$ and $ ( t_1, \cdots, t_k ) \, \mapsto \,  { \sum_{i=1}^k} \ t_i, $
are both isotone and convex;
{suitable compositions of these basic functions with convex functions will remain isotone and convex. For instance, the OCE-of-deviation optimization is a class of compound SP with such isotone and convex functions $\psi$ and $\varphi$. Specifically, given a proper closed concave and nondecreasing utility function $u: \mathbb{R} \to [-\infty, \infty)$ with $u(0)=0$ and $1 \in \partial u(0)$,
the OCE of a random variable $Z$ is defined as
$S_u(Z) \triangleq \displaystyle{
\sup
} \, \left\{ \eta +	\mathbb{E}[ u (Z-\eta)]: {\eta \in \mathbb{R}} \right\}$.  With the loss function $f(x,\xi)$,
a random variable $\tilde{\xi}$, and a compact convex
subset $X$ of $\mathbb{R}^n$, the OCE-of-deviation (from the mean) optimization problem is:
\begin{equation} \label{eq:OCE-deviation objective}
 \underset{{x\in X, \eta \in \mathbb{R}}}{\mbox{minimize}} \,\, \Big\{ \, -\eta -
	\mathbb{E}\Big[ \, u\Big( \, f(x,\tilde{\xi} ) - \mathbb{E}\left[ \,
	f(x,\tilde{\xi} ) + \eta \, \right] \, \Big)
	\, \Big] \, \Big\}.
\end{equation}
Under appropriate assumptions (see Subsection~\ref{sec:OCE-of-deviation optimization}), the domain of $\eta$ can be restricted to a closed interval.
The objective in (\ref{eq:OCE-deviation objective}) is of the form (\ref{eq:compound_composite_SP})
% Its objective can be expressed as  $\mathbb{E}\big[ \,
% \varphi( \, G(x,\eta,\tilde{\xi} ), \mathbb{E} \, [  F(x,\eta,\tilde{\xi} ) ] \, ) \, \big] ,$
with $\psi$ being the identity function, and
\begin{equation} \label{eq:OCEFandG}
F(x,\eta, \xi) \, = \, f(x, \xi) + \eta, \,\, 
G(x,\eta, \xi) \, = \, \left( 
-f(x, \xi), \,
-\eta  \right)^\top \ \mbox{ and } \
\varphi(y_1,y_2, y_3) \, \triangleq \, - u(-y_1 - y_3) + y_2.
\end{equation}
As the composition of the bi-variate convex function
$(t_1,t_2) \mapsto -u(t_1) + t_2$ with the (vector) tri-variate linear function $(t_1,t_2,t_3) \mapsto (-t_1-t_3, t_2)$, the
function $\varphi$ is clearly convex and isotone. % See (\ref{eq:dro_mixed_bpoe_mixture}) for a class of compound SPs where $\psi$ is not the identity function.}
More examples of compound SPs with isotone and convex functions $\psi$ and $\varphi$, and a non-identity function $\psi$ are presented
in Section~\ref{sec:application}. }

% Examples of such functions also include the negative of concave and nonincreasing utility functions; see the OCE-of-deviation optimization in  .
% ; thus, if $u$ is a multivariate concave and nonincreading utility function and $f$ is any isotone plus convex function, then the composite
% function $x \mapsto -u \circ f(x)$ is isotone and convex.
\gap

The two properties---isotonicity and convexity--- of $\psi$ and $\varphi$ play a major role in the algorithmic development
for solving the compound SP; specifically, they enable the  construction of an
upper surrogation function
\[
\psi\left( \, \mathbb{E}\left[ \,
\varphi\left( \wh{G}(x, \tilde{\xi} \, ), \, \mathbb{E} \, [ \, \wh{F}(x, \tilde{\xi} \, ) \, ] \right) \, \right] \, \right)
\]
of the overall objective function $\Theta$ given the upper surrogates $\wh{F}$ and $\wh{G}$ of the inner functions (cf.  Assumption (A3)).
In particular, when the outer functions $\psi$ and $\varphi$ are univariate convex functions without being isotone, we can decompose each into the sum of a
nonincreasing convex function and a nondecreasing convex function, and develop an iterative convex-programming based algorithm by using
both the convex majorizing functions and the concave minorizing functions of $G$ and $F$; such a technique for deterministic composite optimization
problems has been analyzed in \cite{cui2018composite}.
Chapter~7 in the forthcoming monograph \cite{cui2020non} discusses in great detail and generality about the essential role of surrogation
and isotonicity in nonconvex and nonsmooth optimization.   Since our interest in this paper stems from the applications which
have multivariate outer functions $\psi$ and $\varphi$, % in the mentioned section,
our goal is to present a unified class of problems with properties such as isotonicity that encompass the problem classes therein;
this goal leads to the study of problem (\ref{eq:compound_composite_SP}) under the above setting.
	
\gap
	
The conditions (A3$_G$) and (A3$_F$) can be satisfied by the class of
difference-of-convex (dc) functions which appear in the risk-based optimization
problems presented in Section~\ref{sec:application}.   The fundamental roles of the
class of dc functions in optimization and statistics are well documented in the paper
\cite{MaherPangRaza2018}.  Suppose that for each $i = 1, \cdots, \ell_G$,
$G_i(\bullet,\xi)$ is such a function on $X$ with the dc decomposition
$G_i(x,\xi) = g_i(x,\xi) - h_i(x,\xi)$ that satisfies the following assumptions:
	
	\gap
	
{\bf (DC1)} $\{g_i(\bullet,\xi)\}_{\xi \in \Xi}$ and $\{h_i(\bullet,\xi)\}_{\xi \in \Xi}$ are convex functions and  uniformly
continuous on  $\Xi$
with finite expectation functions $\mathbb{E} [ h_i (x, \tilde{\xi})] $ and $\mathbb{E} [ g_i (x, \tilde{\xi})] $.
	\begin{comment}
	that is, for any
	sequence $\{ x^{\, \nu} \} \subset X$ converging to $x^{\, \infty}$,
	\[
	\displaystyle{
	\lim_{\nu \to \infty}
	} \, g_i(x^{\nu},\xi) \, = \, g_i(x^{\, \infty},\xi) \ \mbox{ and } \
	\displaystyle{
	\lim_{\nu \to \infty}
	} \, h_i(x^{\nu},\xi) \, = \, h_i(x^{\, \infty},\xi),
	\epc \mbox{both uniformly in $\xi$}.
	\]
	\end{comment}
	\gap
	
	Under the above assumption on the pair $(g_i,h_i)$,
	for any given $x^{\, \prime} \in X$, we can construct the family
	$\mathcal{G}(x^{\, \prime}, \xi)$ as follows:
	\begin{equation} \label{eq:subgradient linearization}
	{
		\mathcal{G}(x^{\, \prime}, \xi) =  \left\{ \begin{array}{l}
		\wh{G}(\bullet,\xi;x^{\, \prime}) \,
		\triangleq \, \left( \, \wh{G}_i(\bullet, \xi;x^{\, \prime}) \, \right)_{i=1}^{\ell_G}
		\, :\\[0.15in]
		\wh{G}_i(x, \xi;x^{\, \prime})  = g_i(x,\xi) - h_i(x^{\, \prime},\xi) -
		a^{i}(x^{\, \prime},\xi)^{\top} ( x - x^{\, \prime}),  \,  \mbox{for all } \, x \in X, \\[0.1in]
		\mbox{with } a^{i}(x^{\prime},\xi) \in \partial_x h_i(x^{\, \prime},\xi)
 \mbox{ for all $i =1, \cdots, \ell_G$} 	\end{array} \right\}.
	}
	%\wh{G}_i(x,\xi;x^{\, \prime}) \, &= \, g_i(x,\xi) - h_i(x^{\, \prime},\xi) -
	%a^{\, i}(x^{\, \prime},\xi)^{\top} ( \, x - x^{\, \prime} \, )
	% \mbox{and each given subgradient} a^{\, i}(x^{\, \prime},\xi)
	% \in \partial_x h_i(x^{\, \prime},\xi),
	\end{equation}
%	where $\partial_x h_i(x^{\, \prime},\xi)$ denotes the subdifferential of the 	convex function $h_i(\bullet,\xi)$ at $x^{\, \prime}$.
% Note that if $h_i(\bullet,\xi)$ is differentiable, then	$\partial_x h_i(x^{\, \prime},\xi)$ consists of the single element 	$\nabla_x h_i(x^{\, \prime},\xi)$.
In order to satisfy the uniform
	outer semicontinuity condition (4) in (A3$_G$) for the family $\mathcal{G}(\bullet, \xi)$, we
	impose an additional assumption on the functions
	$\{ h_i \}_{i = 1}^{\ell_G}$.
	
	\gap
	
	\textbf{(H)}  for each $i=1, \ldots, \ell_G$, the subdifferential of $h_i(\bullet,\xi)$
	is uniformly outer semicontinuous on $X$.
	% in the sense that there exists a
	% finite-valued integrable function $L_i : \mathbb{R}^m \to \mathbb{R}_+$ such that
	% for any sequence
	% $\{ x^{\, \nu} \}$ converging to $x^{\, \infty}$ and any sequence
	% $\{ a^{\, i,\nu}(x^{\, \nu},\xi) \}$ with
	% $a^{\, i,\nu}(x^{\, \nu},\xi) \in \partial_x h_i(x^{\, \nu},\xi)$ for
	% all $\nu$, there exist an infinite subset $\mathcal{K}$ of $\{ 1, 2, \cdots \}$ and
	% $a^{\, i,\infty}(x^{\infty},\xi) \in \partial_x h_i(x^{\infty},\xi)$ satisfying the
	% following condition: for every $\varepsilon > 0$, an integer
	% $\bar{\nu}(\varepsilon) > 0$ exists, such that
	% \[
	% \| \, a^{\, i,\nu}(x^{\nu},\xi) - a^{\, i,\infty}(x^{\infty},\xi) \, \| \, \leq \,
	% \varepsilon \, L_i(\xi),
	% \epc \mbox{for all $\nu \, (\in \mathcal{K}) \, \geq \, \bar{\nu}(\varepsilon)$
	% and all $\xi \, \in \, \Xi$}.
	% \]
	
	\gap
	
We remark that condition (H) is satisfied when $h_i(\bullet,\xi)$ is differentiable and its
gradient $\nabla_x h_i(\bullet,\xi)$ is continuous uniformly in $\xi$.  For a general extended-valued
convex function, its subdifferential may not be outer semicontinuous.  But for convex functions with finite values,
their subdifferentials are known to be outer semicontinuous.   By
assuming further that its subdifferential is outer semicontinuous uniformly with
respect to $\xi \in \Xi$, the condition (H) holds. Hence,
it is  easy to verify that the family ${\cal G}(\bullet, \xi)$ under the conditions (DC1) and (H) satisfies the four
conditions in (A3$_G$), which is summarized in the proposition below.
A similar result can be stated for  $\{ F_j \}$ which we omit.

\begin{proposition}
Suppose that for each $i = 1, \cdots, \ell_G$,
$G_i(\bullet,\xi)$ is a function on $X$ with the dc decomposition
$G_i(x,\xi) = g_i(x,\xi) - h_i(x,\xi)$ that satisfies (DC1) and (H).
Then the family $\mathcal{G}(\bullet, \xi)$ constructed in \eqref{eq:subgradient linearization} satisfies the
four conditions in (A3$_G$).  \hfill $\Box$
\end{proposition}

\subsection{The SMM algorithm}
	
	Under the pair of assumptions (A3$_G$) and (A3$_F$), we construct below a family of
	convex majorants
	$\wh{\cal V}(x^{\, \prime}) \triangleq
	\left\{ \, \wh{V}(\bullet;x^{\, \prime}) \, \right\}$
	of the objective function $\Theta(x)$ where each function
	$\wh{V}(\bullet;x^{\, \prime}) : {\cal O} \to \mathbb{R}$ is determined by functions $\wh{G}(\bullet,\xi;x^{\, \prime})$ and
	$\wh{F}(\bullet,\xi;x^{\, \prime})$ from the families ${\cal G}(x^{\, \prime}, \xi)$ and
	${\cal F}(x^{\, \prime}, \xi)$ respectively;  for every $ \xi \in \Xi $:
	\begin{equation} \label{eq:Theta surrogate}
	\wh{V}(x;x^{\, \prime}) \, \triangleq \, \psi\Big( \, \mathbb{E} \Big[ \,
	{\varphi}\left( \, \wh{G}(x,\tilde{\xi}; x^{\, \prime}),
	\mathbb{E}[\wh{F}(x,\tilde{\xi}; x^{\, \prime})]
	\,\right) \Big]  \Big).
	\end{equation}
    The function $\wh{V}$ is not always well defined unless $\wh{G}(x,\tilde{\xi};x^{\, \prime})$ and
	$\wh{F}(x,\tilde{\xi};x^{\, \prime})$ are integrable functions with respect to the probability measure $\mathbb{P}$
    for given $(x,x^{\, \prime})$.
	It turns out that such integrability condition is not needed in the algorithm,
	since we discretize the expectations by sampling
	and combine the iterative sampling with the iterative majorized minimization  in the SMM algorithm.
	
	\gap
	
	Specifically, at iteration $\nu$, we approximate the compound expectations by sample
	averages using two independent sets of i.i.d. samples of size $N_\nu$, each of which
	is composed of
	past samples of size $N_{\nu-1}$ and a new  sample set of size
	$\Delta_\nu$, i.e., $N_\nu = N_{\nu-1} + \Delta_\nu$.  Writing
	$[N] \triangleq \{ 1, \cdots, N \}$ for any positive integer $N$, we let
	$\{\xi^{\, t}\}_{t \in [N_{\nu-1}]}$ and $\{\eta^{\, s}\}_{s \in [N_{\nu-1}]}$
	denote the sample sets utilized at iteration $\nu-1$, and let
	$\{ \xi^{N_{\nu-1}+t}\}_{t \in [\Delta_{\nu}]}$ and
	$\{\eta^{N_{\nu-1}+s}\}_{s \in [\Delta_{\nu}]}$ denote the independent sample
	sets generated at iteration $\nu$; then the sample sets utilized at iteration $\nu$ are
	given by
	\begin{equation} \label{eq:samples}
	 \xi^{\, t} \, \}_{t\in [N_{\nu}]} \, \triangleq \,
	\{ \, \xi^{\, t} \, \}_{t\in [N_{\nu-1}]} \, \cup \,
	\{\xi^{N_{\nu-1}+t}\}_{t \in [\Delta_{\nu}]};  \quad	\{ \, \eta^{\, s} \, \}_{s\in [N_{\nu}]} \, \triangleq \,
  \{\, \eta^{\, s} \, \}_{s\in [N_{\nu-1}]} \, \cup \,
	\{ \, \eta^{N_{\nu-1}+s} \, \}_{s\in [\Delta_{\nu}]} .
\end{equation}
	Given the last iterate $x^\nu$ along with the families $\{\mathcal{G}(x^{\nu}, \xi^t)\}_{t \in [N_\nu]}$ and
	$\{\mathcal{F}(x^\nu, \xi^s)\}_{s \in [N_\nu]}$, the family of discretized convex majorization function is
	as follows:
	\begin{equation} \label{eq:saa_mm_family_approximation}
	\wh{\cal{V}}_{N_\nu}(x^\nu) \, \triangleq \, \left\{ \begin{array}{l}
	\wh{V}_{N_{\nu}}(\bullet\, ; x^{\, \nu}) \, : \mbox{for any } x\in X, \\ [0.1in]
	\wh{V}_{N_{\nu}}(x; x^{\, \nu}) \, = \, \psi\Bigg( \,
	\displaystyle{
		\frac{1}{N_{\nu}}
	} \, \displaystyle{
		\sum_{t=1}^{N_{\nu}}
	} \, \varphi\Bigg( \, \wh{G}^{\, t}(x,\xi^{\, t}; x^{\, \nu}), \, \displaystyle{
		\frac{1}{N_{\nu}}
	} \, \displaystyle{
		\sum_{s=1}^{N_{\nu}}
	} \, \wh{F}^{\, s}(x,\eta^{\, s};x^{\, \nu}) \Bigg) \, \Bigg)  \\ [0.25in]
	\mbox{ with}\; \wh{G}^{\,t}(\bullet, \xi^t \, ; x^\nu) \, \in \, \mathcal{G}(x^\nu, \xi^t) \,
	\mbox{ and } \, \wh{F}^{\, s}(\bullet, \eta^s \, ; x^\nu) \, \in \, \mathcal{F}(x^\nu, \eta^s)
	\end{array} \right\}.
	\end{equation}
	% By choosing the surrogate function pair
	% $\left( \, \wh{G}(\bullet,\bullet;x^{\, \nu}),\wh{F}(\bullet,\bullet;x^{\, \nu}) \,
	% \right) \in {\cal G}(x^{\nu}) \times {\cal F}(x^{\nu})$,   a convex majorization
	% function can be formulated as follows:
	% \begin{equation} \label{eq:saa_mm_approximation}
	% \wh{V}_{N_{\nu}}(x; x^{\, \nu}) \, \triangleq \, \psi\left( \,
	% \displaystyle{
	% \frac{1}{N_{\nu}}
	% } \, \displaystyle{
	% \sum_{t=1}^{N_{\nu}}
	% } \, \varphi\left( \, \wh{G}(x,\xi^{\, t}; x^{\, \nu}), \, \displaystyle{
	% \frac{1}{N_{\nu}}
	% } \, \displaystyle{
	% \sum_{s=1}^{N_{\nu}}
	% } \, \wh{F}(x,\eta^{\, s};x^{\, \nu}) \right) \, \right).
	% \end{equation}
	By the convexity of
	$\wh{V}_{N_{\nu}}(\bullet;x^{\, \nu})$, the function $
	\wh{V}_{N_{\nu}}(\bullet;x^{\, \nu}) + \displaystyle{
	\frac{1}{2 \, \rho}
	} \, \| \, \bullet - x^{\, \nu} \, \|_2^2$
	has a unique minimizer on the  convex
	set $X$ for any positive scalar $\rho$.  We are now ready to present the SMM algorithm.
	
	% \begin{algorithm}[h]
	% 	\label{compound_composite_sp_inner_proximal_method}
	%
	%	\caption{Stochastic Majorization-Minimization  Method}
	
	\noindent\makebox[\linewidth]{\rule{\textwidth}{1pt}}
	\vspace{-0.1in}
	
	\noindent
	{\bf The Stochastic Majorization-Minimization Algorithm (SMM)}\\
	\noindent\makebox[\linewidth]{\rule{\textwidth}{1pt}}
	
	\begin{algorithmic}[1]
		\STATE \textbf{Initialization:} Let $x^0\in X$ and a positive scalar
		$\rho$ be given.  Set $N_0 = 0$.
		% $\underline{\rho} < \bar{\rho}$.  Set $N_0 = 0$
		% and $\alpha\in(0,1/2)$,
		\FOR {$\nu=1,2, \cdots$,}
		% choose a scalar $\rho \in ( \underline{\rho}, \bar{\rho} )$ and}		
		\STATE generate two independent sample sets
		$\{ \xi^{N_{\nu-1}+t} \, \}_{t \in [\Delta_{\nu}]}$ and
		$\{ \, \eta^{N_{\nu-1}+s} \, \}_{s \in [\Delta_{\nu}]}$  i.i.d. from the probability
		distribution of the random variable $\tilde{\xi}$ that are also independent from the past
		samples;
		\STATE choose a member $\wh{V}_{N_{\nu}}(\bullet; x^\nu)$ from the family
		$\wh{\cal V}_{N_{\nu}}(x^{\, \nu})$ according to
		\eqref{eq:saa_mm_family_approximation};
		\STATE compute $x^{\nu+1} = \underset{x \, \in \, X}{\mbox{argmin}} \left\{ \,		
		\wh{V}_{N_{\nu}}(x; x^\nu) + \displaystyle{
			\frac{1}{2 \rho}
		} \, \| x -x^{\nu}\|^2 \, \, \right\}$ .
		\ENDFOR
	\end{algorithmic}
	% \end{algorithm}
	\noindent\makebox[\linewidth]{\rule{\textwidth}{1pt}}

\gap

{We may state an inexact version of the algorithm, wherein each $x^{\nu+1}$ does not need to
be an exact solution of the stated convex subproblem.  Indeed, an error-bound rule that ensures $x^{\nu+1}$ is
$\delta_{\nu}$-suboptimal with the accuracy $\delta_{\nu} > 0$ being suitably restricted would be
sufficient to yield the same convergence result; see the remark following the proof of Theorem~\ref{thm:convergence_smm}.
}

\section{Convergence Analysis of the SMM Algorithm} \label{sec: convergence}
	
	% \subsection{Subsequential convergence to a fixed point}
	% We begin the convergence proof of the SMMA by providing some further setting of
	% the probability pair $(\Omega, \mathcal{A})$.
	% The family of all probability measures on this pair is denoted by
	% $\mathcal{P}(\Omega, \mathcal{A})$.  We denote by $\mathcal{L}^m(\Omega, \mathcal{A})$
	% the family of all measurable mappings from $(\Omega, \mathcal{A})$ to
	% $(\mathbb{R}^m, \mathcal{A}_B^m)$, where $\mathcal{A}_B^m$ is the $\sigma$-algebra
	% of $m$-dimensional Borel sets. Let $\mathcal{V}^m(\Omega, \mathcal{A}) \triangleq
	% \mathcal{P}(\Omega, \mathcal{A}) \times \mathcal{L}^m(\Omega, \mathcal{A})$.
	% In a pair $[\mathbb{P}, \xi] \in \mathcal{V}^m(\Omega, \mathcal{A})$ we view
	% the mapping $\xi:\Omega \to \mathbb{R}^m$ as a random variable on the probability
	% space $(\Omega, \mathcal{A}, \mathbb{P})$, with corresponding
	% CDF $F_{[\mathbb{P}, \xi]}.$
	
Our convergence proof of the SMM algorithm below hinges on the construction of  an approximate descent property of
the sequence of SAA objective values with sampling-based error terms, and also on
the finiteness of these accumulated errors by leveraging the non-asymptotic bounds of the SAA estimation errors
for the compound SP with a proper control of the sample size increment. 	%In proving the subsequential convergence, we establish  in Proposition \ref{prop:SAA_error_bound} the convergence rate of the SAA estimation error in expectation with weaker requirements on the random functionals than the results  in \cite{ermoliev2013sample,hu2020sample}; this rate yields the finite accumulation of the sampling errors in the SMM algorithm in expectation.
	
\gap

With two independent sample sets $\{\xi^{\, t}\}_{t=1}^n$ and $\{\eta^{\, s}\}_{s=1}^m$, the SAA function of the objective
$\Theta(x)$ is written by
\begin{equation}
\label{eq:saa_obj}
\overline{\Theta}_{n,m}(x) \, \triangleq \, \psi\Bigg( \, \displaystyle{
\frac{1}{n}
} \, \displaystyle{
\sum_{t=1}^{n}
} \, \varphi\Bigg( G(x,\xi^{\, t} \, ), \displaystyle{
\frac{1}{m}
} \, \displaystyle{
\sum_{s=1}^{m}
} \, F(x, \eta^{\, s} ) \, \Bigg) \, \Bigg).
\end{equation}

The convergence rate of SAA estimation error
$ \sup_{x\in X} \, \mathbb{E} \, \big| \, \overline{\Theta}_{n,m}(x) - \Theta(x) \, \big|$
% where $\overline{\Theta}_{n,m}(x)$ is defined in \eqref{eq:saa_obj},
is given in two references \cite{ermoliev2013sample,hu2020sample}; the former assumes
an appropriate bound of the Rademacher average, while the latter assumes the essential boundedness of random functionals.
The following proposition shows that with   a finite-variance condition and Lipschitz continuity, we can derive the uniform bound
$\mbox{O}\left( (\frac{1}{m}+ \frac{1}{n} +\frac{1}{mn})^{1/2} \right)$ for the SAA estimation error, which is better than the bound
$\mbox{O}((\frac{1}{m}+ \frac{1}{n} +\frac{1}{n\sqrt{m}} )^{1/2} )$ in \cite[Theorem 3.1]{hu2020sample} and the bound
$\mbox{O}( \, \frac{1}{\min\{m,n\}^{\beta}}  \, )$ with $\beta \in (0, 1/2)$ in \cite[Theorem 3.5]{ermoliev2013sample}
when $\tilde{\xi}$ is a continuous random variable and $F(x, \tilde{\xi})$ and $G(x, \tilde{\xi})$ are essentially bounded.

	\gap
	
	\begin{comment}
	With the Lipschitz continuity of the outer function $\psi$, the non-asymptotic
	bound of the SAA in \cite{ermoliev2013sample} is applied to the compound expectation function
	$\mathbb{E}[ \, \varphi_i \, ( \, G(x, \tilde{\xi} ), \mathbb{E}\,[\, F(x, \tilde{\xi} )\, ])]$,
	which can be written as $\mathbb{E}  [\, h_i ( x,\mathbb{E}[ \, F(x, \tilde{\xi} \, ) \,],
	\tilde{\xi} \,  ) \,  ]$ by letting $h_i(x,y,\xi) \triangleq \varphi_i(G(x, \xi), y)$
	for $i=1, \ldots, \ell_{\varphi}$.
	
	\gap
	
	The assumptions B and C in the cited reference are satisfied under the following additional assumptions on the functions $G$ and $F$.
	
	\gap
	
	Assumption B except condition (iii) and Assumption C except condition (iv) from
	\cite{ermoliev2013sample}  are readily satisfied under the blanket assumption of
	continuity, and assumptions (A1), (A2), (A4) and (A5) on the
	problem (\ref{eq:compound_composite_SP}).
	The rest of the conditions in the reference involve the Rademacher averages
	pertaining to the function $F$ and $h_i(x,y,\xi)$.
	For clarity, we present the formal definition of the Rademacher average and confirm
	the validity of the conditions on the Rademacher average imposed in the reference
	via the formal statement of Lemma~\ref{le:rademacher}.
	\end{comment}
	We make the following additional assumptions on the functions $G$ and $F$ and use
it to obtain the SAA estimation bound in expectation in Proposition \ref{prop:SAA_error_bound} with the proof given in Appendix.
	\gap
	
	\textbf{(A4)}
	$  \displaystyle{
		\sup_{x  \in X } \, \,
	} \Big( \, \mathbb{E} \,   \norm{ \, {F}(x, \tilde{\xi}) - \mathbb{E}\,[\,F(x, \tilde{\xi})\,]\, }^2 \Big) ^{1/2}\leq \sigma_F $ and 	$  \displaystyle{
	\sup_{x  \in X } \, \,
} \Big( \, \mathbb{E} \,   \norm{ \, {G}(x, \tilde{\xi}) - \mathbb{E}\,[\,G(x, \tilde{\xi})\,]\, }^2 \Big) ^{1/2}\leq \sigma_G $.

\begin{proposition}\label{prop:SAA_error_bound}

Let $\{\xi^t\}_{t=1}^n$ and $\{\eta^s\}_{s=1}^m$ be two independent sample sets generated from the probability distribution of the random variable $\tilde{\xi}$.
Suppose that the assumptions (A1) and  (A2)  hold. Then the following three statements hold:
\gap

{
(a) $\overline \Theta_{n,m}(x)$ converges to $\Theta(x)$ as $n, m \to +\infty$ with probability 1  uniformly on $X$.}
\gap

{
(b) with probability 1, there exist a positive integer $\bar{n}$ and a positive scalar $U$ such that for any $n,m \geq \bar{n}$,
$\displaystyle{
\sup_{x\in X}}
\,  \left|\overline \Theta_{n,m}(x) \right| < U$. }

\gap

{
(c) under Assumption (A4),   $
\displaystyle\sup_{x\in X} \, \mathbb{E}  \,  \left|{ \, \overline{\Theta}_{n,m}(x) - \Theta(x)  \, }\right|^2  \leq  (\mbox{Lip}_{\psi})^2  \,
(\mbox{Lip}_{\varphi})^2 \cdot  \ell_{\varphi}  \, \Bigg(  \frac{\sigma_F^{\,2}   }{m} + \frac{  \sigma_F^{\,2} \,  }{m\, n}  +
\displaystyle{\frac{ \sigma_G^{\,2}  \,  }{n}} \Bigg),$}
which implies that $
\displaystyle\sup_{x\in X} \, \mathbb{E}  \,  \left|{ \, \overline{\Theta}_{n,m}(x) - \Theta(x)  \, }\right|    \leq   \frac{C_1}{ \,  {\min\{m, n\}}^{1/2}}$ with some positive constant $C_1$.

\end{proposition}

\gap

The following lemma provides a sufficient condition on the sequence of sample sizes for the convergence of the SMM Algorithm with the proof given in Appendix.

\begin{lemma} \label{lm:condition on sample sizes} \rm
% For the  increasing sequence of positive  integers $\{ N_{\nu} \}$,
Suppose that there exist a positive integer $\bar{\nu}$ and scalars $c_1, c_2$ and $c_3$ with $c_3 < \bar{\nu}$,
\[
\max\left\{ \, N_{\nu-1} + 1, \, c_2 \,  \nu^{\, 1+ 2 c_1  } \, \right\} \, \leq \, N_{\nu} \, \leq \, \displaystyle{
 {N_{\nu-1}}\big/{\Big(1-\displaystyle{\frac{c_3}{\nu}}\Big)}
}, \epc \mbox{ for all } \nu > \bar{\nu}.
\]
Then we have $\displaystyle{
\sum_{\nu=1}^{\infty}
} \, \displaystyle{\Big(1 - \Big( \frac{ N_{\nu-1}}{N_{\nu}} \Big)^{\,2} \, \Big)
} \, \displaystyle{
\frac{1}{N_{\nu-1}^{\, 1/2}}
} \, < \, \infty \epc \mbox{and} \epc
\displaystyle{
\sum_{\nu=1}^{\infty}
} \, \displaystyle{
\frac{N_{\nu} - N_{\nu-1}}{N_{\nu}}
} \, \displaystyle{
\frac{1}{( \, N_{\nu} - N_{\nu-1} \, )^{ \, 1/2}}
} \, < \, \infty.$
\end{lemma}
	
	\begin{comment}
\proof{Proof.}
Since  $N_{\nu} \leq \nu \cdot  N_{\nu-1}$ for all $\nu > \bar{\nu}$, we have
$N_{\nu} - N_{\nu-1} \leq  c_3 \cdot \displaystyle{\frac{N_{\nu}}{\nu}} \leq c_3 \cdot N_{\nu-1} $ for all such $\nu$.  Then we have,
\[
\displaystyle{\Big(1 - \Big( \frac{ N_{\nu-1}}{N_{\nu}} \Big)^{\,2} \,\Big)
} \, \displaystyle{
\frac{1}{N_{\nu-1}^{\, 1/2}}} \, = \, \frac{N_{\nu} - N_{\nu-1}}{N_{\nu}} \cdot  \frac{N_{\nu} + N_{\nu-1}}{N_{\nu}} \cdot \frac{1}{N_{\nu-1}^{\,1/2}}
\, \leq \, 2 \, c_3^{\,1/2} \,
\frac{N_{\nu}- N_{\nu-1}}{N_{\nu}} \cdot \frac{1}{(N_{\nu} - N_{\nu-1})^{\,1/2}}.
\]
Hence, we only need to show $\displaystyle{
\sum_{\nu=1}^{\infty}
		} \, \displaystyle{
			\frac{N_{\nu} - N_{\nu-1}}{N_{\nu}}
		} \, \displaystyle{
			\frac{1}{( \, N_{\nu} - N_{\nu-1} \, )^{\, 1/2}}
		} \, < \, \infty.$ In fact,
		\[
		\begin{array}{lll}
		\displaystyle{
			\frac{N_{\nu} - N_{\nu-1}}{N_{\nu}}
		} \, \displaystyle{
			\frac{1}{(   N_{\nu} - N_{\nu-1}   )^{ \, 1/2}}
		} =  \left( \displaystyle{
			\frac{N_{\nu}-N_{\nu-1}}{N_{\nu}}
		} \right)^{1/2}   \displaystyle{
			\frac{1}{N_{\nu} ^{ \,  1/2}}}
		\leq  \left( \displaystyle{
			\frac{c_3 }{\nu}
		} \right)^{1/2} \displaystyle{
			\frac{1}{ (c_2 \cdot \nu^{1+2c_1  })^{1/2} }}  =   \displaystyle{
			\frac{c_3}{(c_2\cdot c_3)^{  1/2}}
		} \, \displaystyle{
			\frac{1}{\nu^{1+ c_1}}
		},
		\end{array}
		\]
		which establishes the finiteness of the second series.
		\hfill \Halmos
	\endproof
\gap

\end{comment}

{The Robbins-Siegmund nonnegative almost supermartingale convergence
lemma (\cite[Theorem~1]{robbins1971convergence})
is a fundamental tool for the subsequential convergence result.
	\begin{lemma}
	 \label{lm:R-S-supermartingale}
	 Let $\{y_k\}$, $\{u_k\}$, $\{a_k\}$, $\{b_k\}$ be sequences of nonnegative integrable
random variables, adapted to the filtration $\{\mathcal H_k\}$, such that for all $k \in \mathbb{N}$,
$\mathbb{E}\,[ \, y_{k+1} \mid \mathcal{H}_k\, ] \leq (1+a_k) \, y_k - u_k +b_k$,
$ {\sum_{k=1}^\infty}\, a_k  < + \infty$, and $ {\sum_{k=1}^\infty} \, b_k < + \infty $   almost surely.
Then with probability 1, $\{y_k\}$
converges and $ {\sum_{k=1}^\infty}  u_k < + \infty$.
\end{lemma}}

	\gap
	
	We are now ready to state and prove the following subsequential convergence result
	of the SMM algorithm. {In the rest of the analysis, we let $\Omega^{\,\nu}$ denote the $(2 \, N_\nu)$-fold Cartesian product of the sample space $\Omega$; let $\mathcal H_\nu$ and $\mathbb{P}_\nu$ be the $\sigma$-algebra
and the probability measure respectively generated by the tuple of random vectors
 $\big(\{ \xi^t\}_{t=1}^{N_\nu}, \,\{ \eta^s\}_{s=1}^{N_\nu}\big)$. We further let $\mathbb{E}_\nu[\,\bullet \, | \,  \mathcal H_{\nu-1}\,]$ denote
the conditional expectation on the probability space $(\Omega^{\nu}, \mathcal H_\nu, \mathbb{P}_\nu)$ given the $\sigma$-algebra $
\mathcal H_{\nu-1}$.}  For the subsequential convergence analysis, there are two sets of assumptions: one set---(A1), (A2), and (A4)---is for the functions
%	and (A5)---is for the functions
	defining the objective $\Theta$ of the problem (\ref{eq:compound_composite_SP}).
	The other assumption (A3) is for the surrogate functional families $\mathcal{G}$ and $\mathcal{F}$.
    % (A4$_{\infty}$) % (A5$_{\infty}$), and (A6$_{\infty}$) below---
	% is for surrogate functions $\wh{G}$ and $\wh{F}$ at an accumulation point,
Lastly, there is also a technical condition (E$_{\infty}$) at the limit point that enables 
% and \textcolor{red}{(E1$_\infty$) and (E2$_\infty$) are technical conditions that justify 
the application of a strong law of  large numbers for the sum of i.i.d.\ Banach space valued random sets \cite[Theorem~2.1]{puri1983strong}.
    % consisting of continuous Banach space valued random  sets.}
	% The latter three conditions are needed because the limiting minimization problem
	% (see conclusion below) is defined at such a point and remains a compound SP,
	% these conditions ensure the applicability of
	% the law of large numbers to the limiting problem.
	% cf.\ the last limit in (\ref{eq:LLN at acc point}).
	% The two sets of assumptions are linked by conditions
	% (A3$_G$) and (A3$_F$).
	
	%  $\mathbb{E}_\nu$  be the expectation operator induced by $\mathbb{P}_\nu$.   Then  so that the family $\{ \mathcal H_\nu \}$  is a filtration on the probability space $(\Omega, \mathcal{A} , \mathbb{P} )$.
	
	\begin{theorem} \label{thm:convergence_smm} \rm
		Let $\rho > 0$ be arbitrary.
		% the sequence of proximal parameters $\{\rho\}$ belong to an interval
		% $[ \, \underline{\rho},\,  \bar{\rho}\, ]$ where $0 < \underline{\rho} \leq
		% \bar{\rho}$.
Let the sequence of sample sizes $\{N_\nu\}$ satisfy the conditions in Lemma~\ref{lm:condition on sample sizes}.
Under assumptions (A1)--(A4), for any accumulation point $x^{\infty}$ of the sequence $\{x^\nu\}$
generated by the SMM algorithm, there exists a well-defined
$\wh{V}(\bullet; x^{\infty}) \in \mathcal{V}(x^{\infty})$ such that  
$x^{\infty} \in \underset{x\in X}{\mbox{argmin}} \, \wh{V}(x; x^{\infty})$ 
with probability 1 if the following  condition holds at $x^\infty$:

\gap
		
		\begin{comment}
		{\bf (A4$_{\infty}$)}  $  \,
		\sup\ \left\{ \,  \norm{\wh{G}(x,\xi;x^{  \infty})} :{( x,\xi ) \in X \times \Xi  }, \,  \wh{G}(\bullet, \xi; x^{\infty}) \in \mathcal{G}(x^{\infty}, \xi) \right\} \, < \, \infty$;\\
		
		\quad \quad \quad \, $  \,
		\sup\ \left\{ \,  \norm{\wh{F}(x,\xi;x^{  \infty})} :{( x,\xi ) \in X \times \Xi  }, \,  \wh{F}(\bullet, \xi; x^{\infty}) \in \mathcal{F}(x^{\infty}, \xi) \right\} \, < \, \infty$,
		\end{comment}

{\textbf{(E$_\infty$)} For almost all $\xi \in \Xi$, the family $\mathcal{G} (x^\infty, \xi)$ is a compact and convex subset in the (Banach) space of 
continuous functions  on a compact domain $X$ {equipped with the supremum norm} such that
% a compact convex set of a Banach space 
% containing continuous functions equipped with the norm $\| \bullet\|_B$. 
% $\sup\left\{\|G(\bullet, \xi; x^\infty)\|_B: G(\bullet, \xi; x^\infty) \in \mathcal{G} (x^\infty, \xi)\right \} %=  \sup \left\{ |G(x, \xi; x^\infty)|: 
% x \in  Y,  G(\bullet, \xi; x^\infty) \in \mathcal{G} (x^\infty, \xi)\right\} < \infty. $
% textcolor{red}{	\textbf{(E2$_\infty$)}
$\mathbb{E}\left[ \, \sup\left\{ \, \displaystyle{
\sup_{x \in {X}}
} \, \| \, \wh{G}(x, \tilde \xi; x^\infty)\| \, \mid \, \wh{G}(\bullet, \tilde \xi; x^\infty) \in \mathcal{G} (x^\infty, \tilde \xi) \, \right\} \, \right] < \infty$.}

\begin{comment}
		Furthermore, by the chain rule of subdifferential, we have
		\[
		\begin{array}{l}
		0 \in  \mbox{conv} \left\{ \begin{array}{ll}
		\sum_{i \in \mathcal{I}} \, u_iv_i: u_i \in \partial_i \phi(\cdot)\Big|_{\{ \mathbb{E}_{p_i(\cdot)} \,[\,\psi_i(c_i(x^{\infty}, \tilde{\xi}^i ), \tilde{\xi}^i)]\}_{i \in \mathcal{I}}}, \\[0.1in]
		v_i \in \mathbb{E}_{p_i(\cdot)} \left[\,\mbox{conv} \left\{ w_io_i: w_i \in \partial \psi_i(\cdot, \tilde{\xi}^i)|_{c(x^{\infty}, \tilde{\xi}^i)} , o_i \in \partial_C c_i(x^{\infty}, \tilde{\xi}^i)\right\} \right]
		\end{array}
		\right\} + \mathcal{N}(x^{\infty}; X \cap U^*).
		\end{array}
		\] where $U^* \triangleq \{x\in X: {W}(x; x^{\infty}) \leq 0\} $.
		\end{comment}
	\end{theorem}
	
\begin{proof}
		% In the proof below, we take, without loss of generality, the two
		% sequences of random variables $\{ \xi^{\, t} \}_{t=1}^{N_{\nu}}$ and
		% $\{ \eta^{\, s} \}_{s=1}^{N_{\nu}}$ to be the same.
		
Suppose that at the $\nu$th iteration, the sample average surrogate function $\wh{V}_{N_\nu} (x; x^{\, \nu})$ is constructed as follows
with $\wh{G}^{\,t}(\bullet, \xi^{\,t}; x^{ \nu}) \in \mathcal{G}(x^{  \nu}, \xi^{\,t})$ and
$\wh{F}^{\, s}(\bullet, \eta^{\, s}; x^{  \nu}) \in \mathcal{F} (x^{  \nu}, \eta^{\,s})$ for  $t , s \in [N_\nu]$,
		 \[
			\wh{V}_{N_{\nu}}(x; x^\nu) \, \triangleq \, \psi\Bigg( \, \displaystyle{
				\frac{1}{N_\nu}
			} \, \displaystyle{
				\sum_{t=1}^{N_\nu}
			} \, \varphi\Bigg( \wh{G}^{\, t}(x, \xi^{\, t}; x^\nu), \frac{1}{N_\nu}\sum_{s=1}^{N_\nu}
			\wh{F}^{\,s}(x, \eta^{\, s}; x^{\, \nu}) \Bigg) \, \Bigg).
			\]
		By the update rule in SMM algorithm, we derive
		\begin{equation} \label{eq:update_rule_mm_descent} 
			\overline{\Theta}_{N_{\nu}}(x^{\nu+1}) + \displaystyle{
				\frac{1}{2\rho}
			} \, \norm{x^{\nu+1}-x^\nu}^2  \leq   \wh{V}_{N_{\nu}}(x^{\nu+1}; x^{\nu}) +
			\displaystyle{
				\frac{1}{2 \rho}
			} \, \norm{x^{\nu+1}-x^{\, \nu}}^2  \leq   \wh{V}_{N_{\nu}}(x^{\, \nu}; x^{\, \nu}) \, = \,
			\overline{\Theta}_{N_{\nu}}(x^{\, \nu}). \end{equation}
		Since
	 $$\displaystyle{
				\frac{1}{N_{\nu}}
			} \, \displaystyle{
				\sum_{s=1}^{N_{\nu}}
			} \, F(x^{\, \nu}, \eta^{\, s}) = \, \displaystyle{
				\frac{N_{\nu-1}}{N_{\nu}}
			}  \Bigg[ \, \displaystyle{
				\frac{1}{N_{\nu-1}}
			} \, \displaystyle{
				\sum_{s=1}^{N_{\nu-1}}
			} \, F(x^{\, \nu}, \eta^{\, s}) \, \Bigg] + \displaystyle{
				\frac{N_{\nu} - N_{\nu-1}}{N_{\nu}}
			}  \Bigg[  \displaystyle{
				\frac{1}{N_{\nu} - N_{\nu-1}}
			} \, \displaystyle{
				\sum_{s=N_{\nu-1}+1}^{N_{\nu}}
			} \, F(x^{\, \nu}, \eta^{\, s}) \, \Bigg], $$
		by the convexity of each component function $\varphi_j$ for $j \in [ \ell_{\varphi}]$,
		we derive that, for all $t = 1, \cdots, N_{\nu}$,
		 \[ \begin{array}{ll}
			\varphi_j\Bigg( G(x^{\,\nu},\xi^{\, t}), \, \displaystyle{
				\frac{1}{N_{\nu}}
			} \, \displaystyle{
				\sum_{s=1}^{N_{\nu}}
			} \, F(x^{\,\nu}, \eta^{\, s}) \Bigg)  &  \leq  \, \displaystyle{
				\frac{N_{\nu-1}}{N_{\nu}}
			} \, \varphi_j\Bigg( G(x^{\, \nu},\xi^{\, t}), \, \displaystyle{
				\frac{1}{N_{\nu-1}}
			}   \displaystyle{
				\sum_{s=1}^{N_{\nu-1}}
			} \, F(x^{\, \nu}, \eta^{\, s}) \Bigg) + \\ [0.2in]
			& \quad \displaystyle{
				\frac{N_{\nu} - N_{\nu-1}}{N_{\nu}}
			} \, \varphi_j\Bigg( G(x^{\, \nu},\xi^{\, t}), \, \displaystyle{
				\frac{1}{N_{\nu} - N_{\nu-1}}
			}  \displaystyle{
				\sum_{s=N_{\nu-1}+1}^{N_{\nu}}
			}  F(x^{\, \nu}, \eta^{\, s}) \, \Bigg) .
			\end{array}
			\]
		Hence, for every $j \in [\ell_\varphi]$,
			\[
			\begin{array}{ll}
			 & \displaystyle{
				\frac{1}{N_\nu}
			} \, \displaystyle{
				\sum_{t=1}^{N_\nu}
			} \, \varphi_j\Bigg(G(x^\nu, \xi^{\, t}), \, \displaystyle{
				\frac{1}{N_\nu}
			} \, \displaystyle{
				\sum_{s=1}^{N_\nu}
			} \, F(x^\nu,\eta^{\, s}) \Bigg)\\[0.25in]
			% \leq \, & \displaystyle{ \frac{N_{\nu-1}}{N_{\nu}^2} } \, \displaystyle{ \sum_{t=1}^{N_{\nu}} } \, \varphi_j\left( G(x^{\, \nu},\xi^{\, t}), \,  \displaystyle{ \frac{1}{N_{\nu-1}} } \, \displaystyle{ \sum_{s=1}^{N_{\nu-1}} } \, F(x^{\, \nu}, \eta^{\, s}) \right) + \\ [0.2in] & \displaystyle{ \frac{N_{\nu} - N_{\nu-1}}{N_{\nu}^2} } \, \displaystyle{ \sum_{t=1}^{N_{\nu}} } \, \varphi_j\left( G(x^{\, \nu},\xi^{\, t}), \,  \displaystyle{ \frac{1}{N_{\nu} - N_{\nu-1}} } \, \displaystyle{ \sum_{t=N_{\nu-1}+1}^{N_{\nu}} } \, F(x^{\, \nu}, \eta^{\, s}) \, \right) \\ [0.2in]
			  = \, & \displaystyle{
				\frac{(N_{\nu-1})^2}{(N_{\nu})^2}
			} \, \Bigg[ \, \displaystyle{
				\frac{1}{N_{\nu-1}}
			} \, \displaystyle{
				\sum_{t=1}^{N_{\nu-1}}
			} \, \varphi_j\Bigg(G(x^\nu, \xi^{\, t}), \,  \displaystyle{
				\frac{1}{N_{\nu-1}}
			} \, \displaystyle{
				\sum_{s=1}^{N_{\nu-1}}
			} \, F(x^\nu,\eta^{\, s}) \Bigg) \, \Bigg] \\ [0.3in]
			& \quad + \, \displaystyle{
				\frac{N_{\nu-1} \left( \, N_{\nu} - N_{\nu-1} \, \right)}{(N_{\nu})^2}
			} \, \Bigg[ \, \displaystyle{
				\frac{1}{N_{\nu} - N_{\nu-1}}
			} \, \displaystyle{
				\sum_{t=N_{\nu-1}+1}^{N_{\nu}}
			} \, \varphi_j\Bigg(G(x^\nu, \xi^{\, t}), \displaystyle{
				\frac{1}{N_{\nu-1}}
			} \, \displaystyle{
				\sum_{s=1}^{N_{\nu-1}}
			} \, F(x^\nu,\eta^{\, s}) \Bigg) \, \Bigg]  \\ [0.3in]
			& \quad +\, \displaystyle{
				\frac{N_{\nu} - N_{\nu-1}}{N_{\nu}}
			} \, \Bigg[ \, \displaystyle{
				\frac{1}{N_{\nu}}
			} \, \displaystyle{
				\sum_{t=1}^{N_{\nu}}
			} \, \varphi_j\Bigg( G(x^{\, \nu},\xi^{\, t}), \,  \displaystyle{
				\frac{1}{N_{\nu} - N_{\nu-1}}
			} \, \displaystyle{
				\sum_{s=N_{\nu-1}+1}^{N_{\nu}}
			} \, F(x^{\, \nu}, \eta^{\, s}) \, \Bigg) \, \Bigg].
			\end{array}
			\]
		Since  $\displaystyle{
				\frac{(N_{\nu-1})^2}{(N_{\nu})^2}
			} + \displaystyle{
				\frac{N_{\nu-1} \, ( \, N_{\nu} - N_{\nu-1} \, )}{(N_{\nu})^2}
			} + \displaystyle{
				\frac{N_{\nu} - N_{\nu-1}}{N_{\nu}}
			} \, = \, 1$,
		by the isotonicity and convexity of $\psi$,  
		\[
			\begin{array}{ll}
			&\overline{\Theta}_{N_\nu}(x^\nu) - \overline{\Theta}_{N_{\nu-1}}(x^\nu)\\[0.1in]
			= &\psi\Bigg(  \displaystyle{
				\frac{1}{N_{\nu}}
			} \, \displaystyle{
				\sum_{t=1}^{N_\nu}
			} \, \varphi\Bigg( G(x^\nu, \xi^{\, t}), \displaystyle{
				\frac{1}{N_\nu}
			} \, \displaystyle{
				\sum_{s=1}^{N_\nu}
			} \, F(x^\nu,\eta^{\, s}) \Bigg)   \Bigg) \\[0.2in]
		&-
			\psi\Bigg( \displaystyle{
				\frac{1}{N_{\nu-1}}
			} \, \displaystyle{
				\sum_{t=1}^{N_{\nu-1}}
			} \, \varphi\Bigg( G(x^{\, \nu},\xi^{\, t}), \, \displaystyle{
				\frac{1}{N_{\nu-1}}
			} \, \displaystyle{
				\sum_{s=1}^{N_{\nu-1}}
			} \, F(x^{\, \nu}, \eta^{\, s}) \Bigg)   \Bigg) \\ [0.3in]
			\leq & e_{\nu,1} + e_{\nu,2} + e_{\nu,3},
			\end{array} \]
		where
	 	{\small\[
			\begin{array}{ll}
			e_{\nu,1} \, &\triangleq \, \left( \, \displaystyle{
				\frac{(N_{\nu-1})^2}{(N_{\nu})^2}
			}  \, -1 \, \right) \, \Bigg| \,
			\psi\Bigg( \, \displaystyle{
				\frac{1}{N_{\nu-1}}
			} \, \displaystyle{
				\sum_{t=1}^{N_{\nu-1}}
			} \, \varphi\Bigg( G(x^{\, \nu}, \xi^{\, t}), \displaystyle{
				\frac{1}{N_{\nu-1}}
			} \, \displaystyle{
				\sum_{s=1}^{N_{\nu-1}}
			} \, F(x^\nu, \eta^{\, s}) \, \Bigg) \, \Bigg) -   \Theta(x^\nu)\,  \Bigg|,\\[0.3in]
			e_{\nu,2} \, & \triangleq \, \displaystyle{
				\frac{N_{\nu-1} \, \left( \, N_\nu-N_{\nu-1} \,  \right)}{(N_\nu)^2}
			} \,
			\Bigg| \,
			\psi\Bigg( \displaystyle{
				\frac{1}{N_\nu-N_{\nu-1}}
			}  \displaystyle{
				\sum_{t=N_{\nu-1}+1}^{N_\nu}
			}   \varphi\Bigg( G(x^\nu,\xi^{\, t} \, ), \displaystyle{
				\frac{1}{N_{\nu-1}}
			} \, \displaystyle{
				\sum_{s=1}^{N_{\nu-1}}
			} \, F(x^\nu, \eta^{\, s}) \Bigg) \, \Bigg)  -\Theta(x^\nu) \,  \Bigg|,\\[0.3in]
			e_{\nu,3} \, & \triangleq \, \left( \, 1 - \displaystyle{
				\frac{N_{\nu-1}}{N_\nu}
			} \, \right) \, \Bigg|\,
			\psi\Bigg( \, \displaystyle{
				\frac{1}{N_\nu}
			} \, \displaystyle{
				\sum_{t=1}^{N_\nu}
			} \, \varphi\Bigg( G(x^\nu,\xi^{\, t} \, ), \displaystyle{
				\frac{1}{N_\nu-N_{\nu-1}}
			} \, \displaystyle{
				\sum_{s=N_{\nu-1}+1}^{N_\nu}
			} \, F(x^\nu, \eta^{\, s})\Bigg) \, \Bigg)
			-\Theta(x^\nu) \, \Bigg|.
			\end{array}  \]}
		Combining with \eqref{eq:update_rule_mm_descent} and taking conditional expectations on both sides, we derive that
		
		\begin{equation} \label{eq:key descent} 
			\mathbb{E}_{\nu}   \left[ \, \overline{\Theta}_{N_{\nu}}(x^{\nu+1})   \mid \mathcal{H}_{\nu-1}  \right]  +
			\displaystyle{
				\frac{1}{2\rho}
			} \,\mathbb{E}_{\nu}   \left[\, \norm{x^{\nu+1}-x^\nu}^2   \mid \mathcal{H}_{\nu-1}  \right]   \leq   \overline{\Theta}_{N_{\nu-1}}(x^\nu) +  \mathbb{E}_{\nu}   \left[ \,  e_{\nu,1} + e_{\nu,2} + e_{\nu,3} \mid  \mathcal H_{\nu-1}\, \right] .	\end{equation}	
		{By Proposition~\ref{prop:SAA_error_bound} above, it follows that
		there exists a positive constant $C_1$
		such that % when $N_\nu = \nu^\alpha$ for some $\alpha >1$,
	 	\begin{equation}
	 	\label{eq:sampling_error} \left\{\begin{array}{lll}
			\mathbb{E}_{\nu} \left[ \, e_{\nu,1} \, \right] & \leq & \displaystyle{
				\frac{\mbox{Lip}_{\psi} \,  C_1}{(N_{\nu-1})^{\, 1/2}}
			} \, \left( \, 1 - \displaystyle{
			\frac{(N_{\nu-1})^2}{(N_{\nu})^2}
		} \, \right) \, % \lesssim
			% = \, O\left( \displaystyle{ 	\frac{1}{(\nu-1)^{\alpha \beta+1}	} \, \right),
			 \\ [0.3in]
			\mathbb{E}_\nu \left[ \, e_{\nu,2} \, \right] & \leq & \displaystyle{
				\frac{\mbox{Lip}_{\psi} \, C_1}{(N_\nu-N_{\nu-1})^{\, 1/2}}
			} \, \displaystyle{
				\frac{N_{\nu-1}(N_\nu-N_{\nu-1})}{(N_\nu)^2}
			} \, % \lesssim O
			% = \, O\left( \, \displaystyle{ 	\frac{1}{(\nu-1)^{\alpha \beta -\beta +1}} } \, \right),
			\\ [0.3in]
			\mathbb{E}_\nu \left[ \, e_{\nu,3} \, \right] & \leq & \displaystyle{
				\frac{\mbox{Lip}_{\psi} \, C_1}{(N_\nu-N_{\nu-1})^{\, 1/2}}
			} \, \displaystyle{
				\frac{N_\nu - N_{\nu-1}}{N_\nu}
			}.  %  \lesssim
			% = \, O\left( \, \displaystyle{ \frac{1}{(\nu-1)^{\alpha \beta -\beta +1}}	} \, \right) .
			\end{array}\right. \end{equation}
}
		
\noindent {From Lemma \ref{lm:condition on sample sizes}, it follows that the summation
${\sum_{\nu=1}^{\infty}
} \, \mathbb{E} \, \Big[  \, e_{\nu,1} + e_{\nu,2} + e_{\nu,3}  \, \Big] $ is finite.
Thus we can derive that $ \sum_{\nu=1}^\infty \mathbb{E}_{\nu} \left[ e_{\nu,1} + e_{\nu,2} + e_{\nu,3}\, \mid \, \mathcal H_{\nu-1}\, \right]$
is finite with probability 1; this can be proved by contradiction as follows.  Specifically, suppose that
$\mathbb{P} \left( \, \sum_{\nu=1}^\infty \mathbb{E}_{\nu} \left[ e_{\nu,1} + e_{\nu,2} + e_{\nu,3}\, \mid \, \mathcal H_{\nu-1}\, \right]\, = \infty \right) >0 $.
Since $\{e_{\nu,1}\}$, $\{e_{\nu,2}\}$, $\{e_{\nu,3}\}$ are all nonnegative scalars, the infinite summation can be interchanged with the expectation operator and thus $
		  {	\sum_{\nu=1}^{\infty}
		} \, \mathbb{E}  \, \Big[  \,      e_{\nu,1} + e_{\nu,2} + e_{\nu,3}  \, \Big]  = \mathbb{E} \,\left[ \,  \sum_{\nu=1}^{\infty}
		 \,  \mathbb{E}_{\nu} \left[ \,     e_{\nu,1} + e_{\nu,2} + e_{\nu,3}  \, \mid \, \mathcal H_{\nu-1}\, \right] \, \right]$ is infinite, which  contradicts the previous finite expectation fact.}
		 \gap
		
{By Proposition \ref{prop:SAA_error_bound} (b), there exists $\bar \nu$ such that with probability 1,  $\inf_{x \in X} \, \overline \Theta_{N_\nu} (x) > - U$ for any $\nu \geq \bar \nu$. Thus, by adding $U$ on both sides of \eqref{eq:key descent}, we can  apply  the Robbins-Siegmund nonnegative almost supermartingale convergence result in Lemma \ref{lm:R-S-supermartingale}, and  deduce that $ {
			\lim_{\nu \to \infty}
		} \, \overline{\Theta}_{N_{\nu}}(x^{\nu+1})$ exists, and $ {	\sum_{\nu=1}^{\infty}} \, \mathbb{E}_{\nu}   \left[\, \norm{x^{\nu+1}-x^\nu}^2 \, \mid \mathcal{H}_{\nu-1}\, \right]$ is finite  with probability 1. The latter result yields that
	$
		  \mathbb{P}  \Big( \displaystyle{	\lim_{\nu \to \infty}	} \,      \mathbb{E}_{\nu}   \Big[\,\norm{x^{\nu+1}-x^\nu}^2 \, \mid \mathcal{H}_{\nu-1}\,\Big] = 0 \Big) =1 $ and thus $ \mathbb{E} \, \Big[ \, \displaystyle{\lim_{\nu \to \infty}} \,      \mathbb{E}_{\nu}   \Big[\,\norm{x^{\nu+1}-x^\nu}^2 \, \mid \mathcal{H}_{\nu-1}\,\Big] \, \Big] = 0.$ 
By the interchange of the limit and expectation operators holds under the compactness of $X$, we have $ \displaystyle{\lim_{\nu \to \infty}} \,      \mathbb{E}    \Big[\,\norm{x^{\nu+1}-x^\nu}^2 \,\Big]   =  \mathbb{E}    \Big[\,\displaystyle{\lim_{\nu \to \infty}} \,      \norm{x^{\nu+1}-x^\nu}^2 \,\Big]   = 0.$
	    Hence by the similar argument of contradiction,  we can derive that
		$\displaystyle{
			\lim_{\nu\to \infty}
		} \, \norm{x^{\nu+1}-x^{\nu}} = 0$ with probability~1.}
		For any accumulation point $x^{\infty}$ of the sequence $\{x^{\nu}\}$,
		let $\mathcal{K}$ be a subset of $\{1, 2, \ldots,\}$  such that $\displaystyle{
			\lim_{\nu (\in \mathcal{K}) \to +\infty}
		} \, x^{\nu} = x^{\infty}$,  which implies that with probability~1,
		$\displaystyle{
			\lim_{\nu (\in \mathcal{K}) \to +\infty}
		} \, x^{\nu+1} = x^{\infty}.$
		By Proposition \ref{prop:SAA_error_bound} and \cite[Proposition~5.1]{shapiro2009lectures},
		%and the fact that $\displaystyle{	\lim_{\nu \to \infty}	} \, \overline{\Theta}_{N_{\nu}}(x^{\nu+1})$ exists by Lemma \ref{lm:R-S-supermartingale},
		it follows that, $\displaystyle{
			\lim_{\nu \to \infty}
		} \, \overline{\Theta}_{N_{\nu}}(x^{\nu+1}) = \Theta(x^{\infty} )$,
		with probability 1.
		
		\gap
		
		Under the uniform outer semicontinuity of set-valued maps $\mathcal{F}$
		and $\mathcal{G}$ in (A3$_F$) and (A3$_G$), there exist functions
		$\wh{G}^{\, t, \infty} (\bullet,\xi^{\,t};x^{\, \infty}) \in {\cal G}(x^{\, \infty}, \xi^{\,t})$,
		$\wh{F}^{\,s, \infty}(\bullet,\eta^{\,s};x^{\, \infty}) \in {\cal F}(x^{\, \infty}, \eta^{\,s})$, and
		a nonnegative integrable function $L(\xi)$ such that for all $x \in X$ and any
		$\varepsilon > 0$, there exist a subset  of $\mathcal{K}$ (which is assumed to be
		$\mathcal{K}$ without loss of generality) and  an integer
		$\bar{\nu}(x,\varepsilon)$,  such that  for any
		$\nu (\in \mathcal{K}) \geq \bar{\nu}(x,\varepsilon)$ and all $t, s \in[N_\nu]$,
		\begin{equation} \label{eq:uniform continuity in proof}
		\norm{\, \wh{G}^{\,t}(x,\xi^{\,t};x^{\, \nu}) -
			\wh{G}^{\, t, \infty} (x,\xi^{\,t};x^{\, \infty}) \, }    \leq \, \varepsilon \, L(\xi^{\,t}), \quad \norm{ \, \wh{F}^{\,s}(x,\eta^{\,s};x^{\, \nu}) -
			\wh{F}^{\, s, \infty} (x,\eta^{\,s};x^{\, \infty}) \,}  \leq \, \varepsilon \, L(\eta^{\,s}).
		\end{equation}
		Let
		\begin{equation} \label{eq:limiting surrogate}
		\wh{V}_{\infty;N_{\nu}}(x; x^{\infty})  \triangleq  \psi\Bigg( \, \displaystyle{
			\frac{1}{N_\nu}
		} \, \displaystyle{
			\sum_{t=1}^{N_\nu}
		} \, \varphi\Bigg( \wh{G}^{\, t, \infty} (x,\xi^{\, t};x^{\infty}), \, \displaystyle{
			\frac{1}{N_\nu}
		} \, \displaystyle{
			\sum_{s=1}^{N_\nu}
		} \, \wh{F}^{\, s, \infty} (x,\eta^{\, s}; x^{\infty}) \, \Bigg) \, \Bigg).
		\end{equation}
		Let $x \in X$ be fixed but arbitrary.  By the Lipschitz continuity of $\psi$ and
		$\varphi$,
		it follows that for any $\nu (\in \mathcal{K}) \geq \bar{\nu}(x,\varepsilon)$, we have
		\[
		\begin{array}{ll}
		\left| \, \wh{V}_{N_{\nu}}(x; x^{\nu})  -
		\wh{V}_{\infty;N_{\nu}}(x;x^{\infty}) \, \right|
		\, & \leq \, \mbox{Lip}_{\psi} \,  \mbox{Lip}_{\varphi} \, \displaystyle{
			\frac{\varepsilon}{N_{\nu}}
		} \, \displaystyle{
			\sum_{t=1}^{N_{\nu}}
		} \,\Bigg[ \, L(\xi^{\, t}) + \displaystyle{
			\frac{1}{N_{\nu}}
		} \, \displaystyle{
			\sum_{s=1}^{N_{\nu}}
		} \, L(\eta^{\, s}) \, \Bigg] \\[0.3in]
		& = \mbox{Lip}_{\psi} \,  \mbox{Lip}_{\varphi} \, \displaystyle{
			\frac{\varepsilon}{N_{\nu}}
		} \, \Bigg[ \,  \displaystyle{
			\sum_{t=1}^{N_{\nu}}
		} \,\, L(\xi^{\, t}) + \displaystyle{
			\sum_{s=1}^{N_{\nu}}
		} \, L(\eta^{\, s})  \, \Bigg] .
		\end{array}
		\]
		% L_2 \sqrt{l_1 + L_2}  \,  \varepsilon
		By the law of large numbers, it follows that $
		\displaystyle{
			\lim_{\nu \to \infty}
		} \, \displaystyle{
			\frac{1}{N_{\nu}}
		} \, \displaystyle{
			\sum_{t=1}^{N_{\nu}}
		} \, L(\xi^{\, t}) = \lim_{\nu \to \infty}
		\, \displaystyle{
			\frac{1}{N_{\nu}}
		} \, \displaystyle{
			\sum_{s=1}^{N_{\nu}}
		} \, L(\eta^{\, s}) \, = \mathbb{E}[ \, L(\tilde{\xi} )\, ] \, < \, \infty$ almost surely.
		Since $\varepsilon$ is arbitrary, it
		follows that for the infinite subset $\mathcal{K}$,
		with probability one, $
			\displaystyle{
				\lim_{\nu (\in \mathcal{K}) \to \infty}
			} \, \left| \, \wh{V}_{N_{\nu}}(x; x^{\nu})  -
			\wh{V}_{\infty;N_{\nu}}(x;x^{\infty}) \, \right| \, = \, 0. $
		% \textcolor{red}{Without loss of generality, under (E1$_\infty$) and (E2$_\infty$),
		Under condition (E$_{\infty}$), the strong law of large numbers for the sum of i.i.d.\ Banach space valued random sets
		\cite[Theorem~2.1]{puri1983strong} yields the existence of
		functions $\wh{G}^{\infty}(\bullet,\xi; x^\infty) \in \mathcal{G}(x^\infty, \xi)$
		and $\wh{F}^{\infty}(\bullet,\xi; x^\infty) \in \mathcal{F}(x^\infty, \xi)$ such that
        $\wh{G}^{\infty}(x,\tilde{\xi}; x^\infty)$ and $\wh{F}^{\infty}(x,\tilde{\xi}; x^\infty)$ are $\mathbb{P}$-integrable for any $x\in X$ and
        that
		\[ \begin{array}{l}
		f^{\infty}(x)  \triangleq  \mathbb{E}\,  \left[ \,\wh{F}^{\infty}(x, \tilde{\xi}; x^\infty) \,  \right] =\displaystyle{\lim_{\nu (\in \mathcal{K})  \to \infty}
			\frac{1}{N_\nu} \sum_{s=1}^{N_\nu}} \, \wh{F}^{\, s, \infty}(x, \eta^{\, s}; \, x^{\infty}), \\[0.2in] 	\varphi^{\infty}(x)  \triangleq \mathbb{E} \, \Big[ \varphi\Big(  \,\wh{G}^{\infty}(x, \tilde{\xi}; x^\infty), \,  \mathbb{E}\,  \Big[ \,\wh{F}^{\infty}(x, \tilde{\xi}; x^\infty) \,  \Big]\,\Big) \, \Big] = \displaystyle{\lim_{\nu (\in \mathcal{K})  \to \infty}
			\frac{1}{N_\nu} \sum_{t=1}^{N_\nu}} \, \varphi\Big( \, \wh{G}^{\, t, \infty}(x, \xi^{\, t}; \, x^{\infty}), \,
		f^{\infty}(x) \Big).
		\end{array}
		\]
 	Thus, $\wh{V}_{\infty}(x; x^{\infty}) \triangleq \psi\left(  \varphi^{\infty}(x) \, \right) $ is well defined.
		By the Lipschitz continuity of $\psi$, we have
	   \[
			\begin{array}{l}
			\Big| \, \wh{V}_{\infty; N_\nu} (x; x^\infty) - \wh{V}_{\infty}(x; x^\infty) \, \Big| \, \leq \mbox{Lip}_{\psi} \, \Bigg\| \, \, \displaystyle{
					\frac{1}{N_\nu}
				} \, \displaystyle{
					\sum_{t=1}^{N_\nu}
				} \, \varphi\Bigg( \wh{G}^{\,t, \infty} (x,\xi^{\, t};x^{\infty}), \, \displaystyle{
					\frac{1}{N_\nu}
				} \, \displaystyle{
					\sum_{s=1}^{N_\nu}
				} \, \wh{F}^{\,s, \infty} (x,\eta^{\, s}; x^{\infty}) \, \Bigg) - \varphi^{\infty}(x) \, \Bigg\| \, \\[0.3in]
			\leq  \mbox{Lip}_{\psi} \, \Bigg\| \, \, \displaystyle{
					\frac{1}{N_\nu}
				} \, \displaystyle{
					\sum_{t=1}^{N_\nu}
				} \, \varphi\Bigg( \wh{G}^{\,t, \infty} (x,\xi^{\, t};x^{\infty}), \, \displaystyle{
					\frac{1}{N_\nu}
				} \, \displaystyle{
					\sum_{s=1}^{N_\nu}
				} \, \wh{F}^{\,s, \infty} (x,\eta^{\, s}; x^{\infty}) \, \Bigg) - \frac{1}{N_\nu} \sum_{t=1}^{N_\nu} \, \varphi\left( \, \wh{G}^{\, t,   \infty}(x, \xi^{\, t}; \, x^{\infty}), \, f^{\infty}(x) \right) \Bigg\| \\[0.3in]
			\quad  + \, \mbox{Lip}_{\psi} \, \Bigg\| \, \displaystyle{\frac{1}{N_\nu} \sum_{t=1}^{N_\nu}}
				\, \varphi\left( \, \wh{G}^{\,t, \infty}(x, \xi^{\, t}; \, x^{\infty}), \, f^{\infty}(x) \right)  -
				\varphi^{\infty}(x) \, \Bigg\|\\[0.3in]
			\leq \, \mbox{Lip}_{\psi} \left[ \,  \mbox{Lip}_{\varphi}  \, \Bigg\| \, \displaystyle{ \frac{1}{N_\nu} \,
					\sum_{s=1}^{N_\nu}
				} \, \wh{F}^{\,s, \infty} (x,\eta^{\, s}; x^{\infty})  - f^{\infty} (x)\, \Bigg\|+ \Bigg\| \, \displaystyle{\frac{1}{N_\nu} \sum_{t=1}^{N_\nu}} \,
\varphi\left( \, \wh{G}^{\,t, \infty}(x, \xi^{\, t}; \, x^{\infty}), \, f^{\infty}(x) \right)  - \varphi^{\infty}(x) \, \Bigg\| \, \, \right].
			\end{array}
			\]
Hence, we have $\displaystyle{\lim_{\nu ( \in \mathcal{K}) \to \infty}} \wh{V}_{\infty; N_\nu}(x; \, x^{\infty}) = \wh{V}_{\infty}(x; x^{\infty})$.
It follows that with probability one,
		\[
		\displaystyle{
			\lim_{\nu (\in \mathcal{K}) \to \infty}
		} \, \wh{V}_{N_{\nu}}(x; x^{\nu})\\[0.1in]
		=  \displaystyle{
			\lim_{\nu (\in \mathcal{K})\to  \infty}
		}  \left( \wh{V}_{N^{\nu}}(x; x^{\nu}) -
		\wh{V}_{\infty;N_{\nu}}(x; x^{\infty}) \right)  + \displaystyle{
			\lim_{\nu (\in \mathcal{K})\to \infty}
		}   \wh{V}_{\infty;N_{\nu}}(x;x^{\infty}) =   \wh{V}_{\infty}(x; x^{\infty}).
		\]
		Let  $
		\wt{V}_{N_{\nu}}^{\rho}(\bullet;x^{\, \nu}) \, \triangleq \,
		\wh{V}_{N_{\nu}}(\bullet;x^{\, \nu}) +
		\displaystyle{\frac{1}{2 \, \rho}}
		\, \| \, \bullet - x^{\, \nu} \, \|_2^2$. From the optimality of $x^{\nu+1}$, it follows that
		\[
		\overline{\Theta}_{N_{\nu}}(x^{\nu+1}) + \displaystyle{
			\frac{1}{2\rho}
		}  \norm{x^{\nu+1}-x^{\nu}}^2 \leq
		\wt{V}_{N_{\nu}}^{\rho}(x^{\nu+1}; x^{\nu})
		% + \displaystyle{
		% \frac{1}{2\rho_{\nu}}
		% } \, \norm{x^{\nu+1}-x^{\nu}}^2
		\, \leq \, \wh{V}_{N_{\nu}}(x; x^{\nu}) + \displaystyle{
			\frac{1}{2\rho}
		} \, \norm{x-x^{\nu}}^2, \ \forall \, x \, \in \, X.
		\]
		Letting $\nu (\in \mathcal{K}) \to \infty$ in the above inequality, we derive that
		with probability 1, for all $x \in X$,
		\begin{equation} \label{eq:last step uses convexity}
		\Theta(x^{\infty}) \, = \, \wh{V}_{\infty}(x^{\infty};x^{\infty})
		\, \leq \, \wh{V}_{\infty}(x; x^{\infty} )+ \displaystyle{
			\frac{1}{2 \rho}
		} \, \norm{x -x^{\infty}}^2.
		\end{equation}
		Since the right-hand side is a convex function in $x$, the stated result of this theorem follows.  
		
		% This is equivelant to $0 \in \partial V^*(x^{\infty}; x^{\infty}) +
		% \mathcal{N}(x^{\infty}; X)$.
	\end{proof}
	\gap
	
	We make an important observation about the proof of Theorem~\ref{thm:convergence_smm}.
	Namely, the convexity of
	the surrogate functions $\wh{G}(\bullet,\xi;x^{\nu})$ and $\wh{F}(\bullet,\xi;x^{\nu})$
	is not used until the last step when we drop the proximal term in the right-hand
	side of (\ref{eq:last step uses convexity}).  Of course, this convexity is needed to
	ensure the convexity of the function $\wh{V}_{N_{\nu}}(\bullet;x^{\nu})$ for the
	practical implementability of the SMM algorithm.  Nevertheless, in situations when the latter
	implementation is not an issue, as in the case discussed in
	Subsection~\ref{subsec:nonsmooth concave},   the convexity requirement of these
	surrogate functions can be removed with a suitable modification of the conclusion
	of Theorem~\ref{thm:convergence_smm}.  In what follows, we state a variant of the
	theorem without assuming the convexity of
	$\wh{G}(\bullet,\xi;x^{\nu})$ and $\wh{F}(\bullet,\xi;x^{\nu})$.
	
	\begin{corollary} \label{co:convergence without convexity} \rm
		Assume the same setting of Theorem~\ref{thm:convergence_smm} but without the convexity
		of $\wh{G}(\bullet,\xi;x^{\nu})$ and $\wh{F}(\bullet,\xi;x^{\nu})$.  Provided that the family of
		surrogate functions $\mathcal{G}(x^{  \infty}, \xi)$ and
		$\mathcal{F}(x^{  \infty}, \xi)$ satisfy condition   %(A4$_{\infty}$)
		 (E$_{\infty}$)
		at the accumulation point $x^{\infty}$
		and that
		$\wh{V}(\bullet;x^{\infty})$ is additionally directionally differentiable
		at $x^{\infty}$, it holds that
		$\wh{V}(\bullet;x^{\infty})^{\, \prime}(x^{\infty};x - x^{\infty}) \geq 0$ for all
		$x \in X$; that is $x^{\infty}$ is a directional stationary solution of
		$\wh{V}(\bullet;x^{\infty})$ on $X$.
	\end{corollary}
	\begin{proof} From (\ref{eq:last step uses convexity}), we have   that
		$x^{\infty} = \displaystyle{
			\operatornamewithlimits{\mbox{argmin}}_{x \in X}
		} \, \left[ \, \wh{V}_{\infty}(x; x^{\infty} )+ \displaystyle{
			\frac{1}{2 \rho}
		} \, \norm{x -x^{\infty}}^2 \, \right]$.  This readily implies that
		$x^{\infty}$ is a directional stationary solution of
		$\wh{V}_{\infty}(\bullet;x^{\infty})$ on $X$.
  
 \end{proof}
	
	\gap
	
Several remarks regarding the above theorem are presented in order.
\gap

{\emph{\bf Remark 1.} The proof of Theorem~\ref{thm:convergence_smm} can easily be extended to an
inexact version of the SMM Algorithm.	In fact, instead of the last inequality in
(\ref{eq:update_rule_mm_descent}), the inexact rule would require that
\[
\wh{V}_{N_{\nu}}(x^{\nu+1}; x^{\nu}) + \displaystyle{
\frac{1}{2 \rho}
} \, \norm{x^{\nu+1}-x^{\, \nu}}^2  \, \leq \, \wh{V}_{N_{\nu}}(x^{\, \nu}; x^{\, \nu}) + \delta_{\nu}.
\]
By suitably restricting the suboptimality tolerance $\delta_{\nu} > 0$, the inequality
(\ref{eq:key descent}) becomes
\[
\begin{array}{l}
\mathbb{E}_{\nu} \left[ \, \overline{\Theta}_{N_{\nu}}(x^{\nu+1}) \, \mid \mathcal{H}_{\nu-1}\, \right] + \displaystyle{
\frac{1}{2\rho}
} \,\mathbb{E}_{\nu} \left[\, \norm{x^{\nu+1}-x^\nu}^2   \mid \mathcal{H}_{\nu-1}  \right] \\ [0.15in]
\epc \leq \, \overline{\Theta}_{N_{\nu-1}}(x^\nu) + \mathbb{E}_{\nu} \left[ \,  e_{\nu,1} + e_{\nu,2} + e_{\nu,3}\, \mid \, \mathcal H_{\nu-1}\, \right]
+ \delta_{\nu}.
\end{array}
\]
% \[
% \overline{\Theta}_{N_{\nu}}(x^{\nu+1}) + \displaystyle{
% \frac{1}{2\rho}
% } \, \norm{x^{\nu+1}-x^\nu}^2 \, \leq \,
% \overline{\Theta}_{N_{\nu-1}}(x^\nu) +  e_{\nu,1} + e_{\nu,2} + e_{\nu,3} + \delta_{\nu}.
% \]
Thus, if $\displaystyle{
\sum_{\nu=1}^{\infty}
} \, \delta_{\nu} < \infty$, the remaining proof continuous to hold with a minor modification to the last step of
proving the inequality (\ref{eq:last step uses convexity}) to accommodate the inexactness of
$x^{\nu+1}$.}
\gap

\emph{\bf Remark 2.} Theorem~\ref{thm:convergence_smm} shows that  with probability 1, any accumulation
	point of the sequence derived from the stochastic majorization-minimization algorithm
	has a certain fixed-point property relative to a well-defined convex surrogate function of the
	objective function $\Theta$ in the original compound SP (\ref{eq:compound_composite_SP}).
	In the present statement, this conclusion is fairly abstract as very little
	details of the
	surrogate function are provided except for the assumptions (A3$_G$) and (A3$_F$).
	In particular, an understanding of how the fixed-point conclusion is related to some
	kind of stationarity of the original problem (\ref{eq:compound_composite_SP}) is
	warranted, which will be addressed in the next section.
	% We omit the treatment of the issue of sequential convergence of the algorithm, but this could
	% be expected under the Kurdyka–-Lojasiewicz (KL) property.
	
	\gap
	
\emph{\bf Remark 3.} When the functions $\psi(t) = \varphi(t) =  t$ (so that $\ell_G = 1$) for any
	scalar $t$, and the random function $F$ is absent ($\ell_F = 0$), the compound SP
	\eqref{eq:compound_composite_SP} reduces to the classical stochastic program.
	In a recent paper \cite{an2019stochastic}, the authors have studied the subsequential
	convergence of the stochastic difference-of-convex (dc) algorithm for solving such
	problems when the random function $G(\bullet\,,\widetilde{\xi})$ is further assumed
	to be a dc function. Unlike our algorithm where the sample size $N_\nu$ obey the conditions
in Lemma~\ref{lm:condition on sample sizes},
% grows at the rate of $\nu^{\,\alpha}$ ($\alpha>1$),
the number of samples used per iteration by the
	proposed stochastic dc algorithm in \cite{an2019stochastic} (for the special case)
	grows linearly. The higher
	sample complexity required by our algorithm is mainly due to the nature of the compound
	SP--the composition of a function and an expectation; specifically, in order for the
	summation of the errors $  {
		\sum_{\nu=1}^{\infty}
	} \, \mathbb{E} \left[ \, e_{\nu,1} + e_{\nu,2} + e_{\nu,3} \, \right] $ to be finite in the proof of
	Theorem \ref{thm:convergence_smm}, we need
	$N_\nu - N_{\nu-1}\to \infty$ as $\nu\to \infty$.
	
	\gap

	\begin{comment}followed by the treatment
	of another issue that has to do with designing stopping rules for the algorithm.
	We omit the
	treatment of the issue of sequential convergence of the algorithm as we deem this to
	be less novel as our study of error-bound based stopping rules (and also to avoid
	lengthening the paper beyond an acceptable limit).
	\end{comment}
	
	\section{Post-convergence Analysis: Stationarity} \label{sec:stationarity}
	
	We discuss the connection between the fixed-point property of a limit point of the
	SMM algorithm and some stationarity properties of (\ref{eq:compound_composite_SP}) for three structural cases of the component random functions of $G$
	and $F$: (i) the smooth case, (ii) the case of a dc function with a smooth concave
	part, and (iii) the dc case with a ``max-smooth'' concave part.  The terminology
	``max-smooth'' will be clear from the expression (\ref{eq:max dc}).  In the first two cases, each surrogation family $\mathcal{G}(x,\xi)$ and $\mathcal{F}(x, \xi)$
has only a  single member, so the validity of assumption  (E$_\infty$) is trivial. In Section~\ref{subsec:nonsmooth concave}, these assumptions are verified
for the dc case with nonsmooth concave part.  In all three cases,
	the objective function $\Theta$ in (\ref{eq:compound_composite_SP}) is not convex.
	
	\subsection{The smooth case}
	
	Suppose that each component function $F_j(\bullet,\xi)$ and
	$G_i(\bullet,\xi)$ are smooth
	functions with the Lipschitz gradient modulus $\kappa$ uniform for all $\xi \in \Xi$.
	Hence, for any $x^{\, \prime} \in X$, we may take their surrogate functions as
	\begin{equation}
	\label{eq:smooth_surrogation}
	\wh{F}_j(x,\xi;x^{\, \prime}) \, = \, F_j(x,\xi) + \displaystyle{
		\frac{\kappa}{2}
	} \, \norm{x- x^{\, \prime}}^2 \epc \mbox{and} \epc
	\wh{G}_i(x,\xi; x^{\, \prime}) \, = \, G_i(x,\xi) + \displaystyle{
		\frac{\kappa}{2}
	} \, \norm{x- x^{\, \prime}}^2.
	\end{equation}
	Thus the families ${\cal F}(x^{\, \prime}, \xi)$ and ${\cal G}(x^{\, \prime}, \xi)$ are
	singletons consisting of the single vector functions
	$\wh{F}(\bullet,\xi;x^{\, \prime}) \triangleq
	\left( \, \wh{F}_j(\bullet,\xi;x^{\, \prime}) \, \right)_{j=1}^{\ell_F}$
	and $\wh{G}(\bullet,\xi;x^{\, \prime}) \triangleq
	\left( \, \wh{G}_i(\bullet,\xi;x^{\, \prime}) \, \right)_{i=1}^{\ell_G}$,
	respectively, with their components
	$\wh{F}_j(\bullet,\xi;x^{\, \prime})$
	and $\wh{G}_i(\bullet,\xi;x^{\, \prime})$ being convex functions.  It is easy to
	verify conditions
	(A3$_F$) and (A3$_G$) by the smoothness assumptions of $F(\bullet,\xi)$ and
	$G(\bullet,\xi)$ and the Lipschitz continuity of their gradients.  At an accumulation
	point $x^{\infty}$ of the sequence of iterates $\{ x^{\nu} \}$ generated by the SMM algorithm,
	the limiting function $\wh{V}_{\infty}(\bullet;x^{\infty})$ is given by (cf.\
	(\ref{eq:limiting surrogate})):
	\[
	\wh{V}_{\infty}(x; x^{\infty}) \, \triangleq \, \psi\Bigg( \mathbb{E}\Bigg[ \,
	\varphi\Bigg( \, G(x,\xi) + \displaystyle{
		\frac{\kappa}{2}
	} \, \norm{x- x^{\, \infty}}^2, \,
	\mathbb{E}\left[ \, F(x,\xi) + \displaystyle{
		\frac{\kappa}{2}
	} \, \norm{x- x^{\, \infty}}^2 \, \right] \,
	\Bigg) \, \Bigg] \, \Bigg).
	\]
	Then under suitable conditions that ensure the directional differentiability of the
	objective function $\Theta$ in (\ref{eq:compound_composite_SP}) and the
	interchangeability of the directional derivatives with expectations (see \cite[Chapter~7.2.4]{shapiro2009lectures}), Theorem \ref{thm:convergence_smm} yields that any accumulation point $x^{\infty}$
	of the sequence generated by the SMM algorithm is a directional stationary point of
	\eqref{eq:compound_composite_SP} in the sense that
	$\Theta^{\, \prime}(x^{\infty};x - x^{\infty}) \geq 0$ for all
	$x \in X$, with probability 1.  As noted in \cite{PangRazaAlvarado18}, directional
	stationarity is a stationarity concept that is the strongest among all stationarity
	properties of a convex constrained optimization problem with a directionally
	differentiable objective, which the function $\Theta$ is.
	
	\subsection{The dc case with a smooth concave part}
	\label{subsec:smooth concave}
	The connection between the fixed-point property and the stationary property
	in the dc case is less straightforward.  It depends on the differentiability of
	the concave part in a dc decomposition of the functions $G_i$ and $F_j$.
	Suppose each $G_i(\bullet,\xi)$ and $F_j(\bullet,\xi)$ are dc functions given by
	\begin{equation}
	\label{eq:dc_smooth}
	G_i(x,\xi) \, = \, g_i^G(x,\xi) - h_i^G(x,\xi), \epc\mbox{and} \epc F_j(x,\xi) \, = \, g_j^F(x,\xi) - h_j^F(x,\xi),
	\end{equation}
	where $g_i^G(\bullet,\xi)$, $h_i^G(\bullet,\xi)$, $g_j^F(\bullet,\xi)$ and
	$h_j^F(\bullet,\xi)$ satisfy the assumption (DC1) in Subsection~\ref{subsec:assumptions},
	with $h_i^G(\bullet, \xi)$ and $h_j^F(\bullet, \xi)$ being additionally differentiable
	with Lipschitz gradient moduli independent of $\xi$.  The family
	${\cal G}(x^{\, \prime}, \xi)$ (and similarly,
	${\cal F}(x^{\, \prime}, \xi)$) is given by (\ref{eq:subgradient linearization}) with the single  element.
	% In this discussion, we rely on the references \cite{PangRazaAlvarado18},
	% \cite{qi2019statistical} where notions of directional stationarity for dc
	% programs are well explained and their connections with more
	% classical notions of stationarity, such as criticality and Clarke stationarity,
	% are made.
	%
	% \gap
	%
	Given a vector
	$x^{\infty}$, the family $\wh{\cal V}(x^{\infty})$ has only one element function
	$\wh{V}(x;x^{\infty})$ given by
	\[
	\wh{V}(x;x^{\infty}) \, \triangleq \, \psi\Bigg( \, \mathbb{E}\Bigg[ \, \varphi\Big( \,
	\wh{G}(x,\tilde{\xi};x^{\infty}), \, \mathbb{E}\left[ \,
	\wh{F}(x,\tilde{\xi};x^{\infty}) \, \right] \, \Big) \, \Bigg] \, \Bigg),
	\]
	where $\wh{G}(x,\xi;x^{\infty}) \triangleq
	\left( \, \wh{G}_i(x,\xi;x^{\infty}) \, \right)_{i=1}^{\ell_G}$ and
	$\wh{F}(x,\xi;x^{\infty}) \triangleq
	\left( \, \wh{F}_j(x,\xi;x^{\infty}) \, \right)_{j=1}^{\ell_F}$ with each
	\begin{equation} \label{eq:linearized dc}
	\left\{\begin{array}{llll}
	\wh{G}_i(x,\xi;x^{\infty}) & \triangleq &
	g_i^G(x,\xi) - h_i^G(x^{\infty},\xi) -  \nabla_x h_i^G(x^{\infty},\xi)^{\top}(
	x - x^{\infty} ),
	& i \, = \, 1, \cdots, \ell_G, \\ [0.1in]
	\wh{F}_j(x,\xi;x^{\infty}) & \triangleq &
	g_j^F(x,\xi) - h_j^F(x^{\infty},\xi) - \nabla_x h_j^F(x^{\infty},\xi)^{\top}(
	x - x^{\infty} ),
	& j \, = \, 1, \cdots, \ell_F.
	\end{array} \right. \end{equation}
	Hence, if $x^{\infty}$
	is an accumulation point of the sequence of iterates $\{ x^{\nu} \}$ produced by the
	SMM algorithm, then  $x^{\infty}$ is
	a global minimizer of the convex compound SP, $ \displaystyle{
		\operatornamewithlimits{\mbox{minimize}}_{x \in X}
	} \,\,  \wh{V}(x; x^\infty)
$.
	Under conditions that ensure the interchangeability of expectations with directional
	derivatives \cite[Theorem~7.44]{shapiro2009lectures}, it can be shown that this limit
	point is a directional stationary point of original compound SP
	\eqref{eq:compound_composite_SP} almost surely.  Again, this is the best one can hope for in a general
	dc stochastic program.
	
	\subsection{The dc case with a max-smooth concave part} \label{subsec:nonsmooth concave}
	
When the concave parts in the dc functions $G_i(x, \xi)$ and $F_i(x, \xi)$ are nonsmooth functions, surrogation function families  $\mathcal{G}(x^{\, \prime}, \xi)$ and $\mathcal{F}(x^{\, \prime}, \xi)$ are constructed  in \eqref{eq:subgradient linearization} by linearizing the concave parts using subgradients. By assuming the following uniform boundedness property, the assumptions  (A4$_\infty$) and (E$_\infty$) made at the limit point $x^{\, \infty}$ in Theorem \ref{thm:convergence_smm}  hold.
\gap

\textbf{(DC2)} $\sup \, \left\{ \,  \max\left\{ \norm{G(x^{\, \infty}, \xi)}, \norm{F(x^{\, \infty}, \xi)} \right\}: \xi \in \Xi \, \right\} < \infty$,
\gap

 \quad \qquad  \, $\sup \, \left\{ \,  \norm{ \, a_i (x^{\, \infty} , \xi ) \,}:  a_i(x^{\, \infty} , \xi ) \in \partial_x h_i^G(x^{\, \infty}, \xi) ,   \xi \in \Xi , i \in [\ell_G] \right\} < \infty$,
 \gap

\quad \qquad  \, $\sup \, \left\{ \,  \norm{ \, b_j (x^{\, \infty} , \xi ) \, }:    b_j(x^{\, \infty}, \xi )  \in \partial_x h_j^F(x^{\, \infty}, \xi) ,   \xi \in \Xi ,  j \in [\ell_F]\, \right\} < \infty$.
\gap

\noindent Under (DC2), it is straightforward to see (A4$_\infty$) is true. For the assumption (E$_\infty$), the random set
	\[\Phi_{\mathcal{G}}(x, x^{\, \infty}, \xi) = \left\{ \left( g_i^G(x^{\, \infty}, \xi) - h_i^G(x^{\, \infty}, \xi) - a^i(x^{\, \infty}, \xi)^\top (x-x^{\, \infty}) \right)_{ i \in [\ell_G]} : a^i(x^{\, \infty}, \xi) \in \partial h_i^G(x^{\, \infty}, \xi) \right\}.\]
	From Theorem 7.47 in \cite{shapiro2009lectures}, we have $\partial\, \mathbb{E}[h^G_i(x^{\, \infty}, \tilde{\xi})] = \mathbb{E}[\partial\,  h^G_i(x^{\, \infty}, \tilde{\xi})] $,  which is a nonempty, closed, convex set. Hence $\mathbb{E} \, [\, \Phi_{\mathcal{G}}(x, x^{\, \infty}, \tilde{\xi}) \, ]$ and $\mathbb{E} \, [\, \Phi_{\mathcal{F}}(x, x^{\, \infty}, \tilde{\xi}) \, ]$ are nonempty, closed and convex sets for any $x\in X$.
	\gap
	
	 As it presently stands, the SMM algorithm falls short of computing a
	 directional stationary point of the problem (\ref{eq:compound_composite_SP}) when
	 the concave part in the dc function is nonsmooth.
	  In the case, the
	fixed-point property of an accumulation point $x^{\infty}$ in
	Theorem~\ref{thm:convergence_smm} asserts that there exist
	subgradients
	$a^{\, G,i}(x^{\infty},\xi) \in \partial_x h_i^{\, G}(x^{\infty},\xi)$ and
	$a^{\, F,i}(x^{\infty},\xi) \in \partial_x h_i^{\, F}(x^{\infty},\xi)$ such that  $x^\infty$ is a global minimizer of the convex compound SP $\displaystyle{
		\operatornamewithlimits{\mbox{minimize}}_{x \in X}
	} \,\,  \wh{V}(x; x^\infty)  \triangleq \, \psi\Big( \, \mathbb{E}\Big[ \, \varphi\Big( \,
\wh{G}(x,\tilde{\xi};x^{\infty}), \, \mathbb{E}\left[ \,
\wh{F}(x,\tilde{\xi};x^{\infty})   \right]  \Big)  \Big]  \Big),$ the same conclusion as in the
last case,  except the functions $\wh{G}_i(x,\xi;x^{\infty})$ and
	$\wh{F}_i(x,\xi;x^{\infty})$ in (\ref{eq:linearized dc}) are replaced by convex surrogate functions with the subgradients
	$a^{\, G,i}(x^{\infty},\xi) $ and
	$a^{\, F,j}(x^{\infty},\xi)  $, respectively.
	Extending the notion of criticality in
	deterministic dc programming, we call this fixed-point property of $x^{\infty}$,
	i.e., $x^{\infty} \in \displaystyle{
		\operatornamewithlimits{\mbox{argmin}}
	} \, \{ \, \wh{V}(x; x^\infty): {x\in X}\} $, the
	\emph{compound criticality} of the compound SP (\ref{eq:compound_composite_SP}), and
	$x^{\infty}$ a \emph{compound critical point} of the same problem.
	
	\gap
	
	It is worth
	mentioning that the directional stationarity of the problem
	(\ref{eq:compound_composite_SP}) is a stronger concept than compound
	criticality, in the sense that a point satisfying the former property must also satisfy
	the latter one.  The reverse implication is not true, which can be seen from the example
	given in \cite[Section 8]{qi2019statistical}. This cited example is a special case of
	(\ref{eq:compound_composite_SP}) with the univariate functions
	$\psi(t) = t$, $\varphi(t) = t^2$,
   the multivariate functions	$G(x,\xi) = | x^\top \xi_{[1:n]}|- \xi_{n+1}$
	and $F(x,\xi) = 0$ for any scalar $t$ and $n$-dimensional vector $x$,
	where $\xi_{[1:n]}$ refers to the first $n$ components of a
	$(n+1)$-dimensional random vector $\xi$ and $\xi_{n+1}$ refers to its
	last component. Based on \cite[Propositions 8.1 and 8.2]{qi2019statistical}, the
	point $x=0$ is a Clarke stationary point (thus a compound critical point), but not
	a directional stationary point.
	
	\gap
	
	If the functions
	$h_i^G(\bullet)$ and $h_j^F(\bullet)$ are deterministic pointwise maxima
	of finitely many
	convex smooth functions (while $g_i^G(\bullet,\xi)$ and $g_j^F(\bullet,\xi)$
	remain random), then we may adopt the enhancement idea for deterministic
	dc programs as detailed in \cite{PangRazaAlvarado18} and develop an
	enhanced version of the SMM algorithm that may be shown to converge subsequentially to a
	directional stationary solution of (\ref{eq:compound_composite_SP}).  We provide a
	bit more details on this class of specially structured dc problems and skip the full
	development.
	The interesting feature of  the enhanced version of SMM algorithm is that
	the surrogate functions $V(\bullet,x^{\, \prime})$ is not convex, yet
	the subproblems can be
	decomposed into finitely many convex programs.  Specifically, let
	\begin{equation} \label{eq:max dc}
	G_i(x,\xi) \, = \, g_i^G(x,\xi) - \displaystyle{
		\max_{1 \leq k \leq K_G}
	} \, h_{i, k}^G(x) \epc \mbox{and} \epc
	F_j(x,\xi) \, = \, g_j^F(x,\xi) - \displaystyle{
		\max_{1 \leq k \leq K_F}
	} \, h_{j, k}^F(x),
	\end{equation}
	where $g_i^G(\bullet,\xi)$, $h_{i, k}^G$ $g_j^F(\bullet,\xi)$ and $h_{j, k}^F$ are
	all convex functions with $h_{i, k}^G$ and $h_{j, k}^F$ being additionally
	differentiable with Lipschitz gradients.
	The pointwise maxima in the above functions explains the terminology ``max-smooth''
	in the heading of this subsection.
	For any $\varepsilon > 0$ and any $x \in X$, let
	{
		$${\cal A}_i^{G;\varepsilon}(x)  \triangleq  \left\{ \, k \, \mid \,
		h_{i, k}^G(x) \, \geq \, \displaystyle{
			\max_{1 \leq k^{\, \prime} \leq K_G}
		} \, h_{i, k^{\prime}}^G(x) - \varepsilon \, \right\}, \; 
		{\cal A}_j^{F;\varepsilon}(x)  \triangleq  \left\{ \, k \, \mid \,
		h_{j, k}^F(x) \, \geq \, \displaystyle{
			\max_{1 \leq k^{\, \prime} \leq K_F}
		} \, h_{j, k^{\prime}}^F(x) - \varepsilon \, \right\}.
		$$}
	
	For a given $x^{\, \prime} \in X$ and with two families of i.i.d. samples
	$\{ \xi^{\, t} \}_{t=1}^{N_{\nu}}$ and
	$\{ \eta^{\, s} \}_{s=1}^{N_{\nu}}$ specified as in the SMM algorithm, and define
	\[
		\wt{V}_{N_{\nu}}^{\, \varepsilon}(x; x^{\, \prime}) \, \triangleq \, \psi\Bigg( \,
		\displaystyle{
			\frac{1}{N_{\nu}}
		} \, \displaystyle{
			\sum_{t=1}^{N_{\nu}}
		} \, \varphi\Bigg( \, \wt{G}^{\, \varepsilon}(x,\xi^{\, t}; x^{\, \prime}), \,
		\displaystyle{
			\frac{1}{N_{\nu}}
		} \, \displaystyle{
			\sum_{s=1}^{N_{\nu}}
		} \, \wt{F}^{\, \varepsilon}(x,\eta^{\, s};x^{\, \prime}) \Bigg) \, \Bigg),
		\]
	where $\wt{G}^{\, \varepsilon}(x,\xi;x^{\, \prime}) \triangleq
	\left( \, \wt{G}_i^{\, \varepsilon}(x,\xi;x^{\, \prime}) \, \right)_{i=1}^{\ell_G}$ and
	$\wt{F}^{\, \varepsilon}(x,\xi;x^{\, \prime}) \triangleq
	\left( \, \wt{F}_j^{\, \varepsilon}(x,\xi;x^{\, \prime}) \, \right)_{j=1}^{\ell_F}$
	whose components are given by
	\begin{equation} \label{eq:max linearized dc}
	\wt{G}_i^{\, \varepsilon}(x,\xi;x^{\, \prime})  \, \, \triangleq   \displaystyle{
		\min_{k_i \in {\cal A}_i^{G;\varepsilon}(x^{\, \prime})}
	} \, \wh{G}_{i, k_i}(x,\xi;x^{\, \prime}), \quad
	\wt{F}_j^{\, \varepsilon}(x,\xi;x^{\, \prime})  \, \, \triangleq    \displaystyle{
		\min_{k_j^{\, \prime} \in {\cal A}_j^{F;\varepsilon}(x^{\, \prime})}
	} \, \wh{F}_{j, k_j^{\, \prime}}(x,\xi;x^{\, \prime}),    \end{equation}
	where \[
	\begin{array}{ll}\wh{G}_{i,k_i}(x,\xi;x^{\, \prime}) \triangleq
	g_i^G(x,\xi) -  h_{i,k_i}^G(x^{\, \prime}) -
	\nabla h_{i,k_i}^G(x^{\, \prime})^{\top}( x - x^{\, \prime} ),\\[0.1in] \wh{F}_{j,k_j^{\, \prime}}(x,\xi;x^{\, \prime}) \triangleq
	g_j^F(x,\xi) -  h_{j,k_j^{\, \prime}}^F(x^{\, \prime}) -
	\nabla h_{j,k_j^{\, \prime}}^F(x^{\, \prime})^{\top}( x - x^{\, \prime} ).
	\end{array}\]
	The functions
	$\wt{G}_{i}^{\, \varepsilon}(\bullet, \xi ;x^{\, \prime})$ and $\wt{F}_{j}^{\, \varepsilon}(\bullet, \xi;x^{\, \prime})$ are no longer convex;
	nevertheless, the optimization problem,   ${\displaystyle{
			\operatornamewithlimits{\mbox{minimize}}_{x \, \in \, X}
		} \left\{ 		
		\wh{V}_{N_{\nu}}^{\, \varepsilon}(x; x^\nu) + \displaystyle{
			\frac{1}{2 \rho}
		} \, \| x -x^{\nu}\|^2 \right\}}  $ is practically implementable as it is equivelant to
	
	 \[ \begin{array}{l}
		{\displaystyle{
				\operatornamewithlimits{\mbox{minimize}}_{
					{ \left\{ \begin{array}{ll}
						k_i \, \in \, {\cal A}_i^{G;\varepsilon}(x^{\, \nu}),\,  i \in [\ell_G]\\
						k_j^{\, \prime} \, \in \, {\cal A}_j^{F;\varepsilon}(x^{\, \nu}), \,  j \in [\ell_F]
						\end{array} \right\}} }}
		}    \quad    \underbrace{\displaystyle{
				\operatornamewithlimits{\mbox{minimize}}_{x \, \in \, X}
			} \ \left\{ 	
			\wh{V}_{N_{\nu}}^{k,k^{\, \prime}}(x; x^\nu) + \displaystyle{
				\frac{1}{2 \rho}
			} \, \| x -x^{\nu}\|^2  \right\}}_{\mbox{a convex program, one for each
				pair $(k,k^{\, \prime}$)}}   ,
		\end{array} \]	
	{where}
	\[
	\begin{array}{cl}
	\wh{V}_{N_{\nu}}^{k,k^{\, \prime}}(x; x^{\, \prime}) \, &
	\triangleq \, \psi\left( \,
	\displaystyle{
		\frac{1}{N_{\nu}}
	} \, \displaystyle{
		\sum_{t=1}^{N_{\nu}}
	} \, \varphi\left( \, \wh{G}^k(x,\xi^{\, t}; x^{\, \prime}), \, \displaystyle{
		\frac{1}{N_{\nu}}
	} \, \displaystyle{
		\sum_{s=1}^{N_{\nu}}
	} \, \wh{F}^{k^{\, \prime}}(x,\eta^{\, s};x^{\, \prime}) \right) \, \right),\\[0.2in]
	\wh{G}^k(x,\xi;x^{\, \prime}) &\triangleq
	\left( \, \wh{G}_{i, k_i}(x,\xi;x^{\, \prime}) \, \right)_{i=1}^{\ell_G},\quad
	\wh{F}^{k^{\, \prime}}(x,\xi;x^{\, \prime}) \triangleq
	\left( \, \wh{F}_{j, k_j^{\, \prime}}(x,\xi;x^{\, \prime}) \, \right)_{j=1}^{\ell_F}.
	\end{array}
	\]
	In this case, the limiting surrogate function at an accumulation point $x^{\, \infty}$
	is
	\[
	\wh{V}_{\infty}^{\, \varepsilon}(x;x^{\infty}) \, \triangleq \,
	\psi\left( \, \mathbb{E}\left[ \, \varphi\left( \,
	\wt{G}^{\, \varepsilon}(x,\tilde{\xi};x^{\infty}), \, \mathbb{E}\left[ \,
	\wt{F}^{\, \varepsilon}(x,\tilde{\xi};x^{\infty}) \, \right] \, \right) \, \right]
	\, \right).
	\]
	Under the technical assumptions stipulated in
	Corollary~\ref{co:convergence without convexity} and the interchangeability of
	directional derivatives with expectations, it can be shown that the accumulation point
	$x^{\infty}$ satisfies that for any
	$\varepsilon^{\, \prime} \in [ \, 0,\varepsilon \, )$,
	$
	\left( \, \wh{V}_{\infty}^{\, \varepsilon^{\, \prime}}(\bullet;x^{\infty}) \,
	\right)^{\prime}(x^{\, \infty};x - x^{\, \infty}) \, \geq \, 0$ for all
	$x \, \in \, X$ (see \cite{PangRazaAlvarado18}).
	In particular, with $\varepsilon^{\, \prime} = 0$, it follows that every limit point
	of the SMM algorithm sequence
	satisfies the directional stationarity property of the given compound
	SP (\ref{eq:compound_composite_SP}).
	
	\gap
	
	With the above strong conclusion, it is natural to ask the question whether this
	conclusion can be extended to the case where the functions $h_{ik}^G$ are $h_{jk}^F$ are
	random.  The challenge of this case is that  both index sets
	${\cal A}_i^{G;\varepsilon}$ and
	${\cal A}_j^{F;\varepsilon}$ will be dependent on the realizations of random variable
	$\tilde{\xi}$,
	and significantly complicate the proof; further analysis of the behavior of the
	index sets ${\cal A}_i^{G;\varepsilon}(x,\xi)$ and
	${\cal A}_j^{F;\varepsilon}(x,\xi)$ is needed.  In this situation, the analysis
	in the paper \cite{qi2019statistical} may prove useful.  A detailed treatment of this
	extended case is left for a future study.
	
	\section{Post-Convergence Analysis: Error Bounds}
	\label{sec:error_bounds}
	
{In stochastic programming, it is important to assess  the quality of the candidate solutions produced by the iterative algorithms via a validation analysis.
The existing approaches for assessing the solution quality in SP include bounding the optimality gap \cite{bayraksan2006assessing,liu2020asymptotic}, and testing the
Karush-Kuhn-Tuckers conditions \cite{higle1991statistical,shapiro2009lectures}, as well as optimality functions \cite{royset2012optimality}.
The optimality-gap method is limited to convex problems, since it involves the availability of global minima.  The latter two methods are applicable to nonconvex SPs,
but require continuous differentiability.  Thus, these
previous results are not applicable to coupled nonsmooth and nonconvex problems; indeed, a practical validation analysis should test the stationarity of a candidate solution
with the type of stationarity dependent on the algorithm and the problem structure together.} For instance, for the compound SP
(\ref{eq:compound_composite_SP}) in which the component functions $G$ and $F$ are dc functions with max-smooth concave parts, the SMM Algorithm
asymptotically computes a compound critical point with probability 1,  whereas
the enhanced version of the SMM Algorithm could	asymptotically compute a directional stationary solution of the same problem with probability 1.
{In short, the issues of validation analysis and the related theory of error bounds for nonconvex and nondifferentiable stochastic programs is
a minimally explored subject at best; our analysis provides a first attempt to understand this important subject that requires further study.}
	
\gap
	
Specifically, in this section, we wish to assess solution quality within the SMM algorithm by the means of establishing a {\it posteriori error bound} with a
computable residual function, such that we are able to affirmatively assert whether an iterate point is such
an approximate ``stationary" point.  The derivation of such an error bound is the main
focus of this section.  We accomplish this in two steps. First, for iterates of the
Majorization-Minimization (MM) algorithm for solving a deterministic nonconvex
nonsmooth optimization problem with a locally Lipschitz continuous objective function and
convex constraints, we provide a surrogation-based error bound for the distance of
the test point to the set of Clarke stationarity points with a computable residual function.
Second, the discussion is extended to the SMM Algorithm for the
compound SP (\ref{eq:compound_composite_SP}) by means of a probabilistic surrogation-based global error bound.  We should note that
while there is an extensive literature on
the theory of error bounds \cite{drusvyatskiy2018error,pang1997error, facchinei2007finite,Royset20} for deterministic optimization problems,
the main departure of our results from this literature is that we aim at addressing the family of surrogation-based algorithms, of which the SMM algorithm
is an example.  In particular, we quantify the distance to stationarity in terms of the deviation of the
fixed-point property of the surrogation and sampling-based algorithmic map.

\gap

{In general, error bound results such as those in Theorems \ref{th:error bound}, \ref{th:pointwise error bound}, \ref{th:probabilistic eb} below provide
conceptual guidance to the design of termination rule of algorithms.
They typically depend on constants that are conceptually available but may not be practically so.   These results allow one to say, with theoretical
support, something like:
``under such a termination criterion, one can expect that the computed solution is a good approximation of an otherwise un-computable solution.''}
	
\subsection{A detour:  deterministic problems with an MM algorithm}
	
Consider the following optimization problem:
	\begin{equation} \label{eq:general optimization}
	\displaystyle{
		\operatornamewithlimits{\mbox{minimize}}_{x \in X}
	} \ \theta(x),
	\end{equation}
	where $X$ is a closed convex set in $\mathbb{R}^n$ and $\theta$ is a locally Lipschitz
	continuous function on $X$ so that the Clarke subdifferential
	$\partial_{\rm C} \, \theta(x)$ is well defined at any $x \in X$.  Assume the following
	three conditions on the family of surrogation (SR) functions
	$\left\{ \, \wh{V}(\bullet; x^{\, \prime}) \, \right\}_{x^{\, \prime} \in X}$:
	
	\gap
	
	{\bf (SR1)}  for any $x^{\, \prime} \in X$, the objective $\theta$ admits a
	convex majorant $\wh{V}(\bullet;x^{\, \prime})$ on $X$ satisfying:
	
	\gap
	
	(i) $\wh{V}(x^{\, \prime}; x^{\, \prime}) = \theta(x^{\, \prime})$
	(ii) $\wh{V}(x; x^{\, \prime}) \geq \theta(x)$ for all $x  \in X$; and
	(iii) a positive scalar $\kappa$ independent of $x^{\, \prime}$ exists such that:
	\begin{equation} \label{eq:the other inequality}
	\theta(x) + \displaystyle{
		\frac{\kappa}{2}
	} \, \| \, x - x^{\, \prime} \, \|^2 \, \geq \, \wh{V}(x;x^{\, \prime}),
	\epc \forall \, x \, \in \, X.
	\end{equation}

{The constant $\kappa$ can be chosen as
\begin{equation} \label{eq:kappa and surrogate}
\kappa \, = \, \displaystyle{
\sup_{\substack{x \neq x^{\, \prime} \\
x, x^{\, \prime} \in X}}
} \ \displaystyle{
\frac{2 \left( \, \wh{V}(x;x^{\, \prime}) - \theta(x) \, \right)}{\| \, x - x^{\, \prime} \, \|^2}
}
\end{equation}
provided that the supremum is finite.  Thus $\kappa$ is an upper bound of the gap of majorization (i.e., the numerator) normalized by
the half squared distance between $x$ and the base vector $x^{\prime}$ of the majorization function $\wh{V}(\bullet,x^{\, \prime})$
(i.e., the denominator).   The assumption (SR1)
essentially postulates that such a relative gap of majorization is finite based on which a majorization-based error bound theory is developed.
This relative gap $\kappa$ is a quantitative measure of the tightness of the majorization: the tighter the majorization, the smaller $\kappa$ is.
The supremum is taken over all $x \neq x^{\, \prime}$ in order
to accommodate the analysis below that deals with all possible base vectors $x^{\, \prime}$ and arbitrary $x \in X$.
The role of the constant $\kappa$ is made explicit in the
error estimates obtained in Theorems~\ref{th:error bound} and \ref{th:pointwise error bound}.}

	% {\bf (SR$_b$)} the set-valued map
	% $(x,x^{\, \prime}) \in X \times X \mapsto \partial_x V(x;x^{\, \prime})$ is closed;
	% i.e., its graph
	% \[
	% \left\{ \, ( a,x,x^{\, \prime} ) \, \in \, \mathbb{R}^n \, \times \, X \, \times X
	% \, \mid \, a \, \in \, \partial_x V(x;x^{\, \prime}) \, \right\}
	% \]
	% is a closed set.  Remark: unlike (SR$_a$), (SR$_b$) pertains to a upper
	% semicontinuity of
	% the subdifferential map $\partial_x V$ of the surrogate function $V$ with reference to
	% its two arguments.  The role of (SR$_b$) is highlighted in Proposition~
	%
	% \gap
	%
	% Notice that while the majorization of $V(\bullet,x^{\, \prime})$ is a global condition,
	% i.e., $V(x,x^{\, \prime}) \geq \theta(x)$ for all $x \in X$, the reverse inequality
	% (\ref{eq:the other inequality}) is local.  We will subsequently show that this local
	% condition is satisfied by the class of dc functions given by (\ref{eq:max dc}).
	For a given but arbitrary scalar
	$\rho > 0$, let $\wt{V}_{\rho}(x;x^{\, \prime}) \triangleq \wh{V}(x;x^{\, \prime})
	+ \displaystyle{
		\frac{1}{2 \,\rho}
	} \, \| \, x - x^{\, \prime} \, \|^2$.  The majorization-minimization (MM) algorithm
	for solving the problem (\ref{eq:general optimization}) generates a sequence
	$\{ x^{\nu} \}$ by the fixed-point iteration:
	$x^{\nu+1} = {\cal M}_{\wt{V}_{\rho}}(x^{\, \nu})$ applied to
	the map
	\[
	{\cal M}_{\wt{V}_{\rho}} : x^{\, \prime} \mapsto \displaystyle{
		\operatornamewithlimits{\mbox{argmin}}_{x \in X}
	} \ \wt{V}_{\rho}(x;x^{\, \prime}).
	\]
	% Before analyzing the convergence of this algorithm,
	% an immediate question to ask is whether the map ${\cal M}_{\wt{V}_{\rho}}$ has a fixed
	% point.  Of course it does if it is continuous.  This is where assumption (SR$_b$) is
	% useful.  We formally state this existence claim in the following proposition.
	%
	% \begin{proposition} \label{pr:existence fixed point} \rm 	
	% Under (SR$_b$), the map ${\cal M}_{\wt{V}_{\rho}}$ is continuous and thus has a fixed
	% point on $X$.
	% \end{proposition}
	%
	% \begin{proof}. We only sketch the proof.  Let $\{ x^k \} \subset X$ be a sequence
	% converging to $x^{\infty}$.  We need to show that the sequence
	% $\{ {\cal M}_{\wt{V}_{\rho}}(x^k) \}$
	% converges to ${\cal M}_{\wt{V}_{\rho}}(x^{\infty})$.  To do this, it suffices to show
	% that every accumulation point of $\{ {\cal M}_{\wt{V}_{\rho}}(x^k) \}$, at least one of
	% which must exist, is an optimal solution of the problem:
	% $\displaystyle{
	% \operatornamewithlimits{\mbox{minimize}}_{x \in X}
	% } \ \wt{V}_{\rho}(x;x^{\infty})$.  In turn, this can proved using a subgradient
	% argument
	% the assumption (SR$_b$), and the uniqueness of the optimal solution of the latter
	% strongly convex problem.
	% \end{proof}
	%
	A natural question to ask about the fixed-point iterations of the map
	${\cal M}_{\wt{V}_{\rho}}$ is when to terminate
	the iterations so that a current iterate $x^{\nu}$ can be certified to be
	``asymptotically stationary'' for the
	problem (\ref{eq:general optimization}).   Let
	$S_{X,\theta}^{\, C}$ be
	the set of \emph{C(larke) stationary points} of the problem
	(\ref{eq:general optimization}); i.e.,
	\[
	S_{X,\theta}^{\, C} \, \triangleq \, \left\{ \, x \, \in \, X \, \mid \,
	0 \, \in \, \partial_C \,  \theta(x) + {\cal N}(x;X) \, \right\},
	\]
	where ${\cal N}(x;X)$ is the normal cone of $X$ at $x \in X$.
	For a given $\wh{x} \in X$, the (Euclidean) distance from $\wh{x}$ to
	$S_{X,\theta}^{\, C}$,
	denoted $\mbox{dist}(\wh{x};S_{X,\theta}^{\, C})$, is a conceptual measure of
	stationarity of $\wh{x}$ in the sense that this distance is equal to zero if and only
	if $\wh{x}$ is C-stationary.  An alternative measure is
	$r_{X,\theta}^C(\wh{x}) \, \triangleq \mbox{dist}(0; \partial_C \,  \theta(\wh{x}) +
	{\cal N}(\wh{x};X))$, which is the (Euclidean) distance from the origin to the set
	$\partial_C \, \theta(\wh{x}) + {\cal N}(\wh{x};X)$.    Note that
	these two measures are related in the following sense,
	\begin{align}
	\label{eq:dist_equivalence}
	\mbox{dist}(\wh{x};S_{X,\theta}^{\, C}) \, = \, 0 \ \Leftrightarrow \
	\,
	r_{X, \theta}^C(\wh{x}) \, = \, 0,
	\end{align}
	which means that either distance can serve as a residual function for the other.
	Yet, it is not easy to compute the distance measure $r_{X, \theta}^C(\wh{x})$  when $\theta$
	is the composition of nonsmooth functions. %or involves expect-value functions.
	Tailored to the MM algorithm, we are interested in establishing an upper bound of
	$\mbox{dist}(\wh{x}; S_{X, \theta}^C)$ by the computable residual
	$r_{\wt{V}_{\rho}}(\wh{x}) \, \triangleq \,
	\left\| \, \wh{x} - {\cal M}_{\wt{V}_{\rho}}(\wh{x}) \, \right\|$ for any $\wh{x}$ of interest.
	We consider two types of error bound conditions, namely local error bound and
	metric subregularity, under which the computable residual $r_{\wt{V}_\rho}(\wh{x})$
	can be used to measure the asymptotic stationarity of a test point globally or locally,
	respectively.
	
	\subsubsection{Local error bounds}
	
	An error bound property extends the above equivalence \eqref{eq:dist_equivalence}
	to an inequality between the two distances when they are not zero.
	Specifically, we say that a \emph{local error bound} (LEB) holds
	for the set $S_{X,\theta}^{\, C}$ with $r_{X,\theta}^C(x)$
	as the residual on $X$ if
	
	\gap
	
	\textbf{(LEB)} there exist positive constants
	$\eta$ and $\varepsilon$ such that
	\begin{equation} \label{eq:local error bound}
	r_{X,\theta}^{\, C}(x) \, \leq \, \varepsilon \ \Rightarrow \
	\mbox{dist}(x;S_{X,\theta}^{\, C}) \, \leq \, \eta \,
	\left[ \, r_{X,\theta}^C(x) \, \right], \epc \forall \, x \, \in \, X.
	\end{equation}
	We refer to
	\cite[Chapter~6]{facchinei2007finite} for some basic discussion about error
	bounds.  In the case when $\theta$ is continuously differentiable so that
	$\partial_C\, \theta(x)$ consists of the gradient $\nabla \theta(x)$ only, the stationarity
	condition of (\ref{eq:general optimization}) becomes a variational inequality (VI)
	defined by the pair $(X,\nabla \theta)$ and the residual $r_{X,\theta}^{\, C}(x)$ is the so-called \emph{normal residual} of the VI.  It
	is shown in the reference \cite[Proposition 6.2.1]{facchinei2007finite}
	that in this differentiable case,
	the error bound (\ref{eq:local error bound}) is intimately
	related to the ``semistability'' of the VI on hand.  Generalization of this connection
	to the nondifferentiable case is outside
	the scope of this paper and deserves a full investigation elsewhere.  By relating
	the two quantities under the LEB condition, one can obtain
	an inequality bounding the distance $\mbox{dist}({x}; S_{X, \theta}^C)$ in terms of
	the computable residual $r_{\wt{V}_{\rho}}(x)$.  Specifically, we are
	interested in establishing the existence of a global and computable error bound:
	\gap
	
	% Since $r_{X,\theta}^C(x)$ is not computable, this quantity cannot be used as a
	% practical termination criterion for an iterative method, such as the MM algorithm,
	% for solving (\ref{eq:general optimization}).
	
	% under the {\sl metric subregularity} at the pair $(\wh{x},0)$ of the set-valued map
	% $\Phi_{X,\theta}^{\, C} : x \mapsto \partial_C \theta(x) + {\cal N}(x;X)$, which
	% we call the {\sl C-stationarity map}.  This property asserts that there exist
	% positive constants $\delta_{\rm MS}$ and
	% $\gamma$ satisfying for all $x \in \mathbb{B}(\wh{x},\delta_{\rm MS})$ the following
	% inequality holds:
	% \begin{equation} \label{eq:metric regularity}
	% \mbox{dist}(x;S_{X,\theta}^{\, C}) \, = \,
	% \mbox{dist}\left(x;( \Phi_{X,\theta}^{\, C} \, )^{-1}(0) \, \right)
	% \, \leq \, \gamma \, \mbox{dist}\left(0;\Phi_{X,\theta}^{\, C}(x) \, \right)
	% \, = \,
	% \mbox{dist}\left(0;\partial_C \theta(x) + {\cal N}(x;X) \right),
	% \end{equation}
	% with $\mathbb{B}(\wh{x},\delta)$ being the Euclidean ball centered at
	% $\wh{x}$ and with radius $\delta_{\rm MS}$.
	% Algorithmically, motivated by the fixed-point property of an accumulation
	% point of the sequence $\{ x^{\nu} \}$ produced by the MM algorithm
	% (cf.\ Theorem~\ref{thm:convergence_smm} for the SMMA),
	
	\textbf{(S-GEB$_\rho$)} for every $\rho > 0$, there exists a positive constant
	$\wh{\eta}$ such that
	\begin{equation} \label{eq:error bound}
	%r_{X,\theta}^{\, C}(x) \, \leq \, \varepsilon \ \Rightarrow \
	\mbox{dist}\left(x;S_{X,\theta}^{\, C} \right) \, \leq \, \wh{\eta} \, \,
	r_{\wt{V}_{\rho}}(x), \epc \mbox{for all $x \in X$ of interest}.
	\end{equation}
	The above inequality is a surrogation-based global error bound  (S-GEB)
	for Clarke stationarity of the problem (\ref{eq:general optimization}) with the
	computable quantity $r_{\wt{V}_{\rho}}(x)$ as the residual.
	Informally, this error bound suggests that the
	smaller this residual is (i.e., the closer $\wh{x}$ is to satisfying the
	fixed-point property of the map $M_{\wt{V}_{\rho}}$), the smaller the left-hand
	side will be; (i.e., the closer $\wh{x}$ is to the set of
	C-stationary solutions of (\ref{eq:general optimization})).  Thus
	(\ref{eq:error bound}) provides a theoretical basis for the
	use of $r_{\wt{V}_{\rho}}(x)$    as
	a practical termination rule for the surrogation-based algorithm.  Note that a
	necessary condition for (\ref{eq:error bound}) to hold is that
	
	\gap
	
	{\bf (SR2)} $\mbox{FIX}( {\cal M}_{\wh{V}} ) \subseteq
	S_{X,\theta}^{\, C}$, where
	$\mbox{FIX}({\cal M}_{\wh{V}}) \triangleq \{ x \in X: x = {\cal M}_{\wh{V}}(x)\}$;
	this is equivalent to assuming $\mbox{FIX}( {\cal M}_{\wt{V}_{\rho}} ) \subseteq
	S_{X,\theta}^{\, C}$ for any $\rho > 0$.
	
	\gap
	
	Assumption (SR2), depending on the choice of surrogation functions, amounts to
	stipulating that the quantity $r_{\wt{V}_{\rho}}(x)$ is a legitimate residual for
	the set $S_{X,\theta}^{\, C}$ for any $\rho > 0$.  Under this assumption,
	we establish the surrogation-based global error bound (\ref{eq:error bound})
	via two steps.  The first step (Lemma~\ref{lm:stationarity_distance}) is to apply
	Ekeland's variational principle to obtain some
	preliminary bounds of an auxiliary vector $\bar{x}$ derived from $x$.  The proof
	of this step,  which is similar to the argument employed in \cite{drusvyatskiy2018error} is given in Appendix.
	The second
	step is to invoke the (LEB) assumption on the set $S_{X,\theta}^{\, C}$
	to deduce the bound (\ref{eq:error bound}) when the right-hand side is sufficiently
	small, which is further extended to  hold for all vectors $x$ in the compact set $X$
	(or in the case when $X$ is not assumed compact, to any compact subset $T$ of $X$)
	from  \cite[Proposition~6.1.3]{facchinei2007finite}.  It is worth mentioning that in the
	proof below, the Clarke subdifferential is used for convenience and can be replaced by
	any suitable subdifferential for which the condition (SR2) holds and a first-order
	optimality condition holds for the optimization
	problem  to which Ekeland's variational principle is applied.  The
	significance of this remark is that in the setting of the subsequent
	Theorem~\ref{th:error bound}, the error bound (\ref{eq:error bound}) holds for any
	suitable subdifferential of $\theta$ that contains $\partial_C \theta(x)$ as a subset.
	For instance, if $\theta = g - h$ is a dc function with $g$ and $h$ both being
	convex, then the inequality (\ref{eq:error bound}) can be replaced by
	\[
	\mbox{dist}\left(x;S_{X, \theta}^{DC}
	\right) \, \leq \, \wh{\eta} \,
	r_{\wt{V}_{\rho}}(x), \, \mbox{ for all } x \in X \mbox{ of interest, }
	\]
	where $S_{X,\theta}^{DC}= \{x\in X: \, 0 \in  \partial g(x) - \partial h(x) +
	{\cal N}(x;X) \}$,
	which establishes an error bound for the set of critical points of $\theta$ on $X$.
	Throughout the discussion below, we continue to employ the Clarke-subdifferential but
	emphasize that a replacement of this subdifferential by other appropriate
	subdifferentials may be warranted for certain classes of objective functions $\theta$.
	
	\begin{lemma}[Ekeland's variational principle] \label{lm:eke} \rm
		Let $f: X \subseteq \mathbb{R}^n \to \mathbb{R}$ be a lower semicontinuous function that
		is bounded from below on a closed convex set $X$.  Suppose that for some scalar
		$\varepsilon >0$ and some vector $\wh{x} \in X$, we have
		$f(\wh{x}) \leq \inf_{x \in X} f(x) + \varepsilon$.  Then for any $\lambda >0$,
		there exists a point $\bar{x}$ satisfying: (i) $f(\bar{x}) \leq f(\wh{x})$;
		(ii) $\norm{\bar{x}- \wh{x}} \leq \varepsilon/\lambda$; and
		(iii) $\bar{x} = \displaystyle{
			\operatornamewithlimits{\mbox{argmin}}_{x\in X}
		} \, \{ \, f(x) + \lambda \norm{x- \wh{x}} \, \}$.
		
	\end{lemma}
	
	\begin{lemma} 
		\label{lm:stationarity_distance} \rm
		Let $X$ be a compact convex set in $\mathbb{R}^n$ and $\theta$ be a locally Lipschitz
		continuous function on $X$.  Suppose the conditions (SR1) and (SR2) hold for the optimization
		problem \eqref{eq:general optimization} and the surrogate function $\wh{V}$.
		For every vector $\wh{x} \in X$ and every scalar $\rho > 0$, there exists
		$\bar{x} \in X$, dependent on $\rho$ and $\wh{x}$, such that
		\[
		\norm{\bar{x}- \wh{x}} \, \leq \, 2r_{\wt{V}_\rho}(\wh{x}), \epc \mbox{and} \epc
		r_{X,\theta}^C(\bar{x}) = \mbox{dist}(0; \partial_C \, \theta(\bar{x}) +
		\mathcal{N}(\bar{x}; X)) \, \leq \, \left( \, \displaystyle{
			\frac{2}{\rho}
		} + \displaystyle{
			\frac{5 \, \kappa}{2}
		} \, \right) \, r_{\wt{V}_\rho}(\wh{x}).\]
	\end{lemma}

	Employing the (LEB) condition (\ref{eq:local error bound}), we can establish the
	(S-GEB$_\rho$) condition (\ref{eq:error bound}).
	
	\begin{theorem} \label{th:error bound} \rm
		Let $X$ be a compact convex set in $\mathbb{R}^n$ with diameter $D$ and $\theta$ be locally Lipschitz
		continuous on $X$.  Suppose that the condition (LEB) with the constant pair  $(\varepsilon, {\eta})$,
and conditions (SR1) and (SR2) with a constant $\kappa$ hold for the
		optimization problem \eqref{eq:general optimization} and the convex surrogate function
		$\wh{V}$. Then for every $\rho > 0$, (S-GEB$_{\rho}$) holds; that is
		there exists a constant {$\wh{\eta} \triangleq   \max\left\{ \, 2 + \eta \,
		\left( \, \displaystyle{
			\frac{2}{\rho}
		} + \displaystyle{
			\frac{5 \, \kappa}{2}
		} \, \right), \displaystyle{\frac{2D}{\varepsilon}} \left( \, \displaystyle{
		\frac{2}{\rho}
	} + \displaystyle{
	\frac{5 \, \kappa}{2}
} \, \right)  \right\}  $} such that
		\begin{equation} \label{eq:eb with wheta}
		%r_{X,\theta}^{\, C}(x) \, \leq \, \varepsilon \ \Rightarrow \
		\mbox{dist}\left(x;S_{X,\theta}^{\, C} \right) \, \leq \, \wh{\eta} \, \,
		r_{\wt{V}_{\rho}}(x), \epc \mbox{for all } x \in X.
		\end{equation}
	\end{theorem}
	
	\begin{proof} % Clearly, if suffices to show the ``only if''.  So suppose that
		% $\mbox{FIX}( {\cal M}_{\wh{V}_{\rho}} ) \subseteq S_{X,\theta}^{\, C}$.
		We first show that positive
		$\varepsilon^{\, \prime}$ and $\eta^{\, \prime}$ exist such that for all $x \in X$,
		\begin{equation} \label{eq:intermediate error bound}
		r_{\wt{V}_{\rho}}(x)\, \leq \,
		\varepsilon^{\, \prime}\ \Rightarrow \
		\mbox{dist}\left(x;S_{X,\theta}^{\, C} \right) \, \leq \, \eta^{\, \prime} \,
		r_{\wt{V}_{\rho}}(x).
		\end{equation}
		Let $x \in X$ be arbitrary, satisfying
		that $x \neq {\cal M}_{\wt{V}_{\rho}}(x)$ without loss of generality.  Corresponding
		to this $x$, let $\bar{x} \in X$ be
		the vector satisfying the two inequalities in Lemma~\ref{lm:stationarity_distance}.
		Let $\varepsilon$ and $\eta$ be the constants with which (LEB) holds and
		$\varepsilon^{\, \prime} \triangleq
		\left(  \displaystyle{
			\frac{2}{\rho}
		} +   \displaystyle{
			\frac{5 \, \kappa}{2}
		}  \right)^{-1} \varepsilon $.
		We then have
		\[ \begin{array}{lll}
		r_{\wt{V}_{\rho}}(x) \, \leq \,
		\varepsilon^{\, \prime} & \Longrightarrow &
		r_{X,\theta}^{\, C}(\bar{x}) \, \leq \, \left( \displaystyle{
			\frac{2}{\rho}
		} + \displaystyle{
			\frac{5 \, \kappa}{2}
		} \, \right) \, r_{\wt{V}_{\rho}}(x) \, \leq \,
		\varepsilon \stackrel{(\rm LEB)}{\Longrightarrow}   \mbox{dist}(\bar{x};S_{X,\theta}^{\, C})
		\, \leq \, \eta \cdot
		  r_{X,\theta}^C(\bar{x}) % \, \leq \, \eta \, \left( \, \displaystyle{\frac{2}{\rho}} + \displaystyle{	\frac{5 \, \kappa}{2}} \, \right)  \, r_{\wt{V}_{\rho}}(x).
		\end{array}
		\]
		By the triangle inequality, we have
		\[ \begin{array}{lll}
			r_{\wt{V}_{\rho}}(x) \, \leq \,
		\varepsilon^{\, \prime}   \Longrightarrow \mbox{dist}(x;S_{X,\theta}^{\, C})  \leq  \mbox{dist}(\bar{x};S_{X,\theta}^{\, C})
		+ \| \, x - \bar{x} \, \|  \leq  \eta \,
		\left( \, \displaystyle{
			\frac{2}{\rho}
		} + \displaystyle{
			\frac{5 \, \kappa}{2}
		} \, \right) \, r_{\wt{V}_{\rho}}(x)
		+ 2 \,r_{\wt{V}_{\rho}}(x),
		\end{array}
		\]
		from which the existence of the constants $\varepsilon^{\, \prime} = \left(  \displaystyle{
			\frac{2}{\rho}
		} +   \displaystyle{
			\frac{5 \, \kappa}{2}
		}  \right)^{-1} \varepsilon$ and
		$\eta^{\, \prime} = 2 + \eta \,
		\left( \, \displaystyle{
			\frac{2}{\rho}
		} + \displaystyle{
			\frac{5 \, \kappa}{2}
		} \, \right)  $ satisfying (\ref{eq:intermediate error bound}) follows readily.
		Since $X$ is a compact set with radius $D$, by a scaling argument employed
		in \cite[Proposition~6.1.2]{facchinei2007finite}, the
		local error (\ref{eq:intermediate error bound}) easily implies the
		surrogation-based global error bound for all $x \in X$. 
	\end{proof}
	\gap

{In the constant $\wh{\eta} \triangleq \max\left\{ \, 2 + \eta \,
	\left( \, \displaystyle{
		\frac{2}{\rho}
	} + \displaystyle{
		\frac{5 \, \kappa}{2}
	} \, \right), \displaystyle{\frac{2D}{\varepsilon}} \left( \, \displaystyle{
		\frac{2}{\rho}
	} + \displaystyle{
		\frac{5 \, \kappa}{2}
	} \, \right)  \right\}$, % in (\ref{eq:eb with wheta}),
the pair $(\varepsilon,\eta)$ consists of the
local error bound (LEB) constants of the stationarity map $S_{X,\theta}^{\, C}$;
$\rho$ and $\kappa$ are parameters in the fixed-point map ${\cal M}_{\wt{V}_{\rho}}$ that is the abstraction
of the majorization algorithm, with the former ($\rho$) being the proximal parameter and the latter ($\kappa$)
being the relative gap of majorization.  As with all error bound results, the constant $\wh{\eta}$ is best considered
a conceptual scalar relating the distance function $\mbox{dist}\left(x;S_{X,\theta}^{\, C} \right)$ with the residual
$r_{\wt{V}_{\rho}}(x)$ and showing that the latter two quantities are of the same order.  The value $\wh{\eta}$
is difficult to be fully specified unless the constants $\varepsilon$, $\eta$, $\rho$, and $\kappa$ are
available.  Among the latter three constants, $\varepsilon$ and $\eta$ are again best considered as   conceptual scalars that can be
explicitly specified only in very particular instances; its existence is related to the qualitative stability of the solution
set $S_{X,\theta}^{\, C} $ under data perturbations.  In spite of the lack of more explicit quantification of $\wh{\eta}$,
its expression as given provides some qualitative information about the role of the two main parameters $\kappa$ and $\rho$ in the
residual bound in terms of the algorithmic map; namely, the smaller the majorization gap, the sharper the error bound; and the
same is true for the proximal constant.}
	% Theorem~\ref{th:error bound} has several important consequences that we highlight
	% in the corollary below.  No proof is needed for the corollary.
	
	% \begin{corollary} \label{co:consequences of LEB} \rm
	% Under the assumptions of Theorem~\ref{th:error bound}, the following statements hold:
	% \begin{description}
	% \item[\rm (a)] every fixed point of the map $\wh{V}_{\rho}$ is a C-stationary point.
	% \item[\rm (b)] every accumulation point of a sequence of iterates produced by the
	% MM algorithm must be a C-stationary point of (\ref{eq:general optimization}). 	
	% \hfill $\Box$
	% \end{description}
	% \end{corollary}
	
	\subsubsection{Metric subregularity} \label{subsubsec:MS}
	
	A benefit of the local error bound postulate is that it yields a quantitative
	surrogation-based global error bound
	(\ref{eq:error bound}) for a measure of approximate stationarity of any test vector
	$x \in X$; in particular, the bound applies to any iterate $x^{\nu}$ produced by
	the MM algorithm.  A pointwise version of this postulate will yield similar conclusions
	but is restricted to points near a given reference point of concern.  We say that
	the point-to-set map $x (\in X) \mapsto \partial_C \theta(x) + {\cal N}(x;X)$ is
	\emph{metrically subregular} \cite[Section~3.8 3H]{dontchev2009implicit}
	at a given C-stationary point $\wh{x}$ of (\ref{eq:general optimization}) if
	
	\gap
	
	\textbf{(MS)} there
	exist positive constants $\delta$ and $\eta$ such that for all $x \in X$,
	\begin{align}
	\label{eq:metric_subregularity}
	\| \, x - \wh{x} \, \| \, \leq \, \delta \ \Rightarrow \
	\mbox{dist}(x;S_{X,\theta}^{\, C}) \, \leq \, \eta \, r_{X,\theta}^C(x);
	\end{align}
	or equivalently,
	\[
	\| \, x - \wh{x} \, \| \, \leq \, \delta \ \Rightarrow \
	\mbox{dist}\left(x;\left[ \,
	\partial_C \, \theta + {\cal N}(\bullet;X) \, \right]^{-1}(0) \, \right) \, \leq \,
	\eta \, \mbox{dist}\left(0;\partial_C \,  \theta(x) + {\cal N}(x;X) \right).
	\]
	By \cite[Proposition~6.1.2]{facchinei2007finite}, it follows that the local error bound
	postulate implies metric subregularity at every C-stationary solution of
	(\ref{eq:general optimization}).  With an additional assumption below,  we establish
	the surrogation-based computable error bound in
	Theorem \ref{th:pointwise error bound} below, which is the pointwise analog of
	Theorem \ref{th:error bound}, but holds in a neighborhood of the C-stationary point
	satisfying metric subregularity.
	
	\gap
	
	{\bf (SR3)} The set-valued map
	$\partial_x \wh{V} : (x,x^{\, \prime}) \in X \times X \mapsto
	\partial_x \wh{V}(x;x^{\, \prime})$ is closed; that is, its graph $
	\left\{ \, ( a,x,x^{\, \prime} ) \, \in \, \mathbb{R}^n \, \times \, X \, \times X
	\, \mid \, a \, \in \, \partial_x \wh{V}(x;x^{\, \prime}) \, \right\} $
	is a closed set.
	
	\gap
	
	{\bf Remark:} Unlike the previous (SR1), (SR3) pertains to an outer semicontinuity of
	the subdifferential map $\partial_x \wh{V}$ of the surrogate function $\wh{V}$
	with reference to its two arguments.  The role of (SR3) is highlighted in the
	following proposition.
	
	\begin{proposition} \label{pr:existence fixed point} \rm 	
		Under (SR3), the map ${\cal M}_{\wt{V}_{\rho}}$ is continuous and thus has a fixed
		point in $X$.
	\end{proposition}
	
	\begin{proof}  We only sketch the proof.  Let $\{ x^k \} \subset X$ be a sequence
		converging to $x^{\infty}$.  We need to show that the sequence
		$\{ {\cal M}_{\wt{V}_{\rho}}(x^k) \}$
		converges to ${\cal M}_{\wt{V}_{\rho}}(x^{\infty})$.  To do this, it suffices to show
		that every accumulation point of $\{ {\cal M}_{\wt{V}_{\rho}}(x^k) \}$, at least one of
		which must exist, solves
		$\displaystyle{
			\operatornamewithlimits{\mbox{minimize}}_{x \in X}
		} \ \wt{V}_{\rho}(x;x^{\infty})$.  In turn, this can be proved using a subgradient
		argument, assumption (SR3), and the uniqueness of the optimal solution of the latter
		strongly convex problem.
	 
	\end{proof}
	\gap
	
	% Unlike the global and computable error bound (S-GEB) \eqref{eq:error bound} established  under (LEB) condition for all $x\in X$, we are  interested in establishing such computable error bound for $x$ in a neighborhood of the fixed point of $\mathcal{M}_{\wt{V}_\rho}$ satisfying (MS).
	
	\begin{theorem} \label{th:pointwise error bound} \rm
		Let $X$ be a compact convex set in $\mathbb{R}^n$ and $\theta$ be locally Lipschitz
		continuous on $X$.  Let the surrogate function $\wh{V}$ satisfy the conditions (SR1),
		(SR2), and (SR3).  Let $\wh{x}$ be a fixed point of ${\cal M}_{\wt{V}_{\rho}}$ (thus
		a C-stationary solution of (\ref{eq:general optimization}) under (SR2)).  If $\wh{x}$
		is metrically subregular satisfying \eqref{eq:metric_subregularity}, there exist
		positive scalars $\delta^{\, \prime}$ and $\eta^{\, \prime} > 0$ such that
		\begin{equation} \label{eq:metric subreg}
		x \, \in \, \mathbb{B}(\wh{x};\delta^{\, \prime}) \ \Rightarrow \
		\mbox{dist}(x;S_{X,\theta}^{\, C}) \, \leq \, \eta^{\, \prime} \, r_{\wt{V}_{\rho}}(x).
%		\norm{x - {\cal M}_{\wt{V}_{\rho}}(x)}.
		\end{equation}
	\end{theorem}
	
	\begin{proof}
		Let $\delta$ and $\eta$ be the constants associated with the metric subregularity of
		$\wh{x}$.  By the continuity of ${\cal M}_{\wt{V}_{\rho}}$, for the given $\delta$,
		we may choose $\delta_1 > 0$ such that
		\[
		\| \, x - \wh{x} \, \| \, \leq \, \delta_1 \ \Rightarrow \
		\norm{ {\cal M}_{\wt{V}_{\rho}}(x) - \wh x } \, \leq \, \delta/4.
		\]
		Let $\delta^{\, \prime} \triangleq \min\left( \, \delta_1, \, \displaystyle{
			\frac{\delta}{6}
		} \, \right)$.  Let $x \in X$ be arbitrary satisfying
		$\| x - \wh{x} \| \leq \delta^{\, \prime}$, and without loss of generality we may
		assume that $x \in X \setminus \mbox{FIX}({\cal M}_{\wt{V}_{\rho}})$.  Thus
		$\norm{x - {\cal M}_{\wt{V}_{\rho}}(x)} \leq \delta/4 + \delta^{\, \prime}$.  Let
		$\bar{x} \in X$ be the vector asserted by Lemma~\ref{lm:stationarity_distance} satisfying
		\[
		\norm{\bar{x}- x} \, \leq \, 2 \, \norm{x - {\cal M}_{\wt{V}_{\rho}}(x)} \, \leq \,
		\displaystyle{
			\frac{\delta}{2}
		}  + 2\, \delta^{\, \prime} \epc \mbox{ and } \epc r_{X,\theta}^C(\bar{x}) \, \leq \, \left( \, \displaystyle{
			\frac{2}{\rho}
		} + \displaystyle{
			\frac{5 \, \kappa}{2}
		} \, \right) \, \norm{x - {\cal M}_{\wt{V}_{\rho}}(x)}.
		\]
		We have
		$
		\| \, \bar{x} - \wh{x} \, \| \, \leq \, \| \, \bar{x} - x \, \| +
		\| \, x - \wh{x} \, \| \, \leq \, 3 \, \delta^{\, \prime} + \displaystyle
		\frac{\delta}{2} \, \leq \, \delta.
		$
		Thus by the metric subregularity at $\wh{x}$, it follows that
		\[
		\mbox{dist}(\bar{x};S_{X,\theta}^{\, C}) \, \leq \, \eta \, r_{X,\theta}^C(\bar{x})
		\, \leq \, \eta \, \left( \, \displaystyle{
			\frac{2}{\rho}
		} + \displaystyle{
			\frac{5 \, \kappa}{2}
		} \, \right) \, \norm{x - {\cal M}_{\wt{V}_{\rho}}(x)}.
		\]
		Therefore,
		\[
		\mbox{dist}(x;S_{X,\theta}^{\, C}) \, \leq \, \| \, x - \bar{x} \, \| +
		\mbox{dist}(\bar{x};S_{X,\theta}^{\, C}) \, \leq \, \left[ \, {2 +
			\left( \, \displaystyle{
				\frac{2}{\rho}
			} + \displaystyle{
				\frac{5 \, \kappa}{2}
			} \,  \right) \eta } \, \right] \,
		\norm{x - {\cal M}_{\wt{V}_{\rho}}(x)},
		\]
		establishing the desired implication (\ref{eq:metric subreg}) with
		$\eta^{\, \prime} = 2 +
		\left( \, \displaystyle{
			\frac{2}{\rho}
		} + \displaystyle{
			\frac{5 \, \kappa}{2}
		} \,  \right) \eta$.
 
\end{proof}

\gap

{Since the proofs of Theorems~\ref{th:error bound} and \ref{th:pointwise error bound} are both based on
Lemma~\ref{lm:stationarity_distance} and similar derivations, it is perhaps not surprising that the error bound
constant $\eta^{\, \prime}$ ends up to be the same as the first term in $\wh{\eta}$ in the previous result.  Established under different assumptions,
the validity of these bounds (\ref{eq:eb with wheta}) and (\ref{eq:metric subreg}) is of different kind: the former is of a
global nature that holds for all $x \in X$, whereas the latter pertains only to test vectors $x$ sufficiently near a given
fixed point $\wh{x}$ of the algorithmic map ${\cal M}_{\wt{V}_{\rho}}$ satisfying metric subregularity.}

	\subsection{Probabilistic stopping rule of the SMM algorithm}
	
	Returning to the compound SP problem \eqref{eq:compound_composite_SP}, whose statement we
	repeat below:
	\begin{equation}
	\label{eq:compound SP repeated}
	\displaystyle{
		\operatornamewithlimits{\mbox{minimize}}_{x\in X}
	}  \;\;\Theta(x) \, \triangleq \, \psi\left( \, \mathbb{E}\left[ \,
	\varphi(G(x, \tilde{\xi} ), \, \mathbb{E}[ \, F(x, \tilde{\xi} ) \, ] \, \right] \, \right);
	\end{equation}
	we aim to assess the quality of the SMM iterate $x^\nu$ by treating it as a fixed point for any iteration $\nu$. Notice that with the
	random variables in $\Theta$, it is difficult to compute the deterministic
	error bound \eqref{eq:error bound}  for compound SP (\ref{eq:compound SP repeated}), unless
	the sample space of the random variable $\tilde{\xi}$ is finite and known. {Thus
	we are interested in the establishment of a   probabilistic error bound with a   sampling-based
	residual function  using  iid  samples $\{\bar  \xi^{\,t}\}_{t \in [N]}$ and $\{\bar \eta^{\, s}\}_{s \in [N]}$, independently with the sample
sets used in the SMM algorithm.  In fact,
	the ultimate bound we derive for the compound SP is of the
	``with high probability'' kind (see \eqref{eq:probabilistic eb}) and relates the distance to stationarity in terms
of the difference between any iterate and its computable image under a sampling-based mapping.  More generally, given any $\wh{x} \in X$, with $\overline{V}_N(\bullet; \wh{x})$
being a sampling-based majorizing approximation  with the i.i.d. sample sets $\{\bar \xi^{\,t}\}_{t \in [N]}$ and $\{\bar \eta^{\, s}\}_{s \in [N]}$ and
$\wt{V}^\rho_N(x; \wh{x}) \triangleq  \overline{V}_N(x; \wh{x}) + \displaystyle{\frac{1}{2 \rho}} \, \norm{ x- \wh x }^2$, we have
$\mathbb{P}\Big( \,\mbox{dist}\Big( \wh{x};S_{X,\Theta}^{\, C} \Big)
	\, \geq \, \wh{\eta} \,
	\big\| \,\wh{x} - {\cal M}_{\wt{V}_N^{\rho}}( \wh{x} ) \,\big\| + \varepsilon
	\, \Big) \, \leq \,
	\alpha $ for any $\varepsilon >0 $, $\alpha \in (0, 1)$, and $N \geq O \displaystyle{\Big(\frac{1}{\varepsilon^{\,2}} \Big( \log \Big(\frac{1}{\varepsilon} \Big) + \log \Big(\frac{1}{\alpha} \Big) \Big) \Big)}$; see Theorem~\ref{th:probabilistic eb}.  By letting $\wh{x}$ be an iterate $x^{\nu}$ of the SMM algorithm,
this result gives a lower bound
of the sample size $N$ in defining the computable mapping  $\mathcal M_{\wt{V}_N^\rho}$ such that, roughly speaking,  a lower bound for the probability of $\mbox{dist}\Big( x^{\nu};S_{X,\Theta}^{\, C} \Big)$ not exceeding a prescribed
upper deviation of the difference $\| x^{\nu} - \mathcal M_{\wt{V}_N^\rho}(x^{\nu}) \|$ is guaranteed.
% treating the iterative points $\{x^\nu\}$ in the SMM algorithm as fixed points, this probabilistic error bound can be applied to the SMM iterates
% by taking into account the randomness  that defines
% the   map ${\cal M}_{\wt{V}_N^{\rho}}( \wh{x} ) $ with requirement of the sample size $N$ of  the newly generated sample set.
One may also be interested in a more direct error bound of stationarity of the obtained iterate $x^{\nu}$ in terms of the sample
size $N_{\nu}$, bypassing the residual.  Such a bound would evaluate
how close $x^{\nu}$ is to the set of Clarke stationary points in terms of the iteration number $\nu$ and the
total number of utilized samples  $N_{\nu}$. %This type of improved error bound result is related to the convergence rate analysis and could possibly be established under the KL condition
\cite{cui2018composite}. % applied to bound the difference $\| x^{\nu} - x^{\nu+1} \|$. 
 Regrettably, this type of improved error bound result is beyond the scope of the present paper, and we have to defer this task to a future work.}
\gap

To achieve the aforementioned probabilistic error bound, we postulate the following condition as an extension of the condition (iii) in  (SR1):
	
	\gap

	{\bf (SR$_{\rm st}$)} in addition to
	(A3$_G$) and (A3$_F$), there exists a constant $\kappa_0 > 0$ such that for all
	$(x^{\, \prime},\xi) \in X \times \Xi$, for any $\wh{G}(\bullet, \xi; x^{\, \prime}) \in \mathcal{G}(x^{\, \prime}, \xi)$  and $\wh{F}(\bullet, \xi; x^{\, \prime}) \in \mathcal{F}(x^{\, \prime}, \xi)$,
	\[
	G(x,\xi) + \displaystyle{
		\frac{\kappa_0}{2}
	} \, \| \, x - x^{\, \prime} \, \|^2 \, \geq \, \wh{G}(x,\xi;x^{\, \prime})
	\epc \mbox{and} \epc
	F(x,\xi) + \displaystyle{
		\frac{\kappa_0}{2}
	} \, \| \, x - x^{\, \prime} \, \|^2 \, \geq \, \wh{F}(x,\xi;x^{\, \prime}), \quad \mbox{for any }  x  \in X.
	\]

The above assumption is satisfied for the three types of component functions $F$ and $G$ discussed in Section \ref{sec:stationarity} under some Lipschitz gradient properties. For the first type of smooth  functions $G_i(\bullet, \xi)$ and $F_j(\bullet, \xi)$ with uniform Lipschitz gradient $\kappa_0$, it is straightforward to see that  (SR$_{\rm st}$) holds for the surrogate functions following  \eqref{eq:smooth_surrogation}. For the second type of dc functions $G_i(x, \xi)$ and $F_j(x, \xi)$  following \eqref{eq:dc_smooth}  with smooth concave parts $h_i^G(x, \xi)$ and  $h_i^F(x, \xi)$ with uniform Lipschitz gradient, a simple calculation can show that the above assumption holds with surrogate functions  \eqref{eq:linearized dc}.  Furthermore, when $G_i(x, \xi)$ and $F_j(x,\xi)$ are dc functions with deterministic max-smooth concave parts following the definition \eqref{eq:max dc}, by constructing the surrogation functions $\wh{G}_{i}^{\varepsilon}(x, \xi; x^{\, \prime})$ and $\wh{F}_{j}^\varepsilon (x, \xi; x^{\, \prime})$ following \eqref{eq:max linearized dc}, we can show that the assumption (SR$_{\rm st}$) holds with some constant $\kappa_0$ when the component functions in the concave parts $\{h_{i,k}^G(x)\}_{k=1,\ldots, K_G}$ and $\{h_{i,k}^F(x)\}_{k=1,\ldots, K_F}$ have the uniform Lipschitz gradient.

	\begin{proposition}
		\rm
		Suppose $G(\bullet, \bullet) : X \times \Xi \to \mathbb{R}^{\ell_G}$, $F(\bullet, \bullet) : X \times \Xi \to \mathbb{R}^{\ell_F}$ are two vector-valued functions, where each element $G_i(x, \xi)$ and $F_j(x,\xi)$ are dc functions with deterministic max-smooth concave parts following the definition \eqref{eq:max dc}. Each surrogation family $\mathcal{G}(x^{\, \prime}, \xi)$ and $\mathcal{F}(x^{\, \prime}, \xi)$ contains a single  surrogation function $\wh{G}^{\varepsilon}(x, \xi; x^{\, \prime})$ and $\wh{F}^\varepsilon (x, \xi; x^{\, \prime})$ respectively following \eqref{eq:max linearized dc} for any $x^{\, \prime} \in X$ and $\xi \in \Xi$. Suppose that  the component functions in the concave parts $\{h_{i,k}^G(x)\}_{\nu=1}^{K_G}$ and $\{h_{i,k}^F(x)\}_{\nu=1}^{K_F}$ have the uniform Lipschitz gradient, then the assumption (SR$_{\rm st}$) holds.
	\end{proposition}
	\begin{proof}
		By the convexity and the Lipschitz gradient property with a constant $\kappa_1$, we have
		\[
		h^G_i(x^{\, \prime}) + \nabla h_i^G(x^{\, \prime})^\top (x-x^{\, \prime}) + \displaystyle{\frac{\kappa_1}{2}} \, \norm{ x- x^{\, \prime}}^2 \, \geq  \, h_i^G(x) \, \geq \,  h_i^G(x^{\, \prime}) + \nabla h_i^G(x^{\, \prime})^\top (x-x^{\, \prime}).
		\]
		Then $
		\wh{G}_i^{\varepsilon}(x, \xi; x^{\, \prime}) - \displaystyle{\frac{\kappa_1}{2}} \, \norm{x- x^{\, \prime}}^2  \, \leq  \, g_i^G(x, \xi) - \displaystyle{\max\left\{ \, h^G_{i, k}(x): {k \in \mathcal{A}_i^{G, \varepsilon}(x^{\, \prime})} \right\}}  \, \leq \,  \wh{G}_i^{\varepsilon}(x, \xi; x^{\, \prime}) $.
		Let $\mathcal{A}_i^G(x) \triangleq  \left\{ k \mid  h_{i, k}^G(x) \geq   \max\, \left\{ h_{i, k^{\,\prime}} ^G(x): {1 \leq k^{\, \prime} \leq K_G} \right\}  \right\}$.  For an arbitrary $x^{\, \prime} \in X$, for any positive $\varepsilon$, $ \exists \,  \delta(x^{\, \prime}, \varepsilon)$ such that $ \forall x\in X \cap \mathbb{B}(x^{\, \prime}; \, \delta(x^{\, \prime}, \varepsilon)),$ we have $\mathcal{A}_i^G(x) \subseteq \mathcal{A}_i^G(x^{\, \prime}) \subseteq \mathcal{A}_i^{G, \varepsilon}(x^{\, \prime})$. This indicates that $ \forall x\in X \cap \mathbb{B}(x^{\, \prime}; \, \delta(x^{\, \prime}, \varepsilon)),$
		\begin{equation}
		\label{eq:dc_surrogation_upper}
		\displaystyle{\max_{k \in \mathcal{A}_i^{G, \varepsilon}(x^{\, \prime})}} \, h^G_{i, k}(x)   = \displaystyle{\max_{k \in \mathcal{A}_i^{G}(x^{\, \prime})}} \, h^G_{i, k}(x), \quad \mbox{and} \quad  \wh{G}_i^{\varepsilon} (x, \xi; x^{\, \prime}) \leq G_i(x, \xi) + \displaystyle{\frac{\kappa_1}{2}} \norm{x- x^{\, \prime}}^2.
		\end{equation}
		We next extend this local condition into a global one as in (SR$_{\rm st}$). Under assumptions (A4) and (A4$_{\rm st}$), both $\| G \| $ and $\| \wh{G}^\varepsilon \|$ are upper bounded for any $x, x^{\, \prime} \in X$ and $\xi \in \Xi$.  Hence, for any $x$ in the set $X \setminus  \mathbb{B}(x^{\, \prime}; \delta(x^{\, \prime}, \varepsilon))$, there exists $\kappa_2 >0$ such that $\wh{G}_i^{\varepsilon} (x, \xi; x^{\, \prime})  \leq G_i(x, \xi) + \displaystyle{\frac{\kappa_2}{2}} \norm{ x -x^{\, \prime}}^2$. By combining with \eqref{eq:dc_surrogation_upper}, the condition (SR$_{\rm st}$) holds with $\kappa_0 = \max\{ \kappa_1, \kappa_2\}$.
 
\end{proof}
\gap
	
	The assumption (SR$_{\rm st}$) with the constant $\kappa_0$, together with (A2), implies that all functions  $\wh{V}(\bullet; x^{\, \prime} ) \in \mathcal{V}(x^{\, \prime}  )$ are well defined and  there exists a constant $\kappa \triangleq \mbox{Lip}_{\psi} \,  \mbox{Lip}_{\varphi} \, \kappa_0 $, such that
	\begin{equation} \label{eq:stochastic other inequality}
	\Theta(x) + \displaystyle{
		\frac{\kappa}{2}
	} \, \| \, x - x^{\, \prime} \, \|^2 \, \geq \, \wh{V}(x;x^{\, \prime}),
	\epc \forall \, x \, \in \, X.
	\end{equation}

	% Moreover, in order for the compound expectation function $\wh{V}(x;x^{\, \prime})$ to be well defined, we postulate the following integrability condition:
	
	% {\bf (A6$_{\rm eb}$)} $\wh{G}(x,\bullet;x^{\, \prime})$ and  $\wh{F}(x,\bullet;x^{\, \prime})$ are both integrable for all  $(x,x^{\, \prime}) \in X \times X$,

	% which extends the pointwise assumption (A6$_{\infty}$) to all $x^{\, \prime} \in X$.
	
	% With $\wh{G}(\bullet, \bullet; x^{\, \prime}) \in \mathcal{G}(x^{\, \prime})$ and
	% $\wh{F}(\bullet, \bullet; x^{\, \prime}) \in \mathcal{F}(x^{\, \prime})$,
	% the convex majorant ${V}(\bullet; x^{\, \prime})$ and the discretized convex majorant
	% $\wh{V}_N(\bullet; x^{\, \prime})$  of $\Theta(x)$ are constructed according
	% to \eqref{eq:Theta surrogate}  and \eqref{eq:saa_mm_approximation} respectively.
	\gap
	
	For any member $\wh{V} (\bullet;x^{\, \prime}) \in
	\mathcal{\wh{V}}(x^{\, \prime})$ and any scalar $\rho > 0$, we continue to write
	$\wt{V}^{\rho}(x;x^{\, \prime}) \triangleq  \wh{V}(x;x^{\, \prime}) +
	\displaystyle{
		\frac{1}{2\rho}
	} \norm{x - x^{\, \prime}}^2$.
	Although being a deterministic quantity, the residual based on the corresponding
	fixed-point map $\mathcal{M}_{\wt{V}_\rho}$  involves compound
	expectations, thus is not practical for use as a stopping rule for the
	SMM algorithm.   Let $\overline{\mathcal{V}}_{N}(x^{\, \prime})$ be the family of discretized convex
	majorants \eqref{eq:saa_mm_family_approximation} associated with newly generated sample sets
	$\{ \bar \xi^t \}_{t \in [N]}$ and $\{ \bar \eta^s \}_{s \in [N]}$, for any member $\overline{V}_N(\bullet;x^{\, \prime}) \in
	\mathcal{\overline{V}}_N(x^{\, \prime})$ and any scalar $\rho > 0$, with $\wt{V}_N^{\rho}(x;x^{\, \prime}) \triangleq  \overline{V}_N(x;x^{\, \prime}) +
	\displaystyle{
		\frac{1}{2\rho}
	} \norm{x - x^{\, \prime}}^2$, we are interested in utilizing the sampling-based
	fixed-point map  $\mathcal{M}_{\wt{V}_N^\rho}$   to construct a
	computable sample-based residual function and establish a corresponding error bound
	of stationarity; such a result must necessarily be
	probabilistic  to take into account the randomness  that defines the
	 sampling-based  map  $\mathcal{M}_{\wt{V}_N^\rho}$. 
 \gap
	
	Further, we want to extend the pointwise conditions
	(A4)  for surrogate functions so that we can apply the error bound to
	an arbitrary test vector.
	We postulate the Lipschitz continuity condition in  (A5$_{\rm {st}}$) for the SAA convergence rate in probability established in \cite{hu2020sample}, which will be  utilized in the probabilistic error bound result.
	 \gap 
	
\textbf{(A4$_{\rm st}$)}
	$  \displaystyle{
		\sup_{x,x^{\,\prime}  \in X } }  \, \max\left\{    
	 \mathbb{E}    \norm{ \, \wh{F}(x, \tilde{\xi}; x^{\,\prime}) - \mathbb{E}\,[\,\wh{F}(x, \tilde{\xi}; x^{\,\prime})\,]\, }^2,    
   \mathbb{E}   \norm{ \, \wh{G}(x, \tilde{\xi}; x^{\,\prime}) - \mathbb{E}\,[\,\wh{G}(x, \tilde{\xi}; x^{\,\prime})\,]\, }^2 \, \right\}  \leq \wh \sigma^2$.

	\begin{comment}
	{\bf (A4$_{\rm st}$)}  $  \,
	\sup\ \left\{ \,  \norm{\wh{G}(x,\xi;x^{\, \prime})} :{( x,\xi,x^{\, \prime} ) \in X \times \Xi \times X}, \,  \wh{G}(\bullet, \xi; x^{\, \prime}) \in \mathcal{G}(x^{\, \prime}, \xi) \right\} \, < \, \infty$;\\
	
	\quad  \qquad  \ $  \,
	\sup\ \left\{ \,  \norm{\wh{F}(x,\xi;x^{\, \prime})} :{( x,\xi,x^{\, \prime} ) \in X \times \Xi \times X}, \,  \wh{F}(\bullet, \xi; x^{\, \prime}) \in \mathcal{F}(x^{\, \prime}, \xi) \right\} \, < \, \infty.$
	\end{comment}
	\gap
	
	{\bf (A5$_{\rm st}$)} $\wh{G}(\bullet,\xi;x^{\, \prime})$ and
	$\wh{F}(\bullet,\xi;x^{\, \prime})$
	are both Lipschitz continuous on $X$, uniformly in
	$(\xi,x^{\, \prime}) \in \Xi \times X$ for any $\wh{G}(\bullet, \xi; x^{\, \prime}) \in \mathcal{G}(x^{\, \prime}, \xi) $ and $\wh{F}(\bullet, \xi; x^{\, \prime}) \in \mathcal{F}(x^{\, \prime}, \xi)$.
	
	\gap
	
	Finally, we postulate that the condition (LEB) holds
	for the pair $(\Theta,X)$ with constant pair  $(\varepsilon, \eta)$ and that (SR$_{\rm st}$) and (SR2) hold for any surrogate $\wh{V} \in \mathcal{V}$ with the constant $\kappa_0$.
	As a deterministic optimization problem and with the latter deterministic surrogate,
	Theorem~\ref{th:error bound} can be applied to the problem
	(\ref{eq:compound SP repeated}) to yield the existence of a scalar $$\wh{\eta} \triangleq   \max\left\{ \, 2 + \eta \,
	\left( \, \displaystyle{
		\frac{2}{\rho}
	} + \displaystyle{
		\frac{5 \, \mbox{Lip}_{\psi} \mbox{Lip}_{\varphi}\kappa_0}{2}
	} \, \right), \displaystyle{\frac{2D}{\varepsilon}} \left( \, \displaystyle{
		\frac{2}{\rho}
	} + \displaystyle{
		\frac{5 \, \mbox{Lip}_{\psi} \mbox{Lip}_{\varphi} \kappa_0}{2}
	} \, \right)  \right\}  $$
	such that
	\begin{equation} \label{eq:error bound repeated}
	\mbox{dist}\left( \wh{x};S_{X,\Theta}^{\, C} \right) \, \leq \, \wh{\eta} \, \,
	r_{\wt{V}_{\rho}}( \wh{x} ), \epc \mbox{for all } \wh{x} \, \in \, X.
	\end{equation}
	The next step is to relate the deterministic residual
	$r_{\wt{V}_{\rho}}( \wh{x} ) = \norm{ \wh{x} - {\cal M}_{\wt{V}_{\rho}}( \wh{x})}$
	to the sampled residual
	$  \norm{ \wh{x} - {\cal M}_{\wt{V}_N^{\rho}}( \wh{x} )}$
	corresponding new  sample sets
	$\{ \bar \xi^{\, t} \}_{t \in [N]}$ and $\{ \bar \eta^{\, s} \}_{s \in [N]}$.  By the triangle
	inequality, we have
	\begin{equation} \label{eq:triangle inequality}
	r_{\wt{V}_{\rho}}( \wh{x} ) = \norm{ \wh{x} - {\cal M}_{\wt{V}_{\rho}}( \wh{x} )} \, \leq \,
	\norm{ \wh{x} - {\cal M}_{\wt{V}_N^{\rho}}( \wh{x} )} +
	\norm{ {\cal M}_{\wt{V}_{\rho}}( \wh{x} ) - {\cal M}_{\wt{V}_N^{\rho}}( \wh{x} )}.
	\end{equation}
	Therefore, it remains to bound the second term on the right-hand side.  Noticing that
	${\cal M}_{\wt{V}_{\rho}}( \wh{x} )$ and ${\cal M}_{\wt{V}_N^{\rho}}( \wh{x} )$
	are, respectively, the unique minimizers of the two strongly convex programs:
	\[
	\displaystyle{
		\operatornamewithlimits{\mbox{minimize}}_{x \in X}
	} \ \left[ \, \wh{V}(x;\wh{x} ) + \displaystyle{
		\frac{1}{2 \, \rho}
	} \, \norm{ x - \wh{x} }^2 \, \right] \epc \mbox{and} \epc
	\displaystyle{
		\operatornamewithlimits{\mbox{minimize}}_{x \in X}
	} \ \left[ \, \overline{V}_N(x;\wh{x} ) + \displaystyle{
		\frac{1}{2 \, \rho}
	} \, \norm{ x - \wh{x} }^2 \, \right],
	\]
	we may apply the following simple result which is fairly standard for a strongly
	convex program.  For completeness, we provide a proof.
	
	\begin{lemma} \label{lm:strongly convex program} \rm
		Let $f$ be a strongly convex function with modulus $\zeta > 0$ on the compact convex set
		$X \subseteq \mathbb{R}^n$ and let $\bar{x}$ be the unique minimizer of $f$ on $X$.
		% there exists a constant $\zeta > 0$, dependent only on the modulus of strong convexity
		% of $f$, such that
		For any function $g$, if $x^g$ is a minimizer of $g$ on $X$, then
		$\norm{ \bar{x} - x^g } \, \leq \, \displaystyle{
			\frac{2}{\sqrt{\zeta}}
		} \, \sqrt{\max_{x \in X} \, | \, f(x) - g(x) \, |}
		$.
	\end{lemma}
	
	\begin{proof}
		By the strong convexity of $f$, it follows that $f(x) - f(\bar{x}) \geq \, \displaystyle{
			\frac{\zeta}{2}
		}  \, \norm{ x - \bar{x} }^2$  for all $x \, \in \, X$.  Thus, for $x^g$ as given in the lemma, we have
		\[ \begin{array}{lll}
		\displaystyle{
			\frac{\zeta}{2}
		}  \,  \, \norm{ x^g - \bar{x} }^2 & \leq & f(x^g) - f(\bar{x})
		\, = \, \left[ \, f(x^g) - g(x^g)\, \right] + \left[ \, g(x^g) - g(\bar{x}) \, \right]
		+ \left[ \, g(\bar{x}) - f(\bar{x}) \, \right] \\ [0.1in]
		& \leq & 2 \, \displaystyle{
			\max_{x \in X}
		} \, | \, f(x) - g(x) \, | \epc \mbox{because $g(x^g) \leq g(\bar{x})$}.
		\end{array}
		\]
		Thus the desired inequality follows readily.
		 
		\end{proof}
		\gap
	
	Since the function $\wh{V}(\bullet;\wh{x}) + \displaystyle{
		\frac{1}{2 \, \rho}
	} \, \| \, \bullet - \wh{x} \, \|^2$ is strongly convex with the modulus
	$\rho^{-1}$ that is independent of
	$\wh{x}$, applying
	Lemma~\ref{lm:strongly convex program} to the inequality (\ref{eq:triangle inequality})
	yields  that for all $\wh{x} \in X$,
	\[
	\norm{ \wh{x} - {\cal M}_{\wt{V}_{\rho}}( \wh{x} )} \, \leq \,
	\norm{ \wh{x} - {\cal M}_{\wt{V}_N^{\rho}}( \wh{x} )} + 2 \, \sqrt{\rho} \,
	\sqrt{\displaystyle{
			\max_{x \in X}
		} \, | \, \wh{V}(x;\wh{x}) - \overline{V}_N(x;\wh{x}) \, |}\;.
	\]
	Substituting this into (\ref{eq:error bound repeated}) we therefore deduce %with $ \wh{\eta} = 2 + \eta \, \left( \, \displaystyle{	\frac{2}{\rho}	} + \displaystyle{	\frac{5 \,   \mbox{Lip}_{\psi} \,  \mbox{Lip}_{\varphi} \, \kappa_0 }{2}	} \, \right) $,
	\begin{equation} \label{eq:two-terms error bound}
	\mbox{dist}\left( \wh{x};S_{X,\Theta}^{\, C} \right) \, \leq \, \wh{\eta} \,
	\norm{ \wh{x} - {\cal M}_{\wt{V}_N^{\rho}}( \wh{x} )} + 2 \, \sqrt{\rho} \, \wh{\eta}  \,
	\sqrt{\displaystyle{
			\max_{x \in X}
		} \, | \, \wh{V}(x;\wh{x}) - \overline{V}_N(x;\wh{x}) \, |}\;.
	\end{equation}
	Notice that $\displaystyle{
		\max_{x \in X}
	} \, | \, \wh{V}(x;\wh{x}) - \overline{V}_N(x;\wh{x}) \, |$ is the  uniform upper bound of the sample average
	approximation error for the compound expectation function $\wh{V}(x; \wh{x})$ on $X$.
	{Under assumptions (A4$_{\rm st}$) and (A5$_{\rm st}$), we may apply
	Corollary 5.1 in \cite{hu2020sample} to deduce the existence of  $\bar{\varepsilon}>0$ and two positive constants
	$C_2$ and $C_3$ which depend on the finite variance $\wh \sigma$, the Lipschitz continuity modulus $\mbox{Lip}_\varphi$, $\mbox{Lip}_\psi$
and the diameter of the feasible set $X$,  such that for all $\wh{x} \in X$, and any $\varepsilon \in (0, \bar{\varepsilon})$ and any $\alpha \in (0,1)$,
if the sample size satisfies
$N \geq C_2\,  \Big( \,  \displaystyle{\frac{2 \,  {\rho} \, \wh{\eta}^{\,2}}{\varepsilon^{\,2}}} \Big( n \log\Big(\, \frac{2 \, \sqrt{\rho} \, \wh{\eta} \,   C_3}{\varepsilon}\,\Big) + \log\Big(\frac{1}{\alpha} \Big)\Big) \Big) $, with
	  i.i.d.\ samples $\{ \bar \xi^{\, t} \}_{t \in [N]}$ and $\{ \bar \eta^{\, s} \}_{s \in [N]}$,
	  	\[
	  \mathbb{P}\left( \, \displaystyle{
	  	\max_{x \in X}
	  } \, | \, \wh{V}(x;\wh{x}) - \overline{V}_N(x;\wh{x})  \, |\, \geq \, \frac{\varepsilon}{2 \, \sqrt{\rho} \, \wh{\eta}} \, \right) \,   \leq  \,  \alpha  .
	  \]
	Combining the above inequality with \eqref{eq:two-terms error bound}, %  and choosing $t = N^{\, \gamma/4}$,
	we obtain  that
	\begin{equation} \label{eq:probabilistic eb}
	\mathbb{P}\left( \,\mbox{dist}\left( \wh{x};S_{X,\Theta}^{\, C} \right)
	\, \geq \, \wh{\eta} \,
	\norm{ \wh{x} - {\cal M}_{\wt{V}_N^{\rho}}( \wh{x} )} + \varepsilon
 \, \right) \, \leq \,
 \alpha  .
	\end{equation}
	We summarize the above derivations in the theorem below.}
	
	\begin{comment}
	\[
	\mathbb{P}\left( \, \displaystyle{
	\max_{x \in X}
	} \, | \, \wh{V}(x;\wh{x}) - \wh{V}_N(x;\wh{x}) \, |\, \geq \, \displaystyle{
	\frac{B +  {t}}{N^{\, 1/2}}} \, \right) \,   \leq  \, C_2 \, \mbox{exp} \,
	\left\{ - \displaystyle{\frac{t^2}{   C_3}} \right\}.
	\]
	\end{comment}
	%\[
	%\mathbb{P}\left( \,\mbox{dist}\left( \wh{x};S_{X, \Theta}^{\, C} \right) - \wh{\eta} \,
	%\norm{ \wh{x} - {\cal M}_{\wt{V}_N^{\rho}}( \wh{x} )} \, \geq \, \displaystyle{
	%\frac{\wh{\zeta} \, ( B + \sqrt{t} )}{N^{\, \gamma/2}}
	%} \, \right) \, \leq \,
	%2 \, \mbox{exp} \, \left\{ - \displaystyle{\frac{t^2}{ 2 \, C}} \right\}.
	%\]
	%By choosing $t = N^{\, \gamma/4}$, we deduce, for some constant $B^{\, \prime} > 0$,
	
\begin{theorem} \label{th:probabilistic eb} \rm
Suppose assumptions (A1)--(A3), (A4$_{\rm st}$) , (A5$_{\rm st}$),  (LEB) with the constant pair $(\varepsilon, \eta)$,  (SR$_{\rm st}$) with the modulus $\kappa_0$  and (SR2) hold for the compound SP \eqref{eq:compound SP repeated}. Then for every $\rho > 0$, with $\wh{\eta} \triangleq   \max\left\{ \, 2 + \eta \,
\left( \, \displaystyle{
	\frac{2}{\rho}
} + \displaystyle{
	\frac{5 \, \mbox{Lip}_{\psi} \mbox{Lip}_{\varphi}\kappa_0}{2}
} \, \right), \displaystyle{\frac{2D}{\varepsilon}} \left( \, \displaystyle{
	\frac{2}{\rho}
} + \displaystyle{
	\frac{5 \, \mbox{Lip}_{\psi} \mbox{Lip}_{\varphi} \kappa_0}{2}
} \, \right)  \right\}  $, there exist positive constants $\bar{\varepsilon}$,
$C_2$ and $C_3$,  such that for all $\wh{x} \in X$ and any $\varepsilon \in (0, \bar{\varepsilon})$ and $\alpha \in (0,1)$,
provided that the sample size  satisfies
$N \geq C_2\,  \Big( \,  \displaystyle{\frac{2 \,  {\rho} \, \wh{\eta}^{\,2}}{\varepsilon^{\,2}}} \Big( n \log\Big(\, \frac{2 \, \sqrt{\rho} \, \wh{\eta} \,   C_3}{\varepsilon}\,\Big) + \log\Big(\frac{1}{\alpha} \Big)\Big) \Big) $,
the probabilistic bound (\ref{eq:probabilistic eb}) holds with i.i.d.\ samples $\{ \bar \xi^{\, t} \}_{t \in [N]}$ and $\{ \bar \eta^{\, s} \}_{s \in [N]}$.  \hfill $\Box$	
	\end{theorem}
	
This suggests that in the SMM algorithm, by treating $x^\nu$ as a given point,
the distance between the current iterate $x^{\nu}$ and the algorithmic map ${\cal M}_{\wt{V}_N^{\rho}}( x^{\nu} ) $ with the new set of i.i.d. generated samples can be adopted as a stopping rule to guarantee the asymptotic stationarity of the iterative point with high probability when the sample size  $N$ is large.
	
	%%%%%%%%%%%%%%%%%%%%%%%%%%%%%%%%%%%%%%%%%%%%%%%%%%%%%%
	\section{Applications to Risk Measure Minimization}
	\label{sec:application}
	In all the problems discussed in this section, there is a bivariate loss function
	$f(x,\xi)$ defined on $\mathbb{R}^n \times \Xi$ that is of the Carath\'eodory kind.
	%, which depends on the
	%deterministic decision variable $x$
	%and the realization $\xi$ of a random variable $\tilde{\xi}: \Omega \to \Xi \subseteq \mathbb{R}^m$ on the probability space $(\Omega, \mathcal{A}, \mathbb{P})$.
	The random variable $\tilde{\xi}$ here may represent  the rates of returns of
	some financial assets.  This loss function is composed with multiple types of risk measures;
	the resulting composite function is then the overall objective  to be minimized.
	%
	%i.e., $f(x, \bullet)$ is measurable for every $x \in X$ and $f(\bullet,\xi)$ is a
	%continuous function for almost all $\xi \in \Xi$.
	We present three classes of
	decision-making problems under uncertainty and show how they can be formulated as
	special instances of
	\eqref{eq:compound_composite_SP}.
	%These problems are: generalized deviation
	%optimization, a distributionally robust bPOE optimization, and bPOE based
	%multi-class misclassification.
	
	\subsection{Risk measures: (C)VaR, OCE, and (b)POE}
	
	% Risk measures quantify the uncertainty by a function of the random variable
	% which represents the attitudes of decision makers towards risk.
	% First proposed in  % \cite{kijima1993mean}.
	% \cite{ArtznerDelbaenEberHeath1999},
	% coherent risk measures have been studied and employed extensively in
	% financial engineering and
	% there is a vast literature on the subject.  In general, a coherent risk measure is a
	% subadditive, monotonically increasing, positively homogenous, and translation invariant
	% function.
	We focus on three risk measures: the conditional value-at-risk (CVaR);
	the optimized certainty equivalent (OCE), and the buffered probability of
	exceedance (bPOE).  Let $Z:  \Omega \to  \mathbb{R}$  be a scalar random variable on the probability space
	$(\Omega, \mathcal{A}, \mathbb{P})$ with cumulative distribution function $F_Z$.
	For a given scalar $\alpha \in (0,1)$, the $\alpha$-Value-at-Risk (VaR) is defined as
	$\mathrm{VaR}_{\alpha}(Z) \triangleq \min\{z : F_Z(z ) \geq \alpha \}$
	while the $\alpha$-Upper Value-at-Risk  of $Z$
	is defined as $\mathrm{VaR}^+_{\alpha}(Z) \triangleq \min\{z : F_Z(z) > \alpha \}$.
	The $\alpha$-CVaR is
	$\mathrm{CVaR}_{\alpha}(Z ) \triangleq \displaystyle{
		\frac{1}{1-\alpha}
	} \, \int_{z \geq \mathrm{VaR}_{\alpha}(Z )}  z\,
	d (F_Z(z ) -\alpha)$.
	For a continuous random variable $Z$, we have
	$\mathrm{CVaR}_{\alpha}(Z) =
	\mathbb{E}\left[ \, Z \, \mid \, Z \geq \mathrm{VaR}_{\alpha}(Z ) \, \right]$.
	It is known that CVaR can be represented as the optimal value of a scalar convex
	optimization problem; i.e.,
	$
	\mathrm{CVaR}_{\alpha}(Z)  \, = \, \displaystyle\min_{\eta \in \mathbb{R}}\left\{ \, \eta + \displaystyle{
		\frac{1}{1-\alpha}
	} \, \mathbb{E}\left[ \, Z-\eta \, \right]_+ \, \right\}.
	$
	Predating the CVaR and generalizing it, the  OCE is a decision-theoretic risk measure that is defined relative
	to a utility function.  By choosing appropriate utility functions, it is shown
	in \cite{ben2007old} that the negative of the OCE covers a wide
	family of convex risk measures.  For instance, the negative OCE of a certain
	piecewise linear utility function reduces to the CVaR. The formal definition of OCE
	is given below.
	
	\begin{definition} \label{df:OCE} \rm
		Let $u: \mathbb{R} \to [-\infty, \infty)$ be a proper closed concave and
		nondecreasing utility function with effective domain
		$\mathrm{dom}(u) = \{ t  \in \mathbb{R}: u(t) > -\infty\}  \neq \emptyset$.
		Assume further that $u(0)=0$ and $1 \in \partial u(0)$.
		The Optimized Certainty Equivalent (OCE) of the uncertain outcome $Z$ is: % defined as
		\begin{equation} \label{eq:oce_formula}
		S_u(Z) \, \triangleq \, \sup_{\eta \in \mathbb{R}} \, \left\{ \, \eta +
		\mathbb{E}[ u (Z-\eta)] \, \right\}.
		\end{equation}
	\end{definition}
	
	The following lemma \cite[Proposition 2.1]{ben2007old} states that whenever
	there is an optimal solution to the problem \eqref{eq:oce_formula}, there is an
	optimal solution in the support of the random variable $Z$.
	
	\begin{lemma} \label{le:oce_solution} \rm
		Let $Z$ be a scalar random variable with the support interval
		$[z_{\min}, z_{\max}]$.
		% that is
		% $Z(\omega)$ belongs to this interval for almost all $\omega$.
		% Let $S_u(\tilde{\xi} )$ be the OCE of $Z$ associated with the utility function $u$ satisfying
		% the definition of the OCE.
		If the supremum in $S_u(Z )$ is attained, then
		$
		S_u(Z) \, = \, \underset{\eta \in [z_{\min}, \, z_{\max}]}{\sup} \,
		\left\{ \, \eta + \mathbb{E} u(Z-\eta) \, \right\}.
		$
		Furthermore, if $u$ is strictly concave and $Z$ is a continuous random variable,
		then the supremum  is uniquely attained in
		$[z_{\min},z_{\max}]$.
	\end{lemma}
	
	We refer the reader to \cite{MaherPangRaza2018} for a demonstration of the
	difference-convexity of the function $S_u(f(\bullet,\tilde{\xi} ))$ when
	$f(\bullet,\xi)$ is a difference-of-convex function for almost all $\xi \in \Xi$.
	Another natural risk measures are the Probability of Exceedance (POE) and its buffered
	extension.  Both of these measures have received growing attention in risk applications
	\cite{RockafellarRoysett10,rockafellar2000optimizationCVAR,rockafellar2002conditional,
		norton2017error,norton2019maximization,mafusalov2018buffered,
		mafusalov2018estimation}.
	%  In an application to the design and optimization of structures, Rockafellar and
	% Royset \cite{RockafellarRoysett10} employed the buffered failure probability as a
	% measure of the structural reliability to be minimized and they related it to the CVaR
	% (\cite{rockafellar2000optimizationCVAR,rockafellar2002conditional}).  The concept
	% of buffered probability of exceedance (bPOE) of a random variable was formally
	% defined and further studied by Uryasev and his collaborators in a series of
	% publications \cite{norton2017error,norton2019maximization,mafusalov2018buffered,
	% mafusalov2018estimation}.
	Being the inverse
	of VaR, the POE quantifies the probability that an uncertainty exceeds a threshold.
	In the same manner that 1 - POE is the inverse of VaR, the buffered probability
	of exceedance is proposed  as the counterpart of CVaR; i.e.,
	1 - bPOE is the inverse of CVaR.  These two probability concepts: POE and bPOE,
	are formally defined as follows.  We let $\sup(Z ) \triangleq \displaystyle{
		\sup_{\omega \in \Omega}
	} \,  Z(\omega) $ denote the
	essential supremum of a scalar-valued random variable $Z$.
	
	\begin{definition} \label{df: bPOE} \rm
		The \emph{probability of exceedance} at the threshold $\tau \in \mathbb{R}$
		of a random variable $Z$ with cumulative distribution function $F_Z$ is
		\[
		\mathrm{POE}(Z; \tau) \, \triangleq \, \mathbb{P} ( Z > \tau) \,  = \, 1 - F_Z(\tau).
		\]
		The \emph{upper buffered probability of exceedance} at the same threshold $\tau$ is
		% $\mathrm{(bPOE)}$ is defined as
		\[
		\begin{array}{lll}
		\mathrm{bPOE}^+(Z; \,\tau) & \triangleq & \begin{cases}
		1 - \min\left\{ \, \alpha \in ( 0,1 ) \, : \,
		\mathrm{CVaR}_\alpha(Z) \geq \tau \, \right\} &  \mbox{if } \tau \leq  \sup(Z) \\
		0 & \mbox{if } \tau > \sup(Z).
		\end{cases}
		\end{array}
		\]
	\end{definition}
	
	As a function of $\tau$,  $\mathrm{bPOE}^+(Z;\bullet)$
	is a nonincreasing function on $\mathbb{R}$ with no more than one discontinuous point,
	which is $\sup(Z)$ if $\mathbb{P}(Z = \sup(Z)) \neq 0$.   Besides upper bPOE,
	lower bPOE is defined in \cite{mafusalov2018buffered} which differs from upper bPOE
	only when $\tau = \sup(Z)$.  In the present paper, we focus on the upper bPOE as the
	risk measure, so we use $\mathrm{bPOE}(Z; \tau)$ as a simplification to represent
	the upper bPOE. % We summarize some properties of bPOE here and refer to
	% the cited reference for their proofs and further reading.  For any $\tau$, we have
	% $\mbox{bPOE} (Z; \tau)  \geq \mbox{POE}(Z; \tau)$.
	% The function $\mbox{bPOE}(\bullet;\tau)$ is a quasiconvex function of the
	% random variable; thus in particular, with the
	% variable $\tilde{\xi}: \Omega \to \Xi \subseteq \mathbb{R}^n$ on the probability space $(\Omega, \mathcal{A}, \mathbb{P})$,
	% the level set
	% $\{ x\in \mathbb{R}^n \mid \mbox{bPOE}(x^\top \tilde{\xi};\, \tau) \leq \alpha\}$ is convex for any scalar $\alpha$.
	As discussed in-depth in \cite{mafusalov2018buffered}, it is advantageous to use the bPOE
	rather than the POE as the risk measure because bPOE includes the tail distribution
	of risk;
	moreover, the minimization of the objective $\mbox{bPOE}(x^{\top} \tilde{\xi};\tau)$
	over the (deterministic) variable $x$ can be formulated as a convex stochastic program
	by using the convex optimization formula of bPOE as follows:	
	\begin{equation} \label{eq:bPOE min formula}
	\mathrm{bPOE}(Z; \tau)  = \displaystyle{
		\underset{a \geq 0}{\mbox{min}}
	} \ \mathbb{E} [ a(Z-\tau)+1 ]_+  \,.
	\end{equation}
	For a continuous random variable $Z$, we have
	$\mathrm{bPOE}  (Z; \tau) = \mathbb{P}(Z > q)$ where
	$q$ satisfies the equation
	$\tau = \mathbb{E}\, [ \, Z\mid Z\geq q \, ]$.   It can be easily shown that bPOE is invariant under
	monotonic linear transformation, i.e.,
	$\mathrm{bPOE}(h(X); h(\tau)) = \mathrm{bPOE}(X; \tau)$ for any nondecreasing linear
	function $h$.  So bPOE is not a coherent risk measure.   The following lemma identifies
	the solution set of the optimization formula of the bPOE, showing in particular that it
	is a closed interval. Its detailed proof can be found in
	\cite[Proposition~1]{norton2019maximization}.
	
	\begin{lemma} \rm \label{le:solution_interval}
		Let $Z$ be a random variable % defined on the probability space
		% $(\Omega, \mathcal{A}, \mathbb{P})$
		satisfying $\mathbb{E}[Z] < \sup(Z)$.  For
		$\tau \in \left( \, \mathbb{E}[Z], \, \sup(Z) \, \right)$, the optimal solution set in
		the formula (\ref{eq:bPOE min formula}) is a closed interval
		$\left[ \, \displaystyle{
			\frac{1}{\tau - \mathrm{VaR}^+_{\alpha^*}(Z)}
		}, \, \displaystyle{
			\frac{1}{\tau - \mathrm{VaR}_{\alpha^*}(Z)}
		} \, \right] $ with $\alpha^* = 1- \mathrm{bPOE}(Z; \tau)$.
		For $\tau \leq \mathbb{E}[Z]$, the corresponding optimal solution in bPOE$(Z;\tau)$
		is 0. For $\tau = \sup(Z)$, $\mathrm{bPOE}(Z; \tau) = \displaystyle{
			\lim_{a \to \infty}
		} \, \mathbb{E}[a(Z-\tau)+1]_+$. For $\tau > \sup(\tilde{\xi} )$, the corresponding optimal
		solution set is $\left[ \, \displaystyle{
			\frac{1}{\tau -\sup(\tilde{\xi} )}
		}, \, +\infty \, \right)$.
	\end{lemma}
	
	We should emphasize that  key conclusions of Lemmas~\ref{le:oce_solution} and
	\ref{le:solution_interval} are the existence of bounded optimal
	solutions of the respective optimization problems defining the OCE and bPOE.
	% under some boundedness conditions on the random variable of interest.
	Such boundedness of the optimizing scalars plays an important role in the treatment
	of the applied problems within the framework of the compound SP
	(\ref{eq:compound_composite_SP}) that requires the blanket assumption of the
	boundedness of the feasible set $X$.
	
	\subsection{Generalized deviation optimization}
	
	The variance is the expected squared deviation from the mean.  Generalizing this classic
	deviation, a generalized deviation measure, which was first proposed in
	\cite{rockafellar2006generalized}, quantifies the deviation from the mean by employing
	risk measures.  For instance, the CVaR of deviation (from the mean) is defined as
	$\mathrm{CVaR}_{\alpha}( \, Z-\mathbb{E}[ \, Z \, ] \, )$, which quantifies the conditional expectation
	of the tail distribution of the difference $Z -\mathbb{E}Z$.  It is useful in many
	applications; e.g., in portfolio management problems when the manager has an
	adverse attitude towards the tail of the  shortfall return.
	It can also be connected with parametric statistical learning problems.
	For instance, the excessive or shortfall forecast of demands may lead to different
	levels of costs in inventory management and thus should be weighted asymmetrically
	in demand forecasting.  Replacing the CVaR of deviation, we consider two other
	generalized deviations in the following subsections, using the OCE (that covers
	the CVaR) and bPOE; we also consider a distributionally robust extension of the
	bPOE problem.  While the treatment in each of these cases is similar, we separate their
	discussion due to their independent interests.
	% to illustrate how the family of surrogate functions can be derived under
	% different assumptions of the loss function $f$.
	The applications presented below broaden the applicability of the
	compound SP framework beyond those noted in the
	reference \cite{ermoliev2013sample}.
	
	\subsubsection{OCE-of-deviation optimization}
	\label{sec:OCE-of-deviation optimization}
	Given the loss function $f(x,\xi)$ of the Carath\'eodory kind,
	a random variable $\tilde{\xi}$, and a compact convex
	subset $X$ of $\mathbb{R}^n$, the OCE-of-deviation (from the mean) optimization
	problem is:
	\begin{align} \label{eq:oce_deviation}
	\displaystyle{
		\operatornamewithlimits{\mbox{minimize}}_{x\in X}
	} \ \phi(x) \, \triangleq -S_u\left( \, f(x,\tilde{\xi} ) - \mathbb{E}[f(x,\tilde{\xi} )]
	\, \right).
	\end{align}
	Here and in the next subsection, we make the following additional
	assumptions on the function $f$:
	
	\gap
	
	% \textbf{(B1)} For any $x^{\, \prime} \in X$, a function
	% $\check{f}(\bullet,\xi;x^{\, \prime})$
	% exists such that $-\check{f}(\bullet,\xi; x^{\prime})$ is a majorizing surrogate of
	% $-f(\bullet,\xi)$ at the point $x^{\, \prime}$ satisfying four conditions similar to
	% those in (A3).
	
	% \gap
	
	\textbf{(B1)} There exist % random variable
	finite scalars $f_{\min}$ and $f_{\max}$ satisfying $f_{\min} < f_{\max}$,
	% bivariate function $f(x,\xi)$ satisfies:
	% has a uniform finite support on $X$, i.e.,
	% for some finite $f_{\min}$ and $f_{\max}$ satisfying $f_{\min} < f_{\max}$,
	\[
	f_{\min} \, \leq \, f(x,\xi) - \mathbb{E}\left[ \, f(x,\tilde{\xi} ) \, \right]
	\, \leq \, f_{\max}
	\epc \mbox{for all } x \, \in \, X \mbox{ and almost all $\xi \in \Xi$}.
	\]
	% \[ \displaystyle{
	% \bigcup_{x\in X}
	% } \, \sup(f(x,\tilde{\xi} )) \, = \, \left[ \, f_{\min}, \, f_{\max} \, \right]
	% \mbox{ where $-\infty < f_{\min} < f_{\max} < + \infty$}.
	% \]
	
	\textbf{(B2)} For any $x \in X$,  the argmax of the OCE defining the objective $\phi(x)$
	in (\ref{eq:oce_deviation}) is nonempty; i.e.,
	\[
	\left\{ \, \eta^* \, \in \, \mathbb{R} \, \mid \, S_u\left( \, f(x,\tilde{\xi} ) -
	\mathbb{E}[ \, f(x,\tilde{\xi} ) ] \, \right)
	\, = \, \eta^* + \mathbb{E}\left[ \, u(f(x,\tilde{\xi} ) -\mathbb{E}[f(x,\tilde{\xi} )]
	-\eta^*) \, \right] \, \right\} \, \neq \, \emptyset.
	\]
With (B1) and (B2) in place and according to Lemma \ref{le:oce_solution}, the OCE-of-deviation optimization \eqref{eq:oce_deviation}
can be formulated as the nonconvex compound stochastic program:
\[
\displaystyle{	
\operatornamewithlimits{\mbox{minimize}}_{x\in X, \, \eta\in \left[ \, f_{\min}, \, f_{\max} \, \right]}
%	\phi(x) \, = \, \underset{{\eta \in [ f_{\min}, f_{\max} ]}}{\mbox{min}}
} \, \left\{ \, -\eta - \mathbb{E}\left[ \, u\left( \, f(x,\tilde{\xi} ) - \mathbb{E}\left[ \,f(x,\tilde{\xi} ) + \eta \, \right] \, \right)
\, \right] \, \right\},
\]
which is in the form of (\ref{eq:compound_composite_SP}) with the identifications of the component functions
given by (\ref{eq:OCEFandG}).
% Hence, the OCE-of-deviation optimization \eqref{eq:oce_deviation} can be formulated
%	as the nonconvex compound stochastic program:
%	\begin{equation} \label{eq:oce_deviation_compound}
%	\displaystyle{	
%		\operatornamewithlimits{\mbox{minimize}}_{x\in X, \, \eta\in \left[ \, f_{\min}, \, f_{\max} \, \right]}
%	} \;\;\; \mathbb{E}\left[ \,
%	\varphi( \, G(x,\eta,\tilde{\xi} ), \mathbb{E} \, [  F(x,\eta,\tilde{\xi} ) ] \, ) \, \right] ,
%	\end{equation}
%	where $F(x,\eta, \xi) \, \triangleq \, f(x, \xi) + \eta$,
%	$G(x,\eta, \xi) \, \triangleq \, \left( \begin{array}{c}
%	-f(x, \xi) \\
%	-\eta
%	\end{array} \right)
%	$,
%	and $\varphi(y_1,y_2, y_3) \, \triangleq \, - u(-y_1 - y_3) + y_2$ is
%	a convex and isotone function.
	With this formulation, we only need to take the surrogation of $\pm f(x,\xi)$
	with respect to $x$ and leave the variable $\eta$ alone.   Depending on the
	structure of $f$, we can construct
	the surrogation accordingly.  For instance, if $f(\bullet,\xi)$
	is a dc function (thus so are $F(\bullet,\eta,\xi)$ and $G(\bullet,\eta,\xi)$),
	we can refer to Subsection~\ref{subsec:assumptions} for the construction of surrogate
	functions of the above objective function.  %  in (\ref{eq:oce_deviation_compound}).
	Details are omitted.

	\subsubsection{bPOE-of-deviation optimization}
	
	Buffered probability of exceedance can be used as a risk measure to quantify
	the deviation as well.  In this and the next subsection, we discuss bPOE-based
	optimization problems.  This subsection pertains to a bPOE of deviation problem
	(\ref{eq:bpoe_deviation_original}) subject to a convex feasibility constraint, similar
	to the OCE-of-deviation problem discussed in the last subsection.  The next subsection
	considers a distributionally robust version with the bPOE optimization problem.
	
	\gap
	
	The $\tau$-bPOE of deviation optimization problem as follows:
	\begin{equation}
	\label{eq:bpoe_deviation_original}
	\underset{x\in X}{\mbox{minimize}} \quad
	\phi(x) \, \triangleq \,
	\mbox{bPOE}\left( \, f(x,\tilde{\xi} ) -\mathbb{E}\left[ \, f(x,\tilde{\xi} ) \, \right];
	\tau \, \right).
	\end{equation}
	% \quad \Leftrightarrow \,  \underset{x\in X, a \geq 0}{\min} \mathbb{E} \left[ \sum_{j=1}^J [a_j(f(x, \tilde{\xi} ) -\mathbb{E} f(x,\tilde{\xi} ) -\tau_j)+1]_+ \right]
	Without loss of generality, we may assume that
	% \textbf{(B6)}
	$\tau \leq \displaystyle{
		\min_{x\in X}
	} \,  \sup_{{\xi} \in \Xi}\, (f(x,  {\xi} ) - \mathbb{E}[f(x, \tilde{\xi} )])$, because
	if $\tau > \displaystyle{
		\min_{x\in X}
	} \,  \sup_{{\xi} \in \Xi} \, (f(x, {\xi} ) - \mathbb{E}[f(x, \tilde{\xi} )])$, then any $\bar{x}\in \displaystyle{
		\operatornamewithlimits{\mbox{argmin}}_{x\in X}
	} \, \sup_{{\xi} \in \Xi}\,(f(x,  {\xi} ) - \mathbb{E}[f(x, \tilde{\xi} )])$
	is an optimal solution  of \eqref{eq:bpoe_deviation_original} because
	$\phi(\bar{x}) = 0$.  Furthermore, to avoid technical details, we
	restrict the level $\tau$ so that Lemma \ref{le:solution_interval} is applicable to
	the random variables for all
	$x \in X$; specifically, we assume
	\gap
	
	\textbf{(C1)} the constant $\tau$ satisfies that $\tau < \displaystyle{
		\min_{x\in X}
	} \,  \sup_{{\xi} \in \Xi}\,\, (f(x,  {\xi} ) - \mathbb{E}[f(x, \tilde{\xi} )]  )$.
	\gap
	
	With this restriction, from Lemma \ref{le:solution_interval}, we deduce the existence of a
	positive constant $\bar{A}$ independent of $x$ such that
	\[
	\phi(x) \, = \, \underset{0 \, \leq a \, \leq \, \bar{A}}{\mbox{min}} \,
	\mathbb{E}\left[ \, a \, ( \, f(x, \tilde{\xi} ) -\mathbb{E}[f(x,\tilde{\xi} )] -\tau \, ) + 1 \,
	\right]_+ \, .
	\]
	By lifting the optimization problem in \eqref{eq:bpoe_deviation_original} to the
	minimization over $x$ and $a$ jointly, we can formulate the following compound SP with a   convex and compact constraint:
	\[
	\displaystyle{
		\operatornamewithlimits{\mbox{minimize}}_{x\in X ,\; 0\leq a \leq \bar{A} }
	}  \;\;\;\mathbb{E}\left[ \,
	\varphi(G(x,a,\tilde{\xi} ), \mathbb{E}[F(x,a,\tilde{\xi} )]) \, \right],
	%\mbox{subject to} & ( \, x,a \, ) \, \in \, \underbrace{X \, \times \,
	%\left[ \, 0, \, \bar{A} \, \right]}_{\mbox{compact convex}},
	\]
	where $\varphi(b_1,b_2) = [ \, b_1 + b_2 \,]_+$,
	$F(x,a,\xi) \, \triangleq \, -a f(x,\xi)$ and
	$G(x,a,\xi) \, \triangleq \, a (f(x,\xi) - \tau) + 1
	$.
	For a host of loss functions $f(x,\xi)$, such as smooth functions and dc functions,
	we can again derive surrogates of the above objective function  based on the general discussion in previous
	sections.

	\subsection{Distributionally robust mixed bPOE optimization}
	
	With the single bPOE measure and  a convex loss function $f(\bullet,\xi)$,
	the bPOE optimization problem can be reformulated as a convex stochastic program
	(see Proposition~4.9 in \cite{mafusalov2018buffered}) by using the operation of
	right scalar multiplication.  Several special cases, such as convex piecewise linear
	cost functions and homogenous cost functions with no constraints are considered in
	\cite{mafusalov2018buffered} and \cite{mafusalov2018estimation}.  In this subsection, we
	consider the bPOE optimization with the objective being a mixture of bPOEs at multiple
	thresholds. In this way, unlike the mixed CVaR which is still a convex function, bPOE
	optimization becomes nonconvex and nonsmooth.  Though the idea of mixed risk measures
	is not new, to the best of our knowledge, the present paper is the first to study mixed
	bPOE optimization. There are many potential applications of mixed bPOE.  For instance,
	a decision maker may expect a solution under which not only the probability of the loss
	exceeding a positive level $\tau_1$, i.e., $\mathrm{POE}( f(x,\tilde{\xi} ); \tau_1)$
	is small but
	also the probability of the return exceeding a positive level $\tau_2$, i.e.,
	$\mathrm{POE}( -f(x,\tilde{\xi} ); \tau_2) = 1- \mathrm{POE}(f(x, \tilde{\xi} );- \tau_2)$
	is large.
	Using the bPOE in the place of POE and a mixture of two objectives, the goal is thus
	to minimize the combined objective $ \beta \, \mathrm{bPOE}( f(x, \tilde{\xi} ); \tau_1) +
	(1-\beta) \, \mathrm{bPOE}  ( f(x, \tilde{\xi} ); -\tau_2) $ for different values of
	the scalar $\beta \in (0,1)$.
	
	\gap
	
	In practice,  the exact probability distribution of the uncertainty is usually unknown.
	If one has the partial knowledge of  the true probability distribution,  one could
	construct an ambiguity set of probability distributions and formulate the
	distributionally robust mixed bPOE optimization.  There are numerous choices of
	ambiguity sets; we refer to \cite{rahimian2019distributionally} for a thorough
	review of the state-of-art of  distributionally robust optimization.  In what follows,
	we use a set of mixture distributions as the ambiguity set and we formulate the
	resulting distributionally robust mixed bPOE optimization as a compound SP program.
	If we use the ambiguity set based on the Wasserstein distance, with the dual result in
	\cite{gao2016distributionally}, a standard SP may be formulated.  Since this falls
	outside the scope of compound SPs, we do not consider this version of the robust bPOE
	problem in this paper.  % save it for the interest of the future research.
	% Instead, in the following discussion of the distributionally robust mixed bPOE
	% optimization, we include the bPOE both in the objective and the constraints and
	% consider multiple threshold levels in each mixture.
	
	\gap
	
	% In what follows, we consider the optimization of a multi-level mixture of bPOEs.
	% and mixtures of bPOEs of deviations in a unified way using given scalars
	% $\gamma_j \in [ 0,1 ]$ where $j = 1, \cdots, J$ indexes the levels.  Specifically,
	In what follows, we consider the optimization of a multi-level mixture of bPOEs.
	% and mixtures of bPOEs of deviations in a unified way using given scalars
	% $\gamma_j \in [ 0,1 ]$ where $j = 1, \cdots, J$ indexes the levels.  Specifically,
	Let $\mathcal{P}$ denote the family of ambiguity sets of cumulative  distribution functions.
	% with the family $\mathcal{P}_0$ being employed
	% in the objective and the others in the $I$ constraints.
	The objective function is the maximum over distributions in $\mathcal{P}$ of
	a weighted summation of the
	bPOE$( \, f(x,\tilde{\xi}_p);\tau_j \,  )$
	% - \gamma_j \, \mathbb{E}\left[ \, f(x,Z_p) \, \right];\tau_j \right)$
	at the level
	$\tau_j$ weighted by the factor $\beta_j \geq 0$ for $j = 1, \cdots, J$,
	where
	$\tilde{\xi}_p : \Omega \to \Xi \subseteq \mathbb{R}^m$
	represents the random variable with cumulative distribution function $p \in \mathcal{P}$.
	With this setup, we consider
	the following distributionally robust mixed bPOE optimization problem:
	\begin{equation} \label{eq:dro_mixed_bpoe_opt}
	\displaystyle{
		\operatornamewithlimits{\mbox{minimize}}_{x \in X}
	} \;\; \displaystyle{
		\operatornamewithlimits{\mbox{max}}_{p \in \mathcal{P}}
	} \,
	% \displaystyle{
	% \operatornamewithlimits{\mbox{minimum}}_{a \geq 0}
	% } \, \displaystyle{
	% \sum_{j=1}^{J}
	% } \, \beta_j \, \mathbb{E}\left[ \, a_j \left( \,
	% f(x, Z_p)
	% - \gamma_j \,
	% \mathbb{E}\left[ \, f(x,Z_p) \, \right]
	% - \tau_j \, \right) + 1 \, \right]_+ \\ [0.15in]
	\displaystyle{
		\sum_{j=1}^{J}
	} \, \beta_j \, \displaystyle{
		\operatornamewithlimits{\mbox{min}}_{a_j \geq 0}
	} \ \mathbb{E}\left[ \, a_j \left( \,
	f(x,\tilde{\xi}_p)
	% - \gamma_j \,
	% \mathbb{E}\left[ \, f(x,Z_p) \, \right]
	- \tau_j \, \right) + 1 \, \right]_+  ,
	\end{equation}
	% \mbox{subject to} \ \mbox{for all $i \, = \, 1, \cdots, I$}:  \\ [0.1in]
	% \underset{p_i \in \mathcal{P}_i}{\max} \,
	% \underset{a^i \geq 0}{\min} \ \displaystyle{
	% \sum_{j=1}^{J_i}
	% } \, \beta_{i,j}  \, \mathbb{E}\left[ \, a_{i,j} \left( \,
	% f_i(x, Z_{i,p_i}) - \gamma_{i,j} \,
	% \mathbb{E}\left[ \, f_i(x,Z_{i,p_i}) \, \right] - \tau_{i,j} \, \right) + 1
	% \, \right]_+ \, \leq \, c_i,
	% \end{array}
	% \end{equation}
	% where % each $c_i$ is a positive constant and
	where ${\bf a} \triangleq \left( a_j \right)_{j=1}^J$.
	% For each $i=0, 1, \cdots, I$,
	In what follows, we take the family $\mathcal{P}$ to consist
	of mixtures (i.e., convex combinations) of given cumulative distribution functions
	$\{ \pi_k \}_{k=1}^K$, i.e.,
	$
	\mathcal{P} \, \triangleq \, \left\{ \,  \displaystyle{
		\sum_{k=1}^K
	} \, \lambda_k \, \pi_k \, : \, \lambda \, \in \, \Lambda \, \right\},
	% \quad \forall \, i = 0, \cdots, I,
	$
	where $\Lambda \triangleq \left\{ \, \lambda \, = \, ( \lambda_k )_{k=1}^K
	% \cdots, \lambda_{i,K_i})
	\, \mid \, \displaystyle{
		\sum_{k=1}^K
	} \, \lambda_k = 1, \mbox{ and } 0 \, \leq \, \lambda_k \, \leq 1 \
	\forall \, k = 1, \cdots, K \, \right\}$ is the unit simplex in $\mathbb{R}^K$.
	Although being convex constrained, problem (\ref{eq:dro_mixed_bpoe_opt}) is a
	highly nonconvex nonsmooth problem with the nonconvexity and nonsmoothness coupled
	in a nontrivial way; both challenging features are in addition to the stochastic
	element of the expectation.  Besides being
	computationally intractable as far as a global minimizer is concerned, it is not even
	clear what kind of stationary solutions one can hope for any algorithm to be able to
	compute.  Our approach to treat this problem is as follows.  First, we show that each
	minimization over $a_j$ can be restricted to a compact interval; this then allows us
	to exchange the maximization over $p \in {\cal P}$ with the minimization over
	${\bf a}  \geq 0$.  The second step is to combine the latter minimization in ${\bf a}$ with
	that in $x$ and to realize that the maximization in $p$ over the unit simplex $\Lambda$
	is equivalent is a discrete pointwise maximum, by the fact that the cumulative distribution function of the random variable $\tilde{\xi}_p$
	is a mixture of $\{ \pi_k \}_{\nu=1}^K$.  The end result
	is that problem (\ref{eq:dro_mixed_bpoe_opt}) becomes a compound SP of the type
	(\ref{eq:compound_composite_SP}).  To accomplish these steps, we assume,
	in addition to the standing assumption of compactness and convexity of the set $X$
	and the Carath\'eodory property of $f$, the following condition similar to the condition (C1):
	
	\gap
	
	% \textbf{(C1)} $f(x,\xi)$ is a Carath\'eodory function satisfying assumption (B2).
	
	% \gap
	
	% \textbf{(C2)} The feasible solution set $X$ is a convex and compact set.
	
	% \textbf{(C2)} The problem (\ref{eq:dro_mixed_bpoe_opt}) is feasible.
	% There exists $x\in X$ such that constraints in \eqref{eq:bpoe_deviation_original} hold.
	
	% \gap
	
	% \textbf{(C3)} $\displaystyle{
	% \max_{x\in X}
	% } \, \sup(f_i(x,T_{i,k})) < +\infty$ for all $i \in [I]$ and $k \in [K_i]$.
	%
	% \gap
	
	\textbf{(C2)} for each $j = 1, \cdots, J$, the constant $\tau_j$ satisfies that $\tau_j < \displaystyle{
		\min_{x\in X}
	} \, \displaystyle{
		\max_{1 \leq k \leq K}
	} \, \sup_{{\xi}_{\pi_k} \in \Xi}\, f(x,{\xi}_{\pi_k})$.
	% where
	% $\wh{T}_{j,k}(x) \triangleq f(x,T_k) - \gamma_j \,
	% \mathbb{E}\left[ \, f(x,T_k) \, \right]$

	\gap
	
	With $p \in {\cal P}$ given by $p = \displaystyle{
		\sum_{k=1}^K
	} \, \lambda_k \, \pi_k$, we write
	\[ % \begin{array}{lll}
	g_j(x,a_j,p) \, \triangleq \,
	\mathbb{E}\left[ \, a_j \left( \, f(x,\tilde{\xi}_p)
	% - \gamma_j \,
	% \mathbb{E}\left[ \, f(x,Z_p) \, \right]
	- \tau_j \, \right) + 1
	\, \right]_+ % \ [0.1in]
	\, = \, \displaystyle{
		\sum_{k=1}^K
	} \, \lambda_k \, \underbrace{\mathbb{E}\left[ \, a_j \left( \,
		f(x, \tilde{\xi}_{\pi_k})
		% - \gamma_j \, \mathbb{E}\left[ \, f(x,T_k) \, \right]}_{\mbox{$= \ \wh{T}_{j,k}(x)$}}
		- \tau_j \, \right) + 1 \, \right]_+}_{\mbox{denoted $g_j(x,a_j,\pi_k)$}}.
	% \end{array}
	\]
	Thus,
	\begin{equation} \label{eq:minsum}
	\underset{a_j \geq 0}{\mbox{minimum}} \ g_j(x,a_j,p)
	\, = \, \underset{a_j \geq 0}{\mbox{minimum}} \
	\displaystyle{
		\sum_{k=1}^K
	} \, \lambda_k \, g_j(x,a_j,\pi_k).
	\end{equation}
	Under (C2), we may deduce from Lemma \ref{le:solution_interval} the existence of a
	positive constant $\bar{A}_j$ independent of $k$ such that
	\begin{equation} \label{eq:separate min}
	\underset{a_j \geq 0}{\mbox{minimum}} \ g_j(x,a_j,\pi_k)
	\, = \, \underset{0 \, \leq \, a_j \, \leq \, \bar{A}_j}{\mbox{minimum}} \
	g_j(x,a_j,\pi_k).
	\end{equation}
	The constants $\bar{A}_j$ cannot be directly applied to (\ref{eq:minsum}) because
	there is the summation over $k$.  We need to invoke the following lemma which shows
	that if each univariate convex
	minimization problem in a finite family has a bounded optimal solution, then so
	does any nonnegative combination of this family of problems.  This allows us to
	bound the minimizer in (\ref{eq:minsum}) and prepare the
	exchange of the maximization and inner
	minimization operators in the objective function of
	\eqref{eq:dro_mixed_bpoe_opt}; see the subsequent
	Proposition \ref{prop:refor_bPOE_mixture_ambiguity_set}.
	
	\begin{lemma} \label{le:solution_interval_sum} \rm
		Let $\{ h_i(t) \}_{i=1}^I$ be a finite family of convex functions all defined
		on the same closed interval ${\cal T}$ of $\mathbb{R}$.  For each $i \in [I]$, let
		$t_i^* \in {\cal T}$ be a global minimizer of
		$h_i$ on ${\cal T}$.  Then for any family of
		nonnegative scalars $\{\beta_i\}_{i=1}^I$, the sum function
		$\wh{h}(t) \triangleq \displaystyle{
			\sum_{i=1}^I
		} \, \beta_i h_i(t)$ has a minimizer on ${\cal T}$ that belongs to the compact
		sub-interval ${\cal T}_* \triangleq \left[ \, \displaystyle{
			\min_{i \in [ I ]}
		} \, t_i^*, \, \displaystyle{
			\max_{i \in [ I ]}
		} \, t_i^* \, \right]$.
	\end{lemma}
	
	\begin{proof} Being continuous, $\wh{h}$ has a minimizer in the
		interval ${\cal T}_*$, which we denote $t_*$.  We claim that $t_*$ is a minimizer
		of $\wh{h}$ on ${\cal T}$.  Assume the contrary.  Then there exists
		$u \in {\cal T} \setminus {\cal T}_*$ such that $\wh{h}(u) < \wh{h}(t_*)$.
		Without loss of generality, we may assume that $t_1^* = \displaystyle{
			\min_{i \in [I]}
		} \, t_i^*$ and $t_I^* = \displaystyle{
			\max_{i \in [I]}
		} \, t_i^*$.  Suppose $u < t_1^*$.  By convexity, we have $h_i(u) \geq h_i(t_1^*)$
		for all $i \in [I]$.  Hence, it follows that
		\[
		\displaystyle{
			\sum_{i \in [ I ]}
		} \, \beta_i \, h_i(t_1^*) \, \leq \, \wh{h}(u) \, < \, \wh{h}(t_*)
		\, \leq \, \displaystyle{
			\sum_{i \in [ I ]}
		} \, \beta_i \, h_i(t_1^*),
		\]
		which is a contradiction.  If $u > t_I^*$, then $h_i(u) \geq h_i(t_I^*)$
		for all $i \in [I]$.  Hence, it follows that
		\[
		\displaystyle{
			\sum_{i \in [ I ]}
		} \, \beta_i \, h_i(t_I^*) \, \leq \, \wh{h}(u) \, < \, \wh{h}(t_*) \, \leq \,
		\wh{h}(t_I^*) \, = \, \displaystyle{
			\sum_{i \in [ I ]}
		} \, \beta_i \, h_i(t_I^*),
		\]
		which is also a contradiction.
 
	\end{proof}
	\gap
	
	Since for each $j$, $\displaystyle{
		\sum_{k=1}^K
	} \, \lambda_k \, g_j(x,a_j,\pi_k)$ is nonnegative combination of the family
	$\left\{ \, g_j(x,a_j,\pi_k) \, \right\}_{k=1}^K$, and each
	$g_j(x,\bullet,\pi_k)$ attains its minimum on the nonnegative axis in
	the interval $\left[ 0, \bar{A}_j \, \right]$, by %(\ref{eq:minsum}), 
	(\ref{eq:separate min}) and the above lemma,
	we deduce the existence of positive constants $\wh{A} \triangleq (\wh{A}_1 , \cdots, \wh{A}_J)$ which are independent of $p \in {\cal P}$ such that
	\[ \begin{array}{ll}
	\displaystyle{
		\operatornamewithlimits{\mbox{min}}_{a \geq 0}
	} \ \displaystyle{
		\sum_{j=1}^J
	} \, \beta_j \, g_j(x,a_j,p) \, & = \,
	\displaystyle{
		\sum_{j=1}^J
	} \, \beta_j \, \displaystyle{
		\operatornamewithlimits{\mbox{min}}_{a_j \geq 0}
	} \, g_j(x,a_j,p)
	= \, \displaystyle{
		\sum_{j=1}^J
	} \, \beta_j \, \displaystyle{
		\operatornamewithlimits{\mbox{min}}_{a_j \geq 0}
	} \, \displaystyle{
		\sum_{k=1}^K
	} \, \lambda_k \, g_{j}(x,a_j,\pi_k) \\ [0.2in]
	& = \, \displaystyle{
		\sum_{j=1}^J
	} \ \beta_j \, \displaystyle{
		\operatornamewithlimits{\mbox{min}}_{0 \, \leq \, a_j \, \leq \, \wh{A}_j}
	} \, \displaystyle{
		\sum_{k=1}^K
	} \, \lambda_k \, g_{j}(x,a_j,\pi_k)  = \, \displaystyle{
		\operatornamewithlimits{\mbox{min}}_{0 \, \leq \,{\bf a}  \leq \, \wh{A}}
	} \ \displaystyle{
		\sum_{j=1}^J
	} \, \beta_j \, g_j(x,a_j,p).
	\end{array}
	\]
	% where $\wh{A} \triangleq \left( \wh{A}_j \right)_{j=1}^J$.
	Since the simplex $\Lambda$ is a compact convex set, it follows from the
	well-known min-max theorem that
	\[ \begin{array}{ll}
	& \displaystyle{
		\operatornamewithlimits{\mbox{max}}_{p \in {\cal P}}
	} \, \, \displaystyle{
		\operatornamewithlimits{\mbox{min}}_{a \geq 0}
	}  \displaystyle{
		\sum_{j=1}^J}
	\, \beta_j \, g_j(x, a_j,p) \,  = \, \displaystyle{
		\operatornamewithlimits{\mbox{min}}_{0 \, \leq \, a \, \leq \, \wh{A}}
	} \, \displaystyle{
		\operatornamewithlimits{\mbox{max}}_{p \in {\cal P}}
	} \, \displaystyle{
		\sum_{j=1}^J}
	\, \beta_j \, g_j(x, a_j,p) \\ [0.2in]
	= \, & \displaystyle{
		\operatornamewithlimits{\mbox{min}}_{0 \, \leq \, a \, \leq \, \wh{A}}
	} \, \displaystyle{
		\operatornamewithlimits{\mbox{max}}_{\lambda \in \Lambda}
	} \, \displaystyle{
		\sum_{k=1}^K
	} \, \lambda_k \, \displaystyle{
		\sum_{j=1}^J
	} \, \beta_j \, g_j(x,a_j,\pi_k)
	= \, \displaystyle{
		\operatornamewithlimits{\mbox{min}}_{0 \, \leq \, a \, \leq \, \wh{A}}
	} \, \displaystyle{
		\operatornamewithlimits{\mbox{max}}_{1 \leq k \leq K}
	} \, \displaystyle{
		\sum_{j=1}^J
	} \, \beta_j \, g_j(x,a_j,\pi_k).
	\end{array}
	\]
	Summarizing the above derivations, we obtain the next result that yields the promised
	equivalent formulation of (\ref{eq:dro_mixed_bpoe_opt}) as one
	of minimizing a pointwise maximum function.  No more proof is needed.
	
	\begin{proposition} \label{prop:refor_bPOE_mixture_ambiguity_set} \rm
		Under condition (C3), the distributionally robust mixed bPOE optimization problem
		\eqref{eq:dro_mixed_bpoe_opt} is equivalent to the following program:
		\begin{equation}
		\label{eq:dro_mixed_bpoe_mixture}
		\displaystyle{
			\operatornamewithlimits{\mbox{minimize}}_{x\in X, \, 0 \leq  a \leq \wh{A}}
		} \quad \displaystyle{
			\operatornamewithlimits{\mbox{max}}_{1 \leq k \leq K}
		} \, \displaystyle{
			\sum_{j=1}^J
		} \, \beta_j \, g_j(x,a_j,\pi_k) =  \displaystyle{
			\operatornamewithlimits{\mbox{max}}_{1 \leq k \leq K}
		} \, \displaystyle{
			\sum_{j=1}^J
		} \, \beta_j \, \mathbb{E}\left[ \, a_j \left( \,
		f(x, \tilde{\xi}_{\pi_k})
		% - \gamma_j \,
		% \mathbb{E}\left[ \, f(x,T_k) \, \right]
		- \tau_j \, \right) + 1 \, \right]_+ .
		\end{equation}
		The equivalence means that if $x$ is an optimal solution of
		(\ref{eq:dro_mixed_bpoe_opt}), then $a$ % \triangleq \{ a^i \}_{i=0}^I$
		exists such that $(x,a)$ is an optimal solution of
		(\ref{eq:dro_mixed_bpoe_mixture}); conversely, if $(x,a)$ is an optimal
		solution of (\ref{eq:dro_mixed_bpoe_mixture}), then
		$x$ is an optimal solution of (\ref{eq:dro_mixed_bpoe_opt}).  \hfill $\Box$
	\end{proposition}
	
	\gap
	To see how (\ref{eq:dro_mixed_bpoe_mixture}) fits the compound SP framework
	(\ref{eq:compound_composite_SP}), we define $\psi : \mathbb{R}^{JK} \to \mathbb{R}$
	by
	$
	\psi(y^1, \cdots, y^J) = \displaystyle{
		\max_{1 \leq k \leq K}
	} \, \displaystyle{
		\sum_{j=1}^J
	} \, \beta_j \, y^j_k$ where each
	$y^j \, = \, \left( y^j_k \right)_{k=1}^K
	$,
	%and
	%the other functions $\varphi$ and $G$ are similar to
	%those in the bPOE of deviation optimization problem except that they are
	%multi-dimensional due to the $J$ thresholds $\{ \tau_j \}_{j=1}^J$ and
	%$K$ basic random variables $\{ \tilde{\xi}_{\pi_k} \}_{\nu=1}^K$ being mixed to yield
	%the probability distributions in the family ${\cal P}$. Specifically,
	% $\varphi = \left( \, \left( \varphi_{j,k} \right)_{\nu=1}^K \, \right)_{j=1}^J$,
	% where each
	% $\varphi_{j,k} : \mathbb{R}^{JK + JK} \to \mathbb{R}_+$ is given by
	$
	G(x,a,\xi) =
	\left( \, \left( \, a_j \, ( \, f(x, {\xi}_{\pi_k}) - \tau_j \, ) + 1 \, \right)_{k=1}^K \,
	\right)_{j=1}^J
	$, $\varphi(b) =  b_+$ for $b \in \mathbb{R}^{JK}$, and the function $F$ to be identically zero.
	One can now readily obtain convex
	majorants of the objective function in (\ref{eq:dro_mixed_bpoe_mixture}).
	Details are omitted. Notice that in this problem, the outer function $\psi$ is a
	non-trivial pointwise maximum of finitely many nondecreasing linear functions.

	\subsection{Cost-sensitive multiclass classification with buffered probability }
	
	In binary classifications, type I error (the conditional probability that the predicted label is 1 given that the true label is 0) and type II error (the conditional probability that    the predicted label is 0 given that the true label is 1)
	may result in significantly different costs. So cost-sensitive learning approach assigns two different costs as weights of the type I and type II errors in order to learn a classifier with
	the asymmetric control of the binary classification errors. Generalizing to the multiclass classification, there could be multiple (more than two) types of errors which need the asymmetric control.  For instance, a medical diagnosis may categorize the serious condition of a
	disease into 5 levels, i.e., $\{1,2,3,4,5\}$, in which the higher number represents the worse
	condition.  We might categorize the errors into 4 groups by the gap between the true level and the categorized level: (i) greater than 2; (ii) less than 2 and greater than 0;  (iii)  less than 0 and greater than -2; (iv) less than -2. Errors in each group can produce similar costs, but the costs of errors among different groups could be significantly different.
	Hence, with this attitude of the asymmetric classification errors,  we need to assign different weights as the costs in learning the classifier.
	
	\gap
	
	In general, for a multiclass classification problem, let $ (X, Y )$ be jointly distributed random variables where
	$Y \in [ \, M \, ] \triangleq \{1, \ldots, M\}$
	denotes the label from $M$ classes and $X^m \triangleq \{ X\mid Y=m\}$ denotes the random vector of attributes in
	class $m$. The classifier makes
	its decision based on the scoring functions
	$\{ h(x,  {\mu}^m) \}_{m=1}^M$ where $ {\mu}^m\in \mathbb{R}^n$ are parameters.
	Given these parameters, we classify a vector $X$ into the  class $j$ if
	$
	j \, \in \,
	\operatornamewithlimits{\mbox{argmax}} \{
	\ h(X, {\mu}^m): {m \, \in \, [ \, M \, ]}\} .
	$
	Since there are $M$ classes, there are $M(M-1)$ types of misclassification errors.
	Suppose that these errors are separated into multiple groups such that each
	group of errors could result in the distinguished costs.
	Mathematically, let
	$
	T \, \triangleq \, \left\{ \, (i,j) \, \in \, [ \, M \, ] \times [ \, M \, ] \, \mid \,
	i \, \neq \, j \, \right\}
	$
	be all pairs of misclassified labels, i.e., a true label $i$ is misclassified as
	$j \neq i$. % represents the true label and $j$ represents the classified label.
	Let $T = \displaystyle{
		\bigcup_{s=1}^S
	} \ T_s$ be a partition of $T$ into $S$ groups each of which is associated with a weight
	in the learning the classification.  The conditional probability that the classified label is $j$ given that the true label is $i$ with $j \neq i$ is given by
	$
	\mathbb{P} \left(h(X^i,  {\mu}^j) \, > \, \displaystyle\max_{m \in [M]} \,  h(X^i, {\mu}^m)
	\right).
	$
	% Suppose that the classes are balanced within the groups, which means that
	% $\mathbb{P}(Y=i)$ are almost the same for $(i,j) \in T_s$.
	To produce classifiers with acceptable margins, we are interested in the following  probability of  having the error $(i,j)\in T$   with the margin $\tau_{i,j} \geq 0$:
	\[
	\mathbb{P} \left(h(X^i,  {\mu}^j) \geq  \max_{m \in [M]} h(X^i,  {\mu}^m ) -
	\tau_{i,j} \right) \,= \, \mbox{POE}\left(h(X^i,  {\mu}^j) - \displaystyle{\max_{m \in [M]}}  \, h(X^i,  {\mu}^m ); \, -
	\tau_{i,j} \right).
	\]
	The above probability is  a discontinuous function of $\mu$, thus difficult to be minimized directly. Nevertheless, when we replace the POE with its buffered version,  then the cost-sensitive multiclass classification model yields a compound SP \eqref{eq:compound_composite_SP}. Specifically, with the weights $\{\alpha_s\}_{s \in [S]}$ for $S$ groups of errors and tolerance $\{\tau_{i,j}\}$, we consider the following buffered cost-sensitive multiclass classification model
	\[
	\displaystyle \operatornamewithlimits{minimize}_{ \{{\mu}^m  \, \in \,  \mathbb{R}^n\}_{m=1}^M} \;\; \displaystyle{
		\sum_{s=1}^{S}
	} \, \alpha_s \, \left\{ \, \max_{(i,j)\in T_s}  \mathrm{bPOE}\left(h(X^i, {\mu}^j)
	- \max_{m \in [M]}  h(X^i,  {\mu}^m ); -\tau_{i,j}\right) \right\}.
	\]
	
	Using the minimization formula of the bPOE that involves lifting such a formula by
	an additional scalar variable, interchanging the resulting minimization in the latter
	variable by the max over the pairs $(i,j)$, we can obtain a compound stochastic
	programming problem \eqref{eq:compound_composite_SP} of the above model
	wherein
	the function $F$ is identically equal to zero.  We omit the details; cf.\ the derivation
	at the end of the last subsection.

	\section{Numerical Experiments} \label{sec:numerics}

In this section, we present some numerical results of the SMM algorithm for the OCE-of-deviation optimization problem \eqref{eq:oce_deviation}
to demonstrate the effectiveness of SMM. We investigate such a compound SP with two particular utility functions \cite{ben2007old}, which are
respectively the exponential utility function
$u_{\rm exp}(t) = 1- e^{- t}$, and the piecewise linear utility function $u_{\rm pl}(t) = \gamma_1  \, \max\{t, 0\} - \gamma_2 \, \max\{-t, 0\}$ with
$ 0 \leq  \gamma_1 < 1 < \gamma_2$. In subsection \ref{subsec:exponential}, we address the   OCE-of-deviation optimization problem with the exponential utility function, which is a compound SP with smooth component functions, and thus enables the application of both the SMM algorithm and the nested stochastic approximation algorithms  \cite{GhadimiRuszczynskiWang20,WangFangLiu17,YangWangFang19}.  We compare the performance
of the SMM algorithm with the Nested Averaged Stochastic Approximation (NASA) algorithm in \cite{GhadimiRuszczynskiWang20}, which has
been shown to have faster convergence rate than other nested stochastic approximation methods \cite{WangFangLiu17,YangWangFang19} for smooth
nonconvex compound SP. In subsection \ref{subsec:piecewise}, we further address the piecewise linear utility-based OCE-of-deviation optimization
problem which requires the solution of a nonsmooth nonconvex compound SP.  We conducted all the computations in Python on
MacBookPro with macOS Catalina system  and 2.7 GHz Intel Core i5 8GB RAM. The subproblems in the SMM algorithm were solved by the {\sc Ipopt} solver
\url{https://coin-or.github.io/Ipopt/}; see also \cite{Wachter02}.
	
\subsection{OCE-of-deviation optimization: exponential utility}
    \label{subsec:exponential}

    From Example 2.1 in \cite{ben2007old}, we know that with the exponential utility function $u_{\rm exp} (t) =  1- e^{-t}$, the  OCE of  a random variable $Z$ can be represented as $S_{u_{\rm exp}}(Z) = - \log \, \mathbb{E} \,[ \, e^{-Z} \,]$. Then with the loss function $f(x, \xi) = (x - \xi)^2$,
	a  normal random variable $\tilde{\xi}$ with mean value $4$ and variance $0.5$, the OCE-of-deviation optimization
	problem restricted to the interval $[0,8]$ is equivalent to:

	\begin{align} \label{eq:oce_deviation_exp}
	\displaystyle{
		\operatornamewithlimits{\mbox{minimize}}_{x\in [0,8]}
	}  \,\,   \mathbb{E} \Big[ \, {\rm exp} \left\{ -(x-\tilde{\xi} \, )^2 + \mathbb{E}[ \,  (x-\tilde{\xi}\, )^2 \, ]\right\}
	\, \Big].
	\end{align}
With the sample average approximation (SAA) of the objective function in Figure~\ref{fig:sol_sample}(a), we can estimate that the global
optimum is around 4 according to the SAA convergence results in \cite{dentcheva2017statistical,ermoliev2013sample,hu2020sample}.  We set the sequence of
incremental sample sizes  $\Delta_{\nu} = \lfloor \nu^{\, \alpha} \rfloor +1 $ for some $\alpha >0$ in order to satisfy the requirement of $\{ N_{\nu} \}$
as stipulated in Lemma \ref{lm:condition on sample sizes}.  With the fixed proximal parameter $\rho = 10$ and random initialization on $[0,8]$, Figure~\ref{fig:sol_sample}(b) plots the results
from the SMM algorithm with 50 replications of 20 iterations each and
% in which the  incremental sample sizes $\Delta_k = \lfloor k^{\, \alpha} \rfloor + 1$ is chosen
with $\alpha = 0.3, \,  0.4,$ and $0.5$ respectively.  We can observe that the sequence of solutions and the associated objective value approximations
from the SMM with the larger incremental sample rate are less volatile than the ones with smaller incremental sample rate.  Furthermore, by fixing
$\alpha =0.4$, we compare the performance of the SMM with multiple proximal parameters $\rho=10, 50, 100$ in Figure \ref{fig:sol_sample}(c).
The sequence with $\rho = 50$ converges faster with smaller volatility than the sequences with the other two parameter values.  Due to the large
variances of the starting points, the initial iterations are excluded in Figure \ref{fig:sol_sample}(b) and Figure \ref{fig:sol_sample}(c); otherwise, the information
of the rest iterations will be shadowed by the large volatility of the first iteration.  We also note that after 20 iterations, the difference
$\| x^{\nu} - x^{\nu+1} \|$ is of the order $10^{-2}$ which can be reduced to $10^{-3}$ after 5 more iterations.
%    \vspace{-0.2in}

    \begin{figure}[!ht]
	\subfigure[SAA of the objective value]{\includegraphics[width=0.33\textwidth]{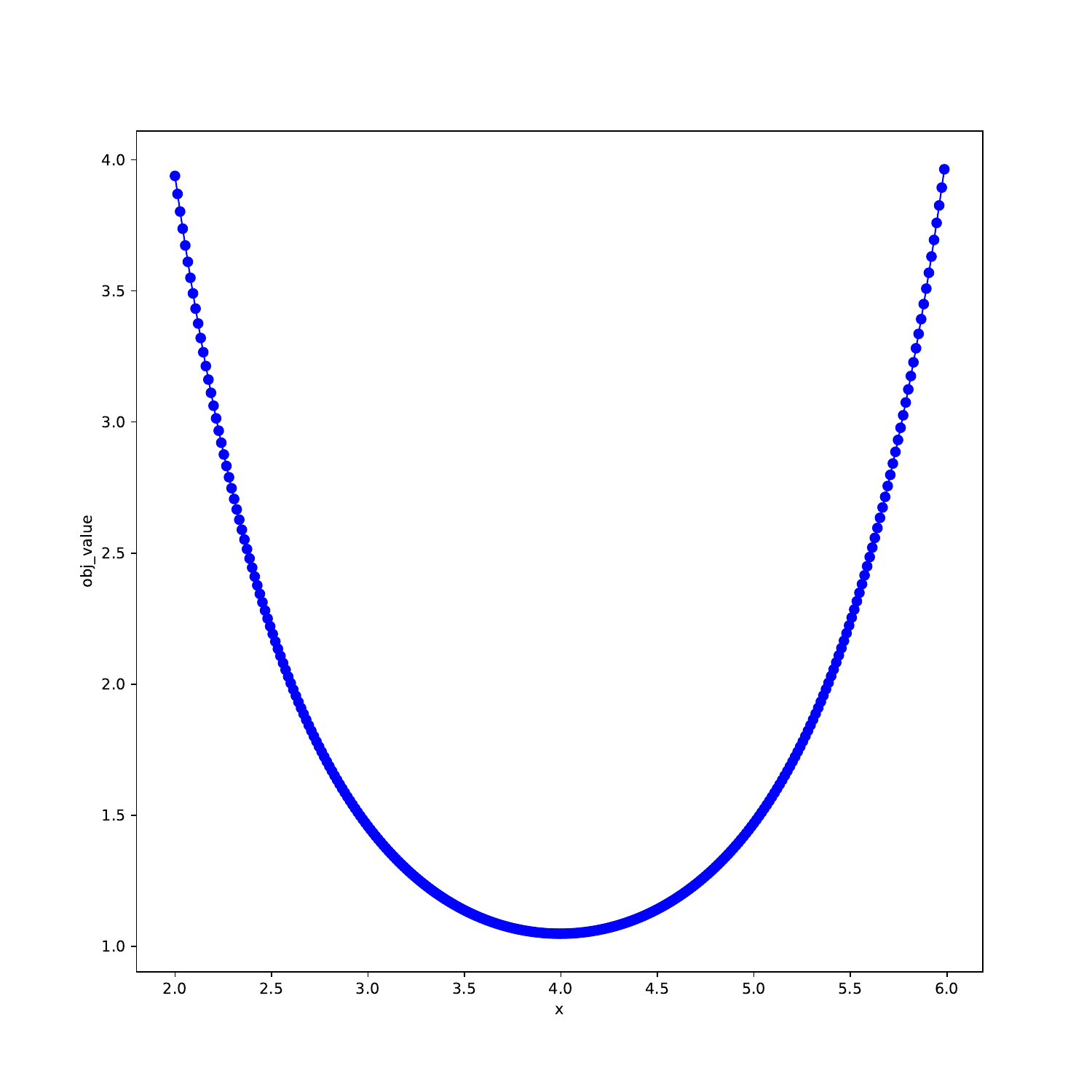}}
	\subfigure[    incremental sample size   $\Delta_{\nu} = \lfloor {\nu}^\alpha \rfloor + 1$ with $\alpha = 0.3, 0.4, 0.5$]{
	\includegraphics[width=0.33\textwidth]{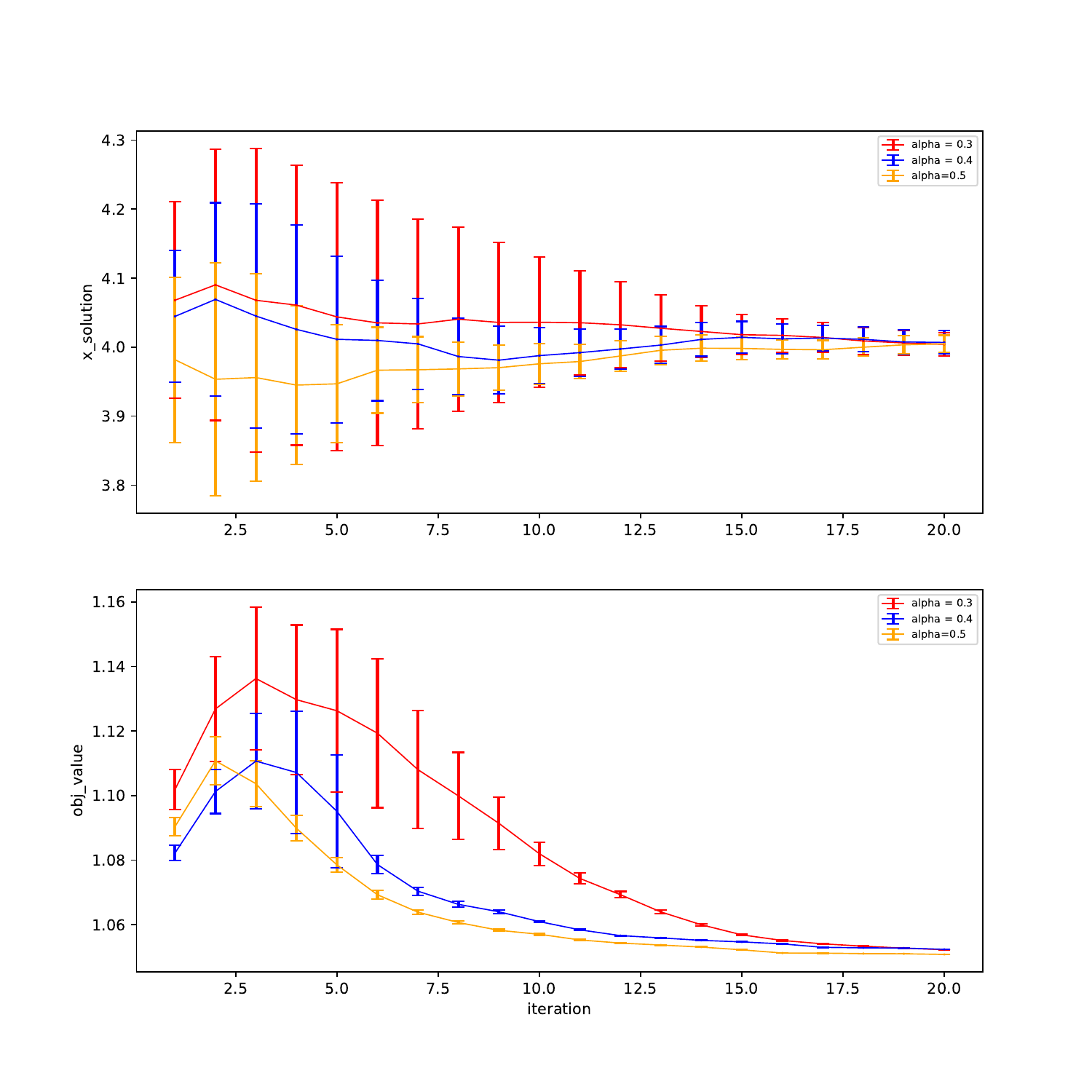}}
	\subfigure[  proximal parameter $\rho = 10, 50, 100$]{\includegraphics[width=0.34\textwidth]{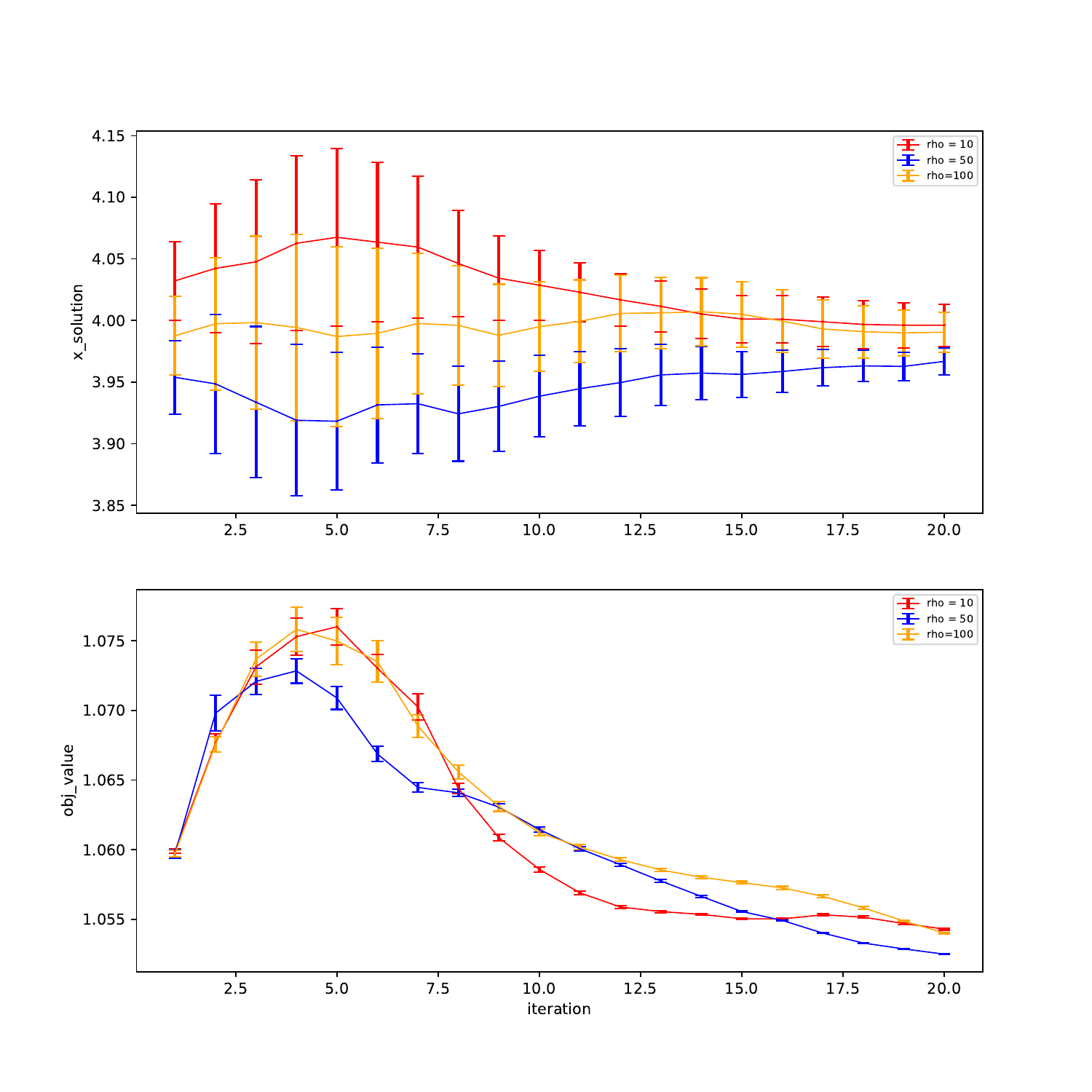}}
	\caption{\rm {In Figure \ref{fig:sol_sample} (b) and (c), the  plots at the top represent  the  solution sequence, and the plots at the bottom  represent the sequence of objective value approximation using the test data set of size $10^5$ excluding first iteration.}}
	\label{fig:sol_sample}
\end{figure}

\gap

We now compare the performance of the SMM algorithm with the NASA algorithm \cite[Algorithm~1]{GhadimiRuszczynskiWang20}.  At each iteration, the NASA algorithm generates only one
new sample  and acquires the corresponding stochastic gradient information to move to the next iterate.  With appropriate control of the parameter
sequences ($\{\tau_k\}$, $\{\beta_k\}$ and parameters $a$ and $b$ in its description in \cite{GhadimiRuszczynskiWang20})
the NASA algorithm  was shown to be asymptotically convergent
in \cite[Theorem 3]{GhadimiRuszczynskiWang20}.  With suitable and fair tuning of these parameters, we implement this algorithm for \eqref{eq:oce_deviation_exp}.
We can first observe from Figure~\ref{fig:nasa}, especially the plots at the bottom regarding the objective values, that the performance of the solution sequence
is highly related to the initial points.  It takes roughly 500 iterations starting at the initial point 0.8 to get the solution with the objective value around 1.12,
and roughly 300 iterations starting at the initial point 3.5 to obtain the solution with the similar objective value. However, when the initial point is 4.2 which
is expected to be very close to the global optimum, the number of iterations needed is less than 20.  Figures~\ref{fig:nasa_rep}(a) and (b) further show that
the solution sequences and objective value sequences from the NASA algorithm over 50 replications with random initialization respectively in $[0,8]$ and $[3,5]$
converge slower and with more volatility than the SMM algorithm.
% \vspace{-1cm}

\begin{figure}[!ht]
\subfigure[ \, initial point 0.8]{\includegraphics[width=0.33\textwidth]{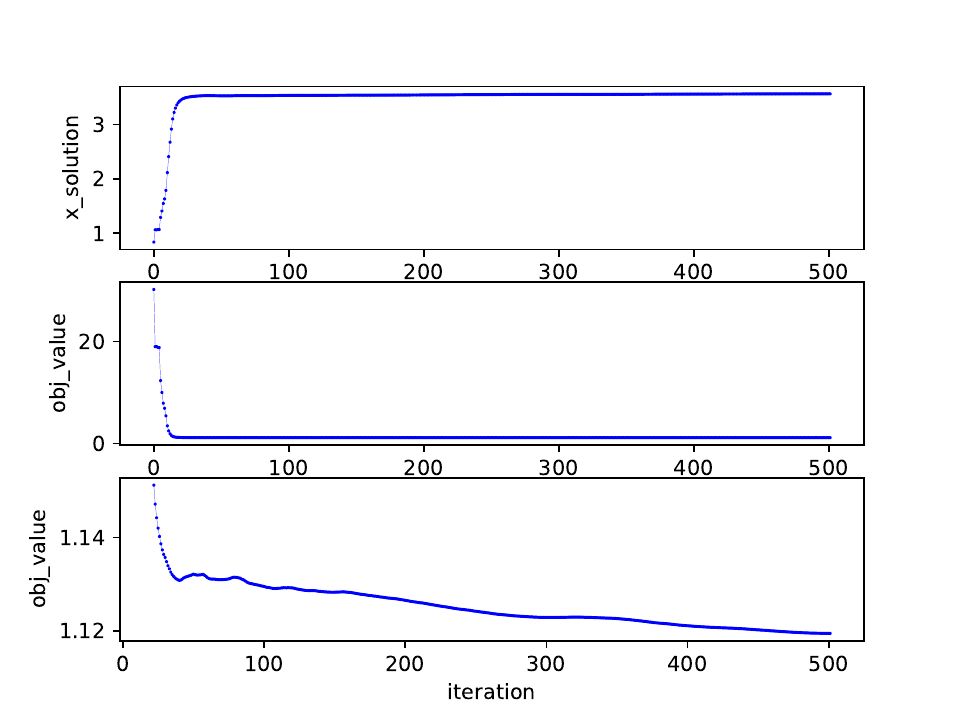}}
\subfigure[ \, initial point 3.5]{\includegraphics[width=0.33\textwidth]{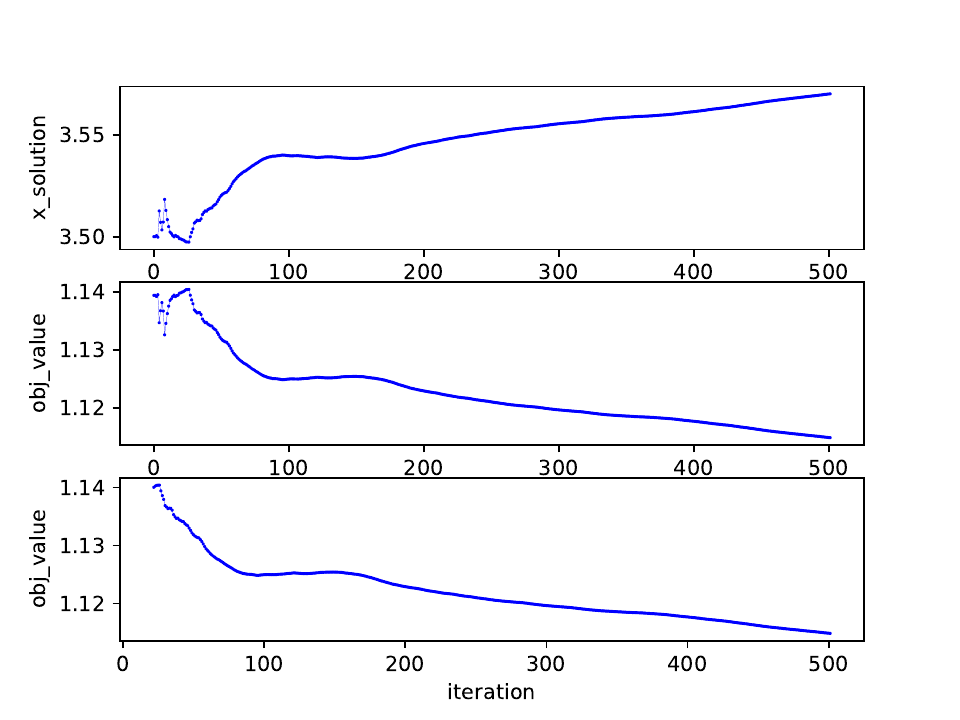}}
\subfigure[ \,  initial point 4.2]{
\includegraphics[width=0.33\textwidth]{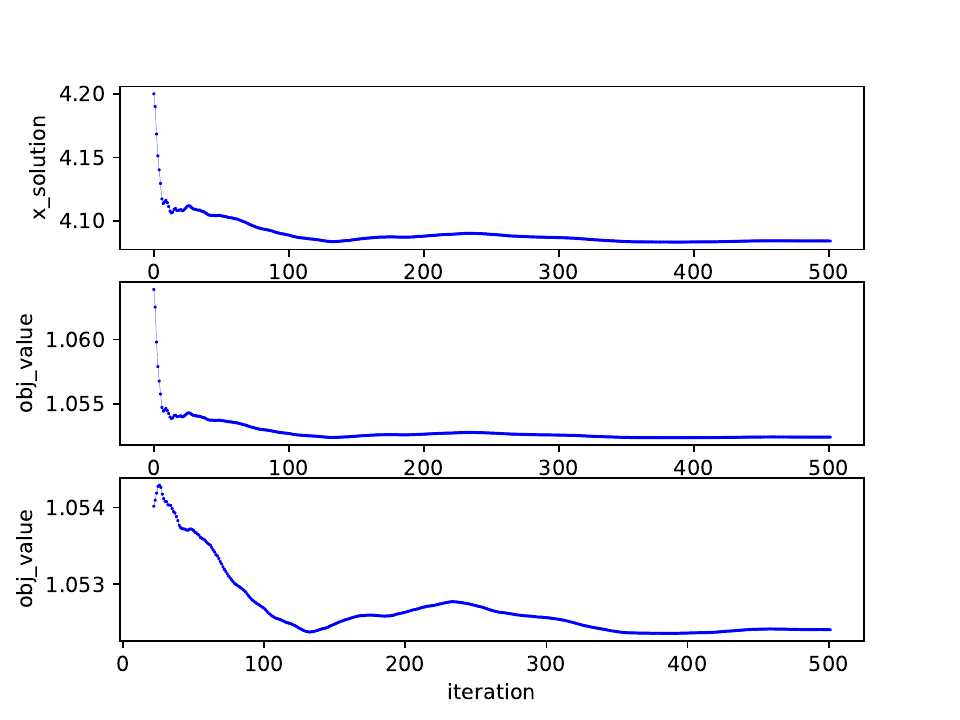}}
\caption{\rm {The top curve represents the  solution sequence, and the middle curve represents the sequence of objective value approximation using
the test data set of samples, and  the bottom curve represents the sequence of objective value approximation excluding the first 20 iterations.}}
\label{fig:nasa}
\end{figure}
% \vspace{-1cm}

\begin{figure}[!ht]
\begin{center}
\subfigure[]{
\includegraphics[width=2.5in]{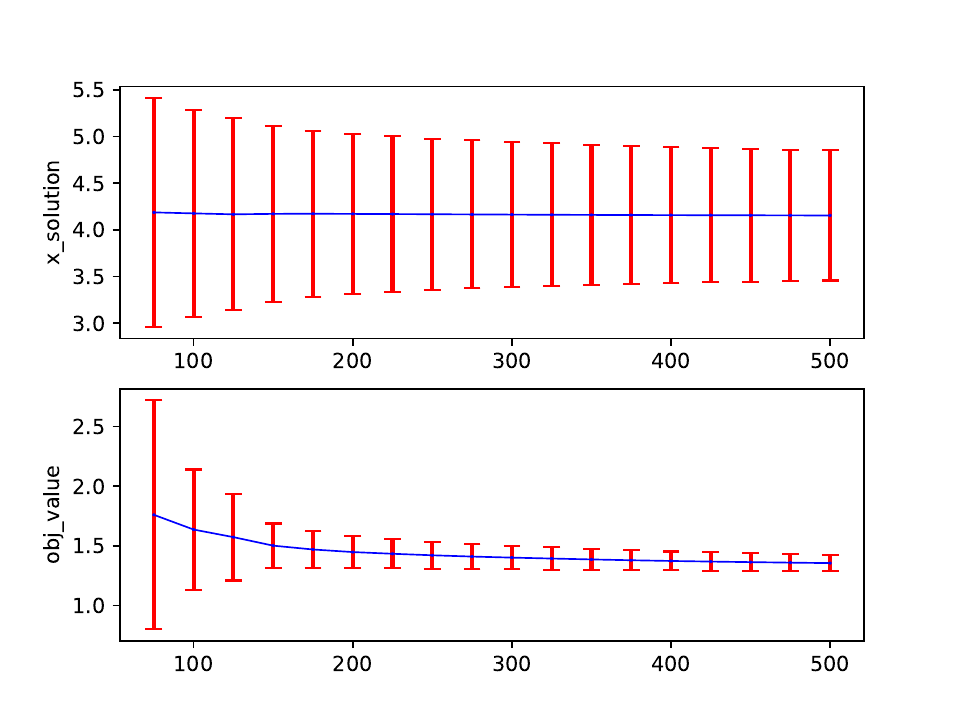}}
\subfigure[]{
\includegraphics[width=2.5in]{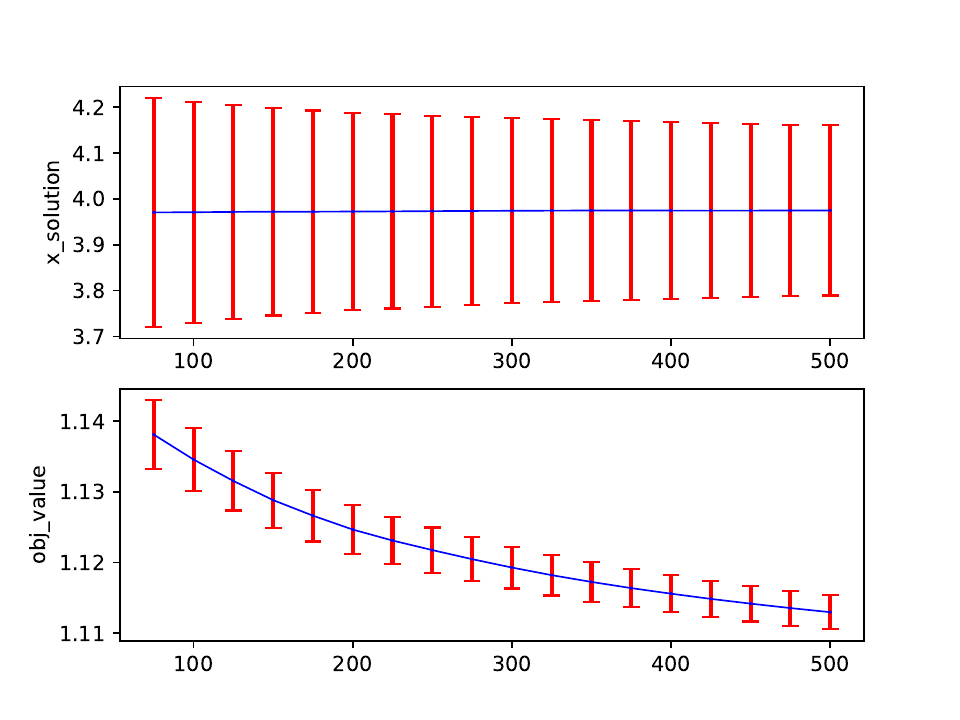}}
\end{center}
\caption{\rm {In Figure \ref{fig:nasa_rep}(a),  the initial points are randomly selected in $[0,8]$ and in Figure \ref{fig:nasa_rep}(b) the initial points are randomly selected in $[3,5]$  over 50 replications. The top figures represent  the  solution sequence, and the bottom figures represent the sequence of objective value approximation using the test data set of samples, both of which exclude the first 50 iterations.}}
\label{fig:nasa_rep}
\end{figure}

In the NASA algorithm, the number of samples is the same as the number of iterations.  To have a fair comparison of the SMM and NASA algorithms,
we compute the total utilized sample size of the SMM algorithm with $\alpha = 0.4$
% after several specific iteration numbers of the SMM algorithm in the following table according to the incremental sample size rate $\Delta_k = \lfloor k^{0.4}\rfloor +1 $.
and summarize the computational results in the following table, where $\{v_r\}_{r=1}^{50}$ denotes the SAA of the objective value at the obtained solutions for 50 replications.
The rows of the SMM algorithm imply that as the iteration number increases, both the mean and the standard deviation (std) of the objective values decrease with a relatively small number of samples, which are generally much smaller than the ones of the NASA algorithm with the same initialization range. % Moreover, in 5 iterations of the SMM which requires a total sample size of 16, the average objective value is smaller than that of NASA with 5000 samples with the initialization range $[0,8]$, while the standard deviations of the objective values for both algorithms are in the 10$^{-1}$ scale.
By setting the initialization range of the NASA algorithm as $[3,5]$, both the mean and the std of the sequences  are comparable to the ones from the SMM, but a significantly larger number of samples and the prior knowledge of a good initialization range are required by the former algorithm.  % achieves a best average objective value that is higher.
Hence, although the convergence of the SMM requires the sequence of sample sizes $\{N_{\nu}\}$ to satisfy the condition stipulated in Lemma~\ref{lm:condition on sample sizes},
the total number of utilized samples turns out to be quite small
% the SMM algorithm may need a small number of sample
in practice to achieve a reasonable accuracy.  Hence,
the SMM algorithm  significantly outperforms the NASA algorithm in terms of the sample complexity, convergence rate, and stability,
while sacrificing a longer computational time based on the column of running time.  Part of the reason for the latter may be because we employed an off-the-shelf solver ({\sc Ipopt}) to solve
each convex subproblem.  As the experiments are meant to provide initial support to the key ideas of the SMM algorithm, we have not attempted to improve the
computational efficiency of this step.
%
% So a plausible method to improve the computational efficiency while maintaining the sample efficiency of the SMM is an inexact version of the
% SMM algorithm, wherein each iterate does not need to be an exact solution of the convex subproblem.

\gap

\begin{center}
\begin{tabularx}
{1\textwidth}{
  | >{\centering\arraybackslash}X
  | >{\centering\arraybackslash}X
  | >{\centering\arraybackslash}X  |>{\centering\arraybackslash}X
  | >{\centering\arraybackslash}X
  | >{\centering\arraybackslash}X
  | >{\centering\arraybackslash}X
  | >{\centering\arraybackslash}X
  | >{\centering\arraybackslash}X
  |}
\hline
 algorithm & range of initial points & iteration number & total sample size   & mean of $\{v_r\}_{r=1}^{50}$ & std of $\{v_r\}_{r=1}^{50}$ & average running time   \\
\hline
\hline
\multirow{4}{2.5em}{SMM} & \multirow{4}{2.5em}{$[\,0, \, 8\,]$}  & 5 & 16 &  1.0847 & 0.0932 &  0.3613\\
& & 10 & 31 &  1.0635 & $0.0210$  &  0.8009\\
&  & 15 & 54 &  1.0545 & $0.0083$  &  1.5682\\
 &   & 20 & 66 &  1.0511 & $0.0063$ & 2.1323  \\
\hline
\hline
\multirow{4}{2.5em}{NASA} & \multirow{3}{2.5em}{$[\,0, \, 8\,]$}  & 50 & 50 & 1.6355 &  0.8684   & 0.0012 \\
 &   & 200 & 200 & 1.3818 &  0.3248   & 0.0096 \\
 &   & 500 & 500 & 1.3131 &  0.2439  & 0.0120\\
 &   & 5000 & 5000 & 1.2771 &  $0.1140$  & 0.1410\\
 \hline
 \multirow{4}{2.5em}{NASA}& \multirow{4}{2.5em}{$[\,3, \, 5\,]$}   & 50 & 50 & 1.1522 & 0.0960  & 0.0013\\
   &   & 200 & 200 & 1.0997 & 0.0565  & 0.0057\\
     &   & 500 & 500 & 1.0960 & 0.0556  & 0.0157\\
  &   & 5000 & 5000 & 1.0890 & 0.0412  & 0.1142\\
\hline
\end{tabularx}
\end{center}
\gap

\subsection{OCE-of-deviation optimization: piecewise linear utility}
    \label{subsec:piecewise}
With the piecewise linear utility function $u_{\rm pl}(z) = 0.8 \, \max\{z, 0\} - 2 \, \max\{-z, 0\}$,  the loss function $f(x, \xi) = (x- \xi)^2$ and
the truncated normal random variable $\tilde{\xi}$ on $[1.0, 1.5]$, the OCE-of-deviation optimization problem is equivalent to:
{\small
    \[
    \underset{ x \in [0,8], t \in [-100, 100]}{\mbox{minimize}} \, \,  - t + \mathbb{E} \left[ \max \left\{ 0.8 \left(  - (x- \tilde{\xi}\,)^2 + \mathbb{E} \, [ (x- \tilde{\xi}\, )^2] + t \right) , 2 \left( - (x- \tilde{\xi}\, )^2 + \mathbb{E} \, [ (x- \tilde{\xi}\, )^2]+t \right)  \right\}  \right].
    \]
}

\noindent  The performance of the SMM algorithm with 30 replications is given in Figure~\ref{fig:piece}. Since the corresponding compound SP is a nonconvex problem, the SMM algorithm might converge to different points with different initial points. Thus the variance of the iterates of $x$ is consistently large during the 30 iterations. Given a fixed initial point, we observe that the difference of consecutive iterates is of the order $10^{-2}$ after 20 iterations.

    \begin{figure}[!ht]
    {\begin{center}
	\includegraphics[width=0.5\textwidth]{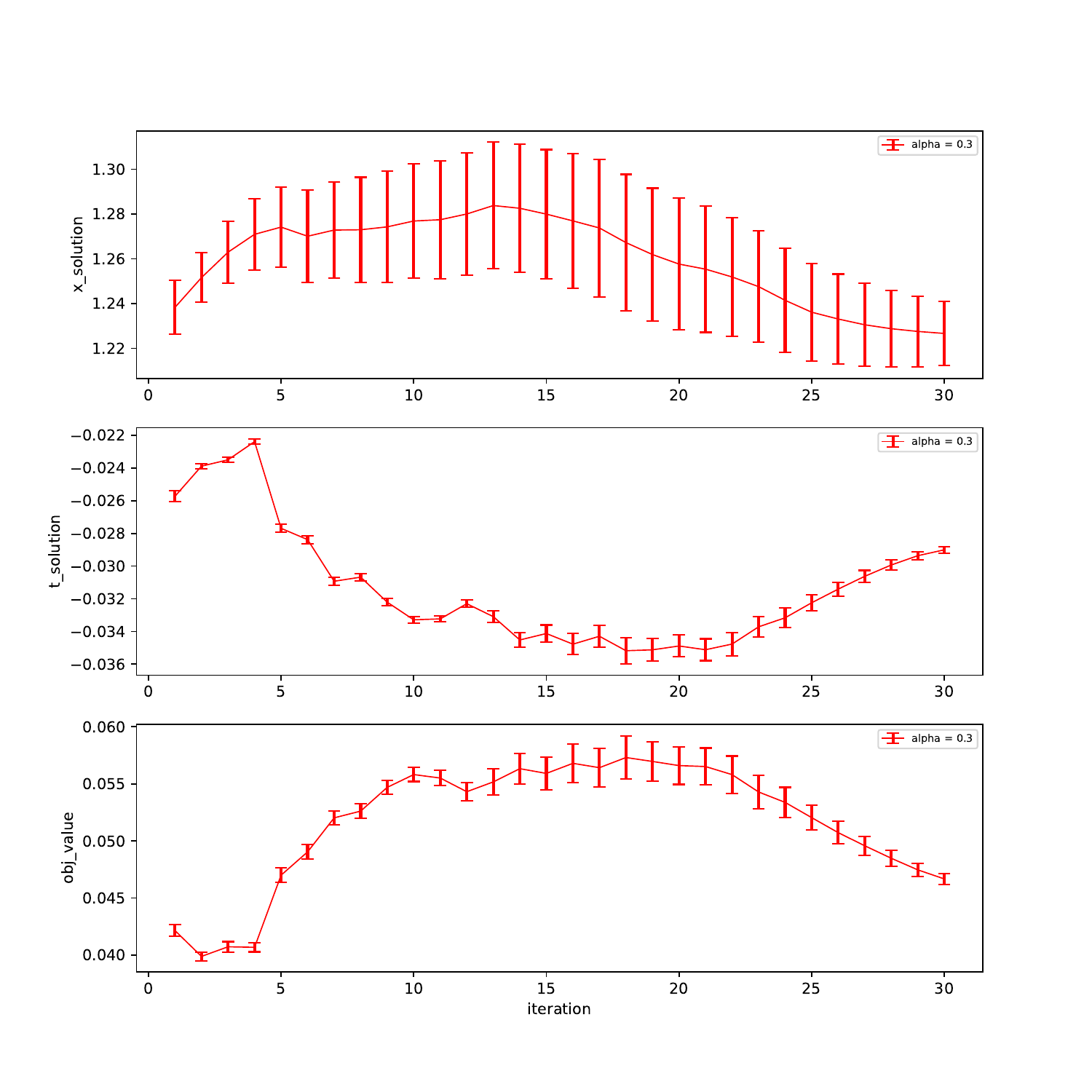}\end{center}}
	\caption{\rm {The top and middle curves represent the solution sequence of $x$ and $t$ respectively, and the bottom curve represents the sequence of objective value approximation using the test data set of samples, excluding the initial iteration.}}
	\label{fig:piece}
\end{figure}

% Acknowledgments here
\section*{Acknowledgments.}
Their work is supported by the U.S.\ National Science Foundation under grant
IIS-1632971 and U.S.\ Air Force Office of Sponsored Research under grant
FA9550-18-1-0382.  The authors are grateful to two reviewers for their
constructive comments that have improved the paper.

	\bibliographystyle{plain}
	\bibliography{references}

 \section{APPENDIX}
 	 
\gap

\textit{Proof of Proposition \ref{prop:SAA_error_bound}.}  (a) By the Lipschitz continuity of $\varphi$ and $\psi$, and Theorem 7.48 in \cite{shapiro2009lectures}, we can derive that $\overline \Theta_{n,m}(x)$ converges to $\Theta(x)$ w.p.1 uniformly on $X$.

\gap

(b) With the uniform law of large numbers of $\overline \Theta_{n,m}(\bullet)$ on $X$ and compactness of $X$, we can derive that with probability 1, there exist a positive integer $\bar{n}$ and a positive value $U$, such that for any $n,m \geq \bar{n}$, we have  $\displaystyle\sup_{x\in X} \; \left| { \overline \Theta_{n,m}(x)}\right| < U$.

\gap

(c)
Consider any $x\in X$.
By the Lipschitz continuity of $\psi$, we have
\[
\begin{array}{ll}
	& \left( \, \mathbb{E}  \,  \left|{ \, \overline{\Theta}_{n,m}(x) - \Theta(x)  \, } \right|\,   \right)^{\,2} \leq  \, \mathbb{E}
	\, \left|{ \, \overline{\Theta}_{n,m}(x) - \Theta(x)  \, } \right| ^{\,2}   \\[0.2in]
	& \leq (\mbox{Lip}_{\psi})^{\, 2} \cdot  \displaystyle{\sum_{j=1}^{\ell_\varphi}}  \, \,  \mathbb{E} \,   \Bigg( \,  \underbrace{\displaystyle{
			\frac{1}{n}
		} \, \displaystyle{
			\sum_{t=1}^{n}
		} \, \varphi_j\Bigg( G(x,\xi^{\, t} \, ), \displaystyle{
			\frac{1}{m}
		} \, \displaystyle{
			\sum_{s=1}^{m}
		} \, F(x, \eta^{\, s} ) \Bigg)}_{\mbox{denoted by  $ \bar{\Psi}^{\,\, j}_{n,m}(x) $ }}- \underbrace{\mathbb{E} \, \Bigg[ \varphi_j  \Big(G(x, \tilde{\xi}), \mathbb{E} \,  [ F(x, \tilde{\xi})\,  ]  \Big)  \Bigg]}_{\mbox{denoted by  $  {\Psi}^{\,j}(x) $ }}  \, \Bigg) ^{\,2}.
\end{array}
\]
Since $\mathbb{E} \,\Big[ \, (\bar{\Psi}^{\,j}_{n,m}(x) - \Psi^{\,j}(x))^2\, \Big] =   \left(\Psi^{\,j}(x) - \mathbb{E} \, [ \, \bar{\Psi}^{\,j}_{n,m}(x) \,  ]  \right)^2 +
\mathbb{E} \,\Big[ \, \left(\bar{\Psi}^{\,j}_{n,m}(x) -\mathbb{E} \, [ \, \bar{\Psi}^{\,j}_{n,m}(x)\, ] \right)^2\, \Big]   $, we analyze these two terms separately.
First, for any $ j\in [\ell_{\varphi}]$,
\[
\begin{array}{ll}
	\Big( \, \Psi^{\,j}(x) - \mathbb{E} \, [ \, \bar{\Psi}^{\,j}_{n,m}(x) \,  ]  \Big)^2   & =    \Bigg( \, \Psi^{\,j}(x) -
	\mathbb{E}_{\tilde{\xi}, \{\eta^s\}} \Bigg[ \,   \varphi_j \Bigg( G(x,\tilde{\xi} \, ), \displaystyle{
		\frac{1}{m}
	} \, \displaystyle{
		\sum_{s=1}^{m}
	} \, F(x, \eta^{\, s} ) \Bigg) \Bigg]  \Bigg)^2   \\[0.25in]
	& \leq {(\mbox{Lip}_{\varphi})^{\, 2}  \cdot   \,\mathbb{E}_{\{\eta^s\}} \Bigg[ \, \Big\| \,  \mathbb{E} \, [ \, F(x, \tilde{\xi} \, ) \,] -  \displaystyle{
			\frac{1}{m}
		} \, \displaystyle{
			\sum_{s=1}^{m}
		} \, F(x, \eta^{\, s} ) \,\Big\|^2 \,  \Bigg] } \, \leq \, \displaystyle{\frac{(\mbox{Lip}_{\varphi})^{\,2} \, \sigma_F^{\,2}}{ {m}}}.
\end{array}
\]
For the second term  $\Psi^j$, we have
\[
\begin{array}{ll}
	& \mathbb{E} \,\Big[ \, \left(\bar{\Psi}^{\,j}_{n,m}(x) -\mathbb{E} \, [ \, \bar{\Psi}^{\,j}_{n,m}(x)\, ] \right)^2\, \Big]  \\[0.15in]
	& =  \displaystyle{\frac{1}{n}} \, \mathbb{E}_{\tilde{\xi}^{\,1}, \{\eta^{\,1,s}\}} \, \Bigg[ \Bigg(\,\varphi_j \Big(G(x, \tilde{\xi}^{\,1}), \displaystyle{\frac{1}{m}} \sum_{s=1}^m\, F(x, \eta^{\,1,s}) \Big) -\mathbb{E}_{\tilde{\xi}^{\, 2}, \{\eta^{\, 2,s}\}} \, \Big[ \,\varphi_j \Big(G(x, \tilde{\xi}^{\,2}), \displaystyle{\frac{1}{m}} \sum_{s=1}^m\, F(x, \eta^{\,2, s }) \Big)  \Big] \Bigg)^2 \,  \, \Bigg]\\[0.2in]
	& = \displaystyle{\frac{1}{2\,n}} \, \mathbb{E}_{\big\{\tilde{\xi}^{\,i}, \{\eta^{\,i,s}\}: \,  i=1,2\big\}} \, \Bigg[ \Bigg(\,\varphi_j \Big(G(x, \tilde{\xi}^{\,1}), \displaystyle{\frac{1}{m}} \sum_{s=1}^m\, F(x, \eta^{\,1,s}) \Big) -  \,\varphi_j \Big(G(x, \tilde{\xi}^{\,2}), \displaystyle{\frac{1}{m}} \sum_{s=1}^m\, F(x, \eta^{\,2, s }) \Big)   \Bigg)^2 \,  \, \Bigg]\\[0.2in]
	& \leq \displaystyle{\frac{(\mbox{Lip}_\varphi)^{\,2}}{2\,n}}  \, \mathbb{E}_{\big\{\tilde{\xi}^{\,i}, \{\eta^{\,i,s}\}: \, i=1,2\big\}}  \Bigg[ \,\Big\| \,  G(x, \tilde{\xi}^{\,1})- G(x, \tilde{\xi}^{\,2}) \, \Big\| ^{\,2}+ \Big\| \,  \displaystyle{\frac{1}{m}} \sum_{s=1}^m\, F(x, \eta^{\,1, s })  - \displaystyle{\frac{1}{m}} \sum_{s=1}^m\, F(x, \eta^{\,2, s })  \, \Big\| ^{\,2} \,\Bigg]\\[0.2in]
	& \leq  \displaystyle{\frac{( \sigma_G^{\,2} +  \sigma_F^{\,2}/m) \, (\mbox{Lip}_\varphi)^{\,2}  }{n}}.
\end{array}
\]
Combining the above two bounds, we derive that
\[
\mathbb{E}  \,  \left|{ \, \overline{\Theta}_{n,m}(x) - \Theta(x)  \, }\right|  \leq  \mbox{Lip}_{\psi} \,  \mbox{Lip}_{\varphi} \cdot  \sqrt{\ell_{\varphi}} \, \Bigg(  \frac{\sigma_F^{\,2}   }{m} + \frac{  \sigma_F^{\,2} \,  }{m\, n}  +  \displaystyle{\frac{ \sigma_G^{\,2}  \,  }{n}} \Bigg)^{1/2}  \leq \mbox{Lip}_{\psi} \,  \mbox{Lip}_\varphi  \,  \sqrt{\ell_{\varphi}} \,   \frac{\sqrt{2  \sigma_F^{\,2} +  \sigma_G^{\,2} } \,  }{ {\min\{m, n\}}^{1/2}},
\]
which completes the proof of this lemma.
 
\hfill $\square$

\gap

The following lemma provides a sufficient condition on the sequence of sample sizes for the convergence of the SMM Algorithm.
\gap

\gap 
 
\textit{Proof of Lemma \ref{lm:condition on sample sizes}.}
Since  $N_{\nu} \leq \nu \cdot  N_{\nu-1}$ for all $\nu > \bar{\nu}$, we have
$N_{\nu} - N_{\nu-1} \leq  c_3 \cdot \displaystyle{\frac{N_{\nu}}{\nu}} \leq c_3 \cdot N_{\nu-1} $ for all such $\nu$.  Then we have,
\[
\displaystyle{\Big(1 - \Big( \frac{ N_{\nu-1}}{N_{\nu}} \Big)^{\,2} \,\Big)
} \, \displaystyle{
	\frac{1}{N_{\nu-1}^{\, 1/2}}} \, = \, \frac{N_{\nu} - N_{\nu-1}}{N_{\nu}} \cdot  \frac{N_{\nu} + N_{\nu-1}}{N_{\nu}} \cdot \frac{1}{N_{\nu-1}^{\,1/2}}
\, \leq \, 2 \, c_3^{\,1/2} \,
\frac{N_{\nu}- N_{\nu-1}}{N_{\nu}} \cdot \frac{1}{(N_{\nu} - N_{\nu-1})^{\,1/2}}.
\]
Hence, we only need to show $\displaystyle{
	\sum_{\nu=1}^{\infty}
} \, \displaystyle{
	\frac{N_{\nu} - N_{\nu-1}}{N_{\nu}}
} \, \displaystyle{
	\frac{1}{( \, N_{\nu} - N_{\nu-1} \, )^{\, 1/2}}
} \, < \, \infty.$ In fact,
\[
\begin{array}{lll}
	\displaystyle{
		\frac{N_{\nu} - N_{\nu-1}}{N_{\nu}}
	} \, \displaystyle{
		\frac{1}{(   N_{\nu} - N_{\nu-1}   )^{ \, 1/2}}
	} =  \left( \displaystyle{
		\frac{N_{\nu}-N_{\nu-1}}{N_{\nu}}
	} \right)^{1/2}   \displaystyle{
		\frac{1}{N_{\nu} ^{ \,  1/2}}}
	\leq  \left( \displaystyle{
		\frac{c_3 }{\nu}
	} \right)^{1/2} \displaystyle{
		\frac{1}{ (c_2 \cdot \nu^{1+2c_1  })^{1/2} }}  =   \displaystyle{
		\frac{c_3}{(c_2\cdot c_3)^{  1/2}}
	} \, \displaystyle{
		\frac{1}{\nu^{1+ c_1}}
	},
\end{array}
\]
which establishes the finiteness of the second series.
 \hfill $\square$

\gap

 \gap

\textit{Proof of Lemma \ref{lm:stationarity_distance}.}
For any $\wh{x} \in \mbox{FIX}({\cal M}_{\wt{V}_{\rho}})$, from (SR2), we have
$\wh{x} \in S_{X, \theta}^{\, C}$. Thus for such $\wh{x}$, choosing $\bar{x}=\wh{x}$,
the lemma holds.  It remains to consider the case when
$\wh{x} \in X \setminus \mbox{FIX}({\cal M}_{\wh{V}_{\rho}})$.  Let
\[
\zeta_{\rho}(x;\wh{x}) \, \triangleq \, \theta(x) + \left( \, \displaystyle{
	\frac{1}{2\rho}
} + \displaystyle{
	\frac{\kappa}{2}
} \, \right) \norm{x-\wh{x}}^2.
\]
By the assumption (SR1),
particularly the key inequality (\ref{eq:the other inequality}), we have
\[
\wt{V}_{\rho}(x;\wh{x}) + \displaystyle{
	\frac{\kappa}{2}
} \, \norm{x-\wh{x}}^2 \, \geq \, \zeta_{\rho}(x;\wh{x} ) \, \geq \,
\wt{V}_{\rho}(x;\wh{x} ).
\]
Consequently, $\zeta_* \triangleq \underset{x\in X}{\mbox{minimize}} \,
\zeta_{\rho}(x;\wh{x})
\geq \wt{V}_{\rho}({\cal M}_{\wt{V}_{\rho}}(\wh{x});\wh{x})$.  Hence
\[
\zeta_{\rho}\left({\cal M}_{\wt{V}_{\rho}}(\wh{x});\wh{x}\right) - \zeta_* \, \leq \,
\zeta_{\rho}\left({\cal M}_{\wt{V}_{\rho}}(\wh{x});\wh{x}\right) -
\wh{V}_{\rho}\left({\cal M}_{\wt{V}_{\rho}}(\wh{x});\wh{x}\right) \, \leq \, \displaystyle{
	\frac{\kappa}{2}
} \, \norm{{\cal M}_{\wt{V}_{\rho}}(\wh{x}) - \wh{x}}^2.
\]
By Ekeland's variational principle,  with $\varepsilon = \displaystyle{
	\frac{\kappa}{2}
} \, \norm{ \, {\cal M}_{\wt{V}_{\rho} }(\wh{x}) - \wh{x} \, }^2 > 0$ and
$\lambda =\displaystyle \frac{\kappa}{2}\norm{ \, {\cal M}_{\wt{V}_{\rho}}(\wh{x}) - \wh{x}\, } > 0$,
we deduce the existence of $\bar{x} \in X$ such that
\begin{equation} \label{eq:eke_val}
	\zeta_{\rho}(\bar{x};\wh{x} ) \, \leq \,
	\zeta_{\rho}({\cal M}_{\wt{V}_{\rho}}(\wh{x}); \wh{x} ),
	\qquad
	\norm{ \, \bar{x} -  {\cal M}_{\wt{V}_{\rho}}(\wh{x}) \, } \, \leq \,
	\norm{ \, \wh{x} - {\cal M}_{\wt{V}_{\rho}}(\wh{x})\, }
\end{equation}	
and $\bar{x} \, = \, \displaystyle{
	\operatornamewithlimits{\mbox{argmin}}_{x \in X}
} \, \left\{\, \zeta_{\rho}(x; \wh{x}) + \lambda \, \norm{x-\bar{x}} \, \right\}$.
By the optimality condition of the latter problem, we have
$0 \in \partial_C \, \theta(\bar{x}) + \left( \displaystyle{
	\frac{1}{\rho}
} \, + \kappa \right) (\bar{x}-\hat{x}) + \lambda \, \mathbb{B} +
\mathcal{N}(\bar{x};X)$. This is the place where
one can choose other suitable subdifferential with which the first-order optimality condition holds.  From \eqref{eq:eke_val}, we have
\[
\norm{\bar{x} - \wh{x}} \, \leq \, \norm{\bar{x} -  {\cal M}_{\wt{V}_{\rho}}(\wh{x})}
+ \norm{\wh{x} - {\cal M}_{\wt{V}_{\rho}}(\wh{x})} \, \leq \,
2 \, \norm{\wh{x} - {\cal M}_{\wt{V}_{\rho}}(\wh{x})}.
\]
Moreover, we deduce
\[
\mbox{dist}\left(0; \partial_C \, \theta(\bar{x}) + \mathcal{N}(\bar{x}; X) \right)
\, \leq \, \left( \displaystyle{
	\frac{1}{\rho}
} + \kappa \, \right) \, \norm{\bar{x}-\wh{x}} + \lambda \, \leq \, \left( \displaystyle{
	\frac{2}{\rho}
} + \displaystyle{
	\frac{5 \kappa}{2}
} \, \right) \norm{\wh{x} - {\cal M}_{\wt{V}_{\rho}}(\wh{x})},
\]
obtaining the two desired bounds. 
 \hfill $\square$

\end{document}